\documentclass[10pt]{amsart}
\usepackage{amssymb,amsfonts, amsmath}
\usepackage[all,arc]{xy}
\usepackage{enumitem}
\usepackage{mathrsfs}
\usepackage{ stmaryrd }
\usepackage{url}

\usepackage
[pdfauthor={Preston Wake and Carl Wang-Erickson},
 pdftitle={The rank of Mazur's Eisenstein ideal},
 bookmarks=false]
{hyperref}

\newcounter{qcounter}

\newcommand\define{\newcommand}

\define\isoto{\xrightarrow{\sim}}
\define\onto{\twoheadrightarrow}
\define\coker{\mathrm{coker}}
\DeclareMathOperator{\Spec}{Spec}

\define\cC{\mathcal{C}}

\newcommand{\dia}[1]{{\langle #1 \rangle}}

\newcommand{\ttmat}[4]{\left( \begin{array}{cc}
#1 & #2 \\
#3 & #4
\end{array}
\right)}

\newcommand{\Z}{\mathbb{Z}}
\newcommand{\Q}{\mathbb{Q}}

\newcommand{\F}{\mathbb{F}}

\newcommand{\zp}[1]{\Z/{p^{#1}\Z}}

\newcommand{\sO}{\mathcal{O}}

\newcommand{\m}{\mathfrak{m}}

\newcommand{\cL}{\mathcal{L}}

\newcommand{\Hom}{\mathrm{Hom}}
\newcommand{\Gal}{\mathrm{Gal}}

\newcommand{\Aut}{\mathrm{Aut}}

\newcommand{\Ext}{\mathrm{Ext}}

\newcommand{\End}{\mathrm{End}}

\newcommand{\Fr}{\mathrm{Fr}}

\newcommand{\lb}{{[\![}}
\newcommand{\rb}{{]\!]}}

\newcommand{\red}{\mathrm{red}}

\define\cT{\mathcal{T}}

\define\bFp{\overline{\F}_p}

\define\ord{{\mathrm{ord}}}
\define\Cl{{\mathrm{Cl}}}
\define\GL{{\mathrm{GL}}}

\define\kcyc{\kappa_{\mathrm{cyc}}}

\define{\Fitt}{\mathrm{Fitt}}
\define{\Ann}{\mathrm{Ann}}

\newtheorem{thm}{Theorem}[subsection] 
\newtheorem*{thm*}{Theorem}

\newtheorem*{claim}{Claim}
\newtheorem{cor}[thm]{Corollary}
\newtheorem{prop}[thm]{Proposition}
\newtheorem{lem}[thm]{Lemma}
\newtheorem{conj}[thm]{Conjecture}

\theoremstyle{definition}
\newtheorem{defn}[thm]{Definition}

\newtheorem{eg}[thm]{Example}

\theoremstyle{remark}
\newtheorem{rem}[thm]{Remark}
\newtheorem{rems}[thm]{Remarks}

\newcommand{\ra}{\rightarrow}

\newcommand{\lra}{\longrightarrow}
\newcommand{\lrisom}{\buildrel\sim\over\lra}
\newcommand{\risom}{\buildrel\sim\over\ra}
\newcommand{\rinj}{\hookrightarrow}

\newcommand{\rsurj}{\twoheadrightarrow}

\newcommand{\bQ}{\mathbb{Q}}

\newcommand{\bR}{\mathbb{R}}
\newcommand{\bT}{\mathbb{T}}
\newcommand{\cE}{\mathcal{E}}

\newcommand{\Db}{{\bar D}}

\newcommand{\cG}{{\mathcal{G}}}

\newcommand{\bQp}{{\overline{\Q}_p}}

\newcommand{\Jm}{{J^{\min{}}}}
\newcommand{\Bm}{{B^{\min{}}}}
\newcommand{\Cm}{{C^{\min{}}}}

\newcommand{\fl}{{\mathrm{flat}}}
\newcommand{\nfl}{{\text{non-flat}}}
\newcommand{\loc}{{\mathrm{loc}}}

\usepackage{color}

\newcommand{\tr}{{\mathrm{tr}}}

\newcommand{\sm}[4]{\ensuremath{\big(\begin{smallmatrix}#1 & #2 \\ #3 & #4\end{smallmatrix}\big)}}

\newcommand{\Zr}{\Z/p^r\Z}

\newcommand{\oQ}{{\overline{\mathbb{Q}}}}
\newcommand{\et}{\text{\'et}}

\makeatletter
\let\c@equation\c@thm
\makeatother
\numberwithin{equation}{subsection}

\title{The rank of Mazur's Eisenstein ideal}

\author{Preston Wake}
\address{Institute for Advanced Study\\
1 Einstein Drive \\
Princeton, NJ 08540}
\email{pwake@ias.edu}

\author{Carl Wang-Erickson}
\address{Department of Mathematics, Imperial College London \\
	London SW7 2AZ, UK}
\email{c.wang-erickson@imperial.ac.uk}

\begin{document}

\begin{abstract}
We use pseudodeformation theory to study Mazur's Eisenstein ideal. Given prime numbers $N$ and $p>3$, we study the Eisenstein part of the $p$-adic Hecke algebra for $\Gamma_0(N)$. We compute the rank of this Hecke algebra (and, more generally, its Newton polygon) in terms of Massey products in Galois cohomology, answering a question of Mazur and generalizing a result of Calegari--Emerton. We also also give new proofs of Merel's result on this rank and of Mazur's results on the structure of the Hecke algebra.
\end{abstract}

\dedicatory{To Barry Mazur, on his 80th birthday}

\maketitle

\tableofcontents

\section{Introduction}

Let $N$ and $p>3$ be prime numbers.  Let $\bT$ denote the completion of the Hecke algebra with weight $2$ and level $\Gamma_0(N)$ at the Eisenstein maximal ideal with residual characteristic $p$, and let $\bT^0$ denote the cuspidal quotient of $\bT$. In his influential paper \cite{mazur1978}, Mazur studied $\bT^0$ and showed that $\bT^0 \ne 0$ if and only if $p \mid (N-1)$. In that same paper, he posed the question ``Is there anything general that can be said about the Newton polygon of $\bT^0$, or even about $\mathrm{rank}_{\Z_p}(\bT^0)$?''\ [pg.\ 140, \emph{loc.~cit.}]. In this paper, we give a complete answer to his question, showing that this Newton polygon can be computed exactly in terms of arithmetic invariants present in Galois cohomology. 

It remains an interesting avenue of research to systematically compute these new arithmetic invariants. This might be achieved by relating them to more analytic invariants coming from the theory of $L$-functions. We do this in the case of the lowest-order invariants, by relating our invariant to one studied by Merel \cite{merel1996}, thus giving a new proof of Merel's result. The higher-order invariants remain mysterious; however, see Remark \ref{rem:lecouturier} about a recent preprint of Lecouturier that extends Merel's approach. Also, for results along these lines when $p=2$ or $p=3$, see \cite{CE2005}. 

\subsection{Galois cohomology} 
\label{subsec:GC}
In order to state the main theorems, we need to establish some notation for certain Galois cohomology groups; full definitions are found in Appendix \ref{sec:galois cohomology appendix}. Let $(\zp{s})_{/\Z_p}$ and $(\mu_{p^s})_{/\Z_p}$ denote the constant and multiplicative group schemes of order $p^s$ over $\Z_p$, respectively. We let 
\begin{gather*}
H^1_{p,\fl}(\zp{s}) = \Ext^1((\zp{s})_{/\Z_p}, (\zp{s})_{/\Z_p}), \\  H^1_{p,\fl}(\zp{s}(1)) = \Ext^1((\zp{s})_{/\Z_p}, (\mu_{p^s})_{/\Z_p}), \\ \text{and }  
H^1_{p,\fl}(\zp{s}(-1)) = \Ext^1((\mu_{p^s})_{/\Z_p}, (\zp{s})_{/\Z_p}),
\end{gather*}
where the extension groups are taking place in the category of finite flat $p^s$-torsion group schemes over $\Z_p$.

For each $i=0,1,-1$, we have $H^1_{p,\fl}(\zp{s}(i)) \subset H^1(\Q_p,\zp{s}(i))$, an inclusion into local Galois cohomology. Let 
\[
H^1_\fl(\Z[1/Np],\zp{s}(i)) = \ker \left( H^1(\Z[1/Np],\zp{s}(i)) \to  \frac{H^1(\Q_p,\zp{s}(i))}{H^1_{p,\fl}(\zp{s}(i))} \right)
\]
be the resulting global cohomology groups, which are instances of Selmer groups. We will see that, if $p^s\mid (N-1)$, each of the spaces $H^1_\fl(\Z[1/Np],\zp{s}(i))$ is a free $\zp{s}$-module of rank 1 (Corollary \ref{cor:flat cohom of ad}). If $p^t\mid\mid(N-1)$, choose generators
\begin{align*}
a\in H^1_\fl(\Z[1/Np],\zp{t}), \quad b\in H^1_\fl(\Z[1/Np],\zp{t}(1)), \\
 c\in H^1_\fl(\Z[1/Np],\zp{t}(-1)). \qquad \qquad \qquad
\end{align*}
Below we consider these elements as being in $H^1(\Z[1/Np],\zp{t}(i))$.

\subsection{Criterion for rank $1$} We can now state the first main theorem.

\begin{thm}
\label{thm:cup products and rank 1}
Suppose that $p\mid(N-1)$. The following are equivalent:
\begin{enumerate}
\item $\mathrm{rank}_{\Z_p}(\bT^0) \ge 2$
\item The cup product $b \cup c$ vanishes in $H^2(\Z[1/Np],\F_p)$
\item The cup product $a \cup c$ vanishes in $H^2(\Z[1/Np],\F_p(-1))$.
\end{enumerate}
\end{thm}

Theorem \ref{thm:cup products and rank 1} is the special case $n=s=1$ of Theorem \ref{thm:main in reorg} below (see also the remarks following that theorem). This theorem can also be interpreted in terms of class groups. We denote the ideal class group of a number field $K$ by $\Cl(K)$. 

\begin{cor}
\label{cor:class groups and rank 1}
Suppose that $p\mid(N-1)$. Let $\Q(\zeta_N^{(p)},\zeta_p)$ denote the degree $p$ subextension of $\Q(\zeta_N,\zeta_p)/\Q(\zeta_p)$. Consider the following conditions:
\begin{enumerate}
\item The group $\Cl(\Q(N^{1/p}))[p]$ is cyclic. 
\item The group $\Cl(\Q(\zeta_N^{(p)},\zeta_p))[p]_{(-1)}$ is cyclic. Here the subscript ``${(-1)}$'' refers to the $\omega^{-1}$-eigenspace for the action of $\Gal(\Q(\zeta_p)/\Q)$.
\item The rank of the $\Z_p$-algebra $\bT^0$ is one.
\end{enumerate}
Then (1) implies (2), and (2) is equivalent to (3).
\end{cor}
In this paper, we relate (1) and (2) to the cup products $b \cup c$ and $a \cup c$, respectively, in Proposition \ref{prop:cup gives class}. Then it is a corollary of Theorem \ref{thm:cup products and rank 1} that either of (1) and (2) imply (3). The remaining implication that (3) implies (2) is due to Lecouturier \cite{lecouturier2018}.

\begin{rems} 
We note the following relations with other works. 
\begin{itemize}[leftmargin=2em]
\item The implication (1) $\Rightarrow$ (3) of this corollary was first obtained by Calegari--Emerton, and is the main theorem of \cite{CE2005} (for $p > 3$). 
\item The converse implication (3) $\Rightarrow$ (1) is false in general, but, for regular $p$, a partial converse is provided by Schaefer--Stubley \cite[Thm.\ 1.1.2]{SS2018}. In particular, they prove that the converse is true for $p=5$.
\item  The implication (3) $\Rightarrow$ (2) was proven by Lecouturier \cite[Thm.\ 1.7]{lecouturier2018} using Merel's theorem (Theorem \ref{thm:merel} below).
Also using that theorem, he proves a partial converse of (1) $\Rightarrow$ (3): he proves that
\[
\dim_{\F_p}(\Cl(\Q(N^{1/p}))[p]) = 1+ \sum_{i=1}^{p-2} r_i(\chi_0)
\]
 for some non-negative integers $r_i(\chi_0)$, and proves that $r_1(\chi_0)=1$ if and only if $\mathrm{rank}_{\Z_p}(\bT^0)=1$ (see the proof of \cite[Thm.\ 1.1]{lecouturier2018} on pg.~54).
\item It would interesting to see if these finer results can be deduced from Sharifi's theory relating class groups of Kummer extensions to cup products and Massey products \cite{sharifi2007}.
\end{itemize}
\end{rems}
\subsection{Higher rank and Massey products} 
Consider the following matrix of cocycles:
\[
M = \ttmat{a}{b}{c}{-a} \in \ttmat{H^1(\Z[1/Np],\zp{t})}{H^1(\Z[1/Np],\zp{t}(1))}{H^1(\Z[1/Np],\zp{t}(-1))}{H^1(\Z[1/Np],\zp{t})}.
\]
We have the ``matrix cup product" $M \cup M$ given by
\[
M \cup M = \ttmat{a\cup a + b \cup c}{a \cup b - b \cup a}{c \cup a - a \cup c}{c \cup b + a \cup a}.
\]
Using the skew-commutativity of (scalar) cup products, we can see that, if $M \cup M=0$, then $b \cup c = c \cup a = 0$. In fact, one can show that $M \cup M= 0$ if and only if $b \cup c = c \cup a = 0$. This suggests that, in order to generalize Theorem \ref{thm:cup products and rank 1} to higher rank, one should consider ``higher cup powers" of $M$.

We can formalize this by considering $M$ as an element of
\[
 H^1(\Z[1/Np],\End(\zp{t}(1)\oplus \zp{t})) 
 \]
and using the product on $\End(\zp{t}(1)\oplus \zp{t})$. Roughly, for $s\le t$, we define Massey product powers $\dia{M}^k \in  H^2(\Z[1/Np],\End(\zp{s}(1)\oplus \zp{s}))$ of $M$ inductively, assuming $\dia{M}^{k-1}=0$. The base case is the cup product $\dia{M}^2=M \cup M$.

\begin{thm}
\label{thm:intro main 2}
Let $k>1$ and suppose that $\mathrm{rank}_{\Z_p}(\bT^0) \ge k-1$. The following are equivalent:
\begin{enumerate}
\item $\mathrm{rank}_{\Z_p}(\bT^0) \ge k$
\item $\dia{M}^k =0$ in $H^2(\Z[1/Np],\End(\F_p(1)\oplus \F_p))$.
\end{enumerate}
\end{thm}

This theorem is morally correct, but not quite precise: we actually need to choose the extra data of a \emph{defining system} for the Massey product $\dia{M}^k$ to be defined. See Theorem \ref{thm:main in reorg} for a precise statement, and Remark \ref{rem:defining systems} for the fact that vanishing behavior does not depend on the choice of defining system. As with Theorem \ref{thm:cup products and rank 1} in the case $k=2$, for general $k$ the matrix Massey product vanishing $\dia{M}^k =0$ is equivalent to the vanishing of one of its coordinates. See Appendix \ref{sec:massey appendix} for the definition of Massey products and their coordinates, and see Proposition \ref{prop:massey equiv} for the equivalence.

\subsection{Newton polygons}
\label{subsec:newton poly} 

For this subsection, we let $e=\mathrm{rank}_{\Z_p}(\bT^0)$ and recall that $t=v_p(N-1)$. As Mazur noted, there is an isomorphism
\[
\bT^0 \simeq \Z_p\lb y \rb /(F(y)) 
\]
where $F(y)$ is a polynomial of the form
\[
F(y)=\alpha_1+\alpha_2 y + \dots \alpha_{e}y^{e-1}+ y^e, \ F(y) \equiv y^e \pmod{p}, \quad v_p(\alpha_1)=t.
\]
The polynomial $F(y)$ is not determined canonically, but its Newton polygon is. Mazur's original question addressed this Newton polygon, which influences the factoring behavior of $F(y)$. This is interesting because one can read off partial (but, often, complete) information about the number of Eisenstein-congruent cusp forms and the ``depth'' of these congruences from the Newton polygon -- see \cite{BKK2014} for a careful discussion. In particular, one knows that the normalization $\tilde{\bT^0}$ of $\bT^0$ is of the form
\[
\tilde{\bT^0} = \prod_{i=1}^m \sO_{f_i},
\]
where the product is over the normalized eigenforms $f_i$ congruent to the Eisenstein series, and $\sO_{f_i}$ is the valuation ring in the $p$-adic field $\Q_p(f_i)$ generated by the coefficients of $f_i$. In particular, $\mathrm{rank}_{\Z_p}(\sO_{f_i})=[\Q_p(f_i):\Q_p]$, and $m$ equals the number of factors of $F(y)$.

\begin{thm}
\label{thm:newton poly}
The Newton polygon of $\bT^0$ is completely and explicitly determined by the list of integers $t=t_1 \ge t_2 \ge \dots \ge t_{e}>0$, where, for $i>1$, $t_i$ is the maximal integer $s \le t_{i-1}$ such that $\dia{M}^i=0$ in $H^2(\Z[1/Np],\End(\zp{s}(1)\oplus \zp{s}))$.
\end{thm}

For the precise result, see Theorem \ref{thm:main in reorg}. Here is a precise consequence.

\begin{cor}
\label{cor:newton poly}
Assume that $\max(e,t) > 2$ and $\min(e,t) > 1$ and that $M \cup M$ is non-zero in $H^2(\Z[1/Np],\End(\zp{2}(1)\oplus \zp{2}))$.

Then the vertices of the Newton polygon of $\bT^0$ are $\{(0,t),(1,1),(e,0)\}$. In particular, $\bT^0$ is not irreducible, and, moreover, there is a cuspidal eigenform $f$ with coefficients in $\Z_p$ that is congruent modulo $p$ to the Eisenstein series of weight $2$ and level $N$. 
\end{cor}

Note that if $\max(e,t) \le 2$ or $\min(e,t) \le 1$, then there is only one possibility for the Newton polygon. Next, we consider an analytic interpretation of $M \cup M$.

\subsection{Relation to Merel's work} Mazur give a different criterion for $\mathrm{rank}_{\Z_p}(\bT^0)=1$ in terms of the geometry of modular curves \cite[Prop.\ II.19.2, pg.~140]{mazur1978}, and Merel \cite{merel1996} gave a number-theoretic interpretation of this criterion.

\begin{thm}[Merel]
\label{thm:merel}
Assume that $p \mid (N-1)$. The following are equivalent:
\begin{enumerate}
\item $\mathrm{rank}_{\Z_p}(\bT^0)=1$
\item The element of $(\Z/N\Z)^\times$ given by the formula
\[
\prod_{i=1}^{\frac{N-1}{2}} i^i \pmod{N}
\]
is not a $p$-th power.
\end{enumerate}
\end{thm}

There is an alternate formulation, by Calegari and Venkatesh, of this theorem in terms of zeta values, which was explained to us by Venkatesh. Assume that $p^s \mid (N-1)$, and let $G=(\Z/N\Z)^\times$  and $I_G=\ker(\zp{s}[G] \to \zp{s})$ be the augmentation ideal. Consider the element
\[
\zeta = \sum_{i \in (\Z/N\Z)^\times}B_2(\lfloor i/N \rfloor)  [i] \in \zp{s}[G],
\]
where $B_2(x)=x^2-x+1/6$ is the second Bernoulli polynomial and where $\lfloor i/N \rfloor \in [0,1) \cap \frac{1}{N}\Z$ is the fractional part of $i/N$. This element comes from considering the function
\[
\{\chi:(\Z/N\Z)^\times \to \bFp\} \to \bFp, \quad \chi \mapsto L(-1,\chi)
\]
where $\chi$ is a character and $L(s,\chi)$ is the Dirichlet $L$-function. One knows that $L(-1,\mathrm{triv})=\frac{1-N}{12}$, which vanishes in $\zp{s}$ when $p^s \mid (N-1)$, and that $L(-1,\chi) =-\frac{1}{2}B_{2,\chi}$. We use $\zeta$ to give meaning to the ``order of vanishing of $L(-1,\chi)$ at $\chi=\mathrm{triv}$.'' 

Following Mazur and Tate \cite{MT1987}, we let $\mathrm{ord}_s(\zeta) \in \Z$ be the maximal integer $r$ such that $\zeta \in I_G^r$. Then Merel's theorem can be restated as follows.

\begin{thm}[Merel]
\label{thm:merel and zeta}
Assume that $p \mid (N-1)$, and take $s=1$ in the above discussion. Then $\mathrm{ord}_1(\zeta)\ge 1$, and the following are equivalent:
\begin{enumerate}
\item $\mathrm{rank}_{\Z_p}(\bT^0)=1$
\item $\mathrm{ord}_1(\zeta)=1$.
\end{enumerate}
\end{thm}

We give a new proof of this theorem, combining Theorem \ref{thm:cup products and rank 1} with the following proposition.

\begin{prop}
\label{prop:cup equivalence with zeta}
Assume that $p^s \mid (N-1)$. Then $\mathrm{ord}_s(\zeta)\ge 1$, and the following are equivalent:
\begin{enumerate}
\item $M \cup M$ is non-zero in $H^2(\Z[1/Np],\End(\zp{s}(1)\oplus \zp{s}))$.
\item $\mathrm{ord}_s(\zeta)=1$.
\item Merel's number $\prod_{i=1}^{\frac{N-1}{2}} i^i$ is a not a $p^s$-th power modulo $N$.
\end{enumerate}
\end{prop} 

Combining this result with Corollary \ref{cor:newton poly}, we see that if Merel's number is not a $p^2$-th power modulo $N$, then there is only one possibility for the Newton polygon of $\bT^0$. The proof is a variant of Stickelberger theory, and is inspired by the work of Lecouturier \cite{lecouturier2018} and unpublished work of Calegari and Emerton.

In Iwasawa-theoretic parlance, one could see the condition $\mathrm{rank}_{\Z_p}(\bT^0) =1$ as an intermediary between the \emph{algebraic side} (the non-vanishing of cup products) and the \emph{analytic side} ($\zeta$ vanishes to order 1). Based on this, one might conjecture that $\mathrm{rank}_{\Z_p}(\bT^0) = \mathrm{ord}_1(\zeta)$. This is not quite correct, as the examples below show, but it is true strikingly often. In particular, we optimistically conjecture the following, for which our only evidence is a computation for $N<10000$. 
\begin{conj}
\label{conj: rank 2}
Assume that $\mathrm{rank}_{\Z_p}(\bT^0) \ge 2$. Then the following are equivalent:
\begin{enumerate}
\item $\mathrm{rank}_{\Z_p}(\bT^0)=2$
\item $\mathrm{ord}_1(\zeta)=2$.
\end{enumerate}
\end{conj}
\begin{rem}
\label{rem:lecouturier}
After posting an earlier version of this paper on arXiv, we learned that Lecouturier was working on an approach to some of the problems considered in this paper using an ``analytic side'' approach. In particular, he gives another new proof of Merel's Theorem \ref{thm:merel} and proves Conjecture \ref{conj: rank 2}. Lecouturier's work has since appeared as a preprint \cite{lecouturier2018b}. It would be interesting to study the connections between our work and his.
\end{rem}

More generally, Theorem \ref{thm:intro main 2} relates $\mathrm{rank}_{\Z_p}(\bT^0)$ to an ``algebraic side" (vanishing of Massey products). It is natural to ask whether there is a corresponding object on the analytic side -- is there a zeta element $\tilde{\zeta}$ such that $\mathrm{ord}(\tilde{\zeta}) =\mathrm{rank}_{\Z_p}(\bT^0)$?

Finally, we remark that, although we give a new proof of Merel's theorem, it is intriguing to consider the possibility of, in a different context, doing the opposite. That is, Theorem \ref{thm:cup merel equiv} relates an algebraic side (vanishing of cup product) to an analytic side (order of vanishing of zeta element). In a different context where one wants to prove the same type of result (e.g.\ BSD conjecture, Bloch--Kato conjecture), it is interesting to consider if there is an analog of $\mathrm{rank}_{\Z_p}(\bT^0)$ that can serve as an intermediary: on one hand being related to the algebraic side via deformation theory, and on the other hand being related to the analytic side via geometry.

\subsection{Examples} We give some explicit examples, computed using the SAGE computer algebra software. See \cite[Table, pg.~40]{mazur1978} for some relevant computations of $\mathrm{rank}_{\Z_p} (\bT^0)$ (denoted $e_p$ there).

\subsubsection{An example witnessing Corollary \ref{cor:class groups and rank 1}(2)} Take $p=5$ and $N=31$. In this case we have $\mathrm{rank}_{\Z_p} (\bT^0) =2$. One can compute that 
\[
\mathrm{Cl}(\Q(\zeta_N^{(p)},\zeta_p)) \simeq \Z/2\Z \oplus \Z/2\Z \oplus \Z/10\Z \oplus \Z/10\Z.
\]
We see that the $p$-torsion subgroup is non-cyclic, as predicted by Corollary \ref{cor:class groups and rank 1} (2).

\subsubsection{An example where the converse to Calegari--Emerton's result is false}
\label{subsubsec:C-E converse} Take $p=7$ and $N=337$ and note that $7 \mid 336$. One can compute that $\mathrm{Cl}(\Q(N^{1/p})) \simeq \Z/7\Z \oplus \Z/7\Z$. One also checks Merel's number is
\[
\prod_{i=1}^{\frac{N-1}{2}} i^i \equiv 227 \pmod{337}
\]
which is not a $7$th power modulo $337$. In particular, we have $\mathrm{rank}_{\Z_p} (\bT^0) =1$ even though $\mathrm{Cl}(\Q(N^{1/p}))[p]$ is not cyclic. This example was found independently by Lecouturier \cite{lecouturier2018} and in unpublished work of Calegari--Emerton.

\subsubsection{Examples of higher order vanishing} We computed $\mathrm{rank}_{\Z_p}(\bT^0)$ and $\mathrm{ord}(\zeta)$ for every value of $(N,p)$ with $N<10000$. In Table 1, we give a list of all the examples with $\mathrm{rank}_{\Z_p}(\bT^0)>2$. In the $\mathrm{ord}(\zeta)$ column, we only list the result if $\mathrm{rank}_{\Z_p}(\bT^0) \ne \mathrm{ord}(\zeta)$. Note that for all examples with $\mathrm{rank}_{\Z_p}(\bT^0)=2$, we found that $\mathrm{ord}(\zeta)=2$, and vice versa, confirming Conjecture \ref{conj: rank 2} for $N<10000$.
\begin{table}[t]
\begin{center}
\label{table:big}
\begin{tabular}{||c|c|c|c||}
\hline \hline
$N$ & $p$ &  $\mathrm{rank}_{\Z_p}(\bT^0)$ & $\mathrm{ord}(\zeta)$ \\ \hline 
181 & 5 &  3 & \\
 1321 & 11 & 3 & \\
 1381 & 23 & 3 & \\
 1571 & 5 & 3 & \\
 2621 & 5 & 3 & \\
 3001 & 5 & 6 & 7 \\ 
 3671 & 5 & 5 &  3\\ 
 4159 & 7 & 4 & 5\\ 
 4229 & 7 & 3 &  4\\ 
 4931 & 5 & 3 & \\
 4957 & 7 & 3 & \\
 5381 & 5 & 3 & \\
 5651 & 5 & 4 & 5 \\ 
 5861 & 5 & 4 & \\
 6451 & 5 & 3 & \\
 6761 & 13 & 3 &  4 \\
 7673 & 7 & 3 &  4\\ 
 9001 & 5 & 4 & \\
 9521 & 5 & 3 & \\
 \hline \hline
\end{tabular}
\end{center}
\bigskip
\caption{All examples with rank at least 3, $N<10000$}
\end{table}

\subsubsection{Examples where $\bT^0$ is not irreducible} We also computed the ranks of the irreducible components of $\bT^0$ for for every value of $(N,p)$ with $N<10000$. (See \S \ref{subsec:newton poly} for the significance of these ranks.)

In Table 2, we give all examples where $\bT^0$ is not irreducible and list the ranks of the components. For each example having either $v_p(N-1)>2$ or $\mathrm{rank}_{\Z_p}>2$, except for $(N,p)=(3001,5)$, we computed that Merel's number is not a $p^2$-th power modulo $N$, and so the Newton polygon is given by Theorem \ref{thm:newton poly}. In the case $(N,p)=(3001,5)$, Merel's number is a $p^2$-th power modulo $N$, and the Newton polygon has vertices $\{(0,3),(1,2),(3,1),(6,0)\}$.

\begin{table}[t]
\begin{center}
\begin{tabular}{||c|c|c|c||}
\hline \hline
$N$ & $p$ &  $\mathrm{rank}_{\Z_p}(\bT^0)$ & Ranks \\ \hline 
751 & 5 & 2 & (1, 1) \\
2351 & 5 & 2 & (1, 1) \\
3001 & 5 & 6 & (1, 2, 3) \\
3251 & 5 & 2 & (1, 1) \\
3631 & 11 & 2 & (1, 1) \\
3701 & 5 & 2 & (1, 1) \\
4001 & 5 & 2 & (1, 1) \\
5651 & 5 & 4 & (1, 3) \\
6451 & 5 & 3 & (1, 2) \\
6761 & 13 & 3 & (1, 2) \\
7253 & 7 & 2 & (1, 1) \\
9001 & 5 & 4 & (1, 3) \\
9901 & 5 & 2 & (1, 1) \\
\hline \hline
\end{tabular}
\end{center}
\bigskip
\caption{Ranks of irreducible components of $\bT^0$}
\label{table:irr}
\end{table}

\subsection{Statistics} In the previous subsection, we gave examples of pairs $(N,p)$ where $\bT^0$ exhibits exceptional behavior. In this subsection, we analyze the statistical behavior of the examples we computed. This discussion was influenced by discussions with Ravi Ramakrishna. We will consider the situation for $p$ fixed and $N$ varying. To emphasize the dependence on $N$, in this subsection we will write $\bT^0_N$, instead of $\bT^0$, for the Hecke algebra associated to the pair $(N,p)$. 

For fixed $p$, let $P(x)=\{N \ | \ N \text{ is prime}, \ N<x, N \equiv 1 \pmod{p}\}$. Consider the function $r(d,x): \mathbb{N} \times \mathbb{N} \to [0,1]$ given by
\[
r(d,x) = \frac{\#\{N \in P(x) \ | \ \mathrm{rank}_{\Z_p}(\bT^0_N)=d\}}{ \#P(x)}.
\]
Since we computed examples for all $N <10000$, we let $r(d)=r(d,10000)$, and give the values of $r(d)$ for various $p$ and $d$. Before doing this, we explain a heuristic guess for $r(d,x)$ for comparison.

For $N \in P(x)$, we know that $\mathrm{rank}_{\Z_p}(\bT^0_N)=\dim_{\F_p}(\bT^0_N/p)$, and that 
\[
\bT^0_N/p \cong \F_p \lb y \rb / (a_1(N)y + a_2(N) y + \dots ),
\]
where $a_i(N) \in \F_p$. In particular, $\dim_{\F_p}(\bT^0_N/p)= \min\{i \ | \ a_i(N) \ne 0\}$. Our main Theorem \ref{thm:intro main 2} may be interpreted as saying that the numbers $a_i(N)$ can be extracted from values of certain Massey products. If we make the guess that the values $a_i(N)$ are distributed uniformly randomly in $\F_p$ as $N$ varies, we arrive at the following heuristic guess $g(d)$ for $r(d,x)$,
\[
g(d) = \left(\frac{1}{p}\right)^{d-1} \left(\frac{p-1}{p}\right).
\]
Indeed, this is the probability that, for a uniformly randomly chosen sequence $b_1, b_2, \dots, b_d, \dots$ of elements of $\F_p$, we have $b_1=b_2=\cdots=b_{d-1}=0$ and $b_d \ne 0$.

In Table 2, we give our computed values of $r(d)$ for $p=5, 7, 11, 13$, and the relevant values of $g(d)$, to three decimals of precision. In each, case, we let $n=\#P(10000)$, the size of the ``sample space.''
\begin{table}
\begin{center}
\begin{tabular}[t]{|c|c|c|}
    \multicolumn{3}{c}{$\mathbf{p=5}$} \\ \hline \hline
    \multicolumn{3}{|l|}{${n=306}$} \\ \hline
    $d$ & $r(d)$ & $g(d)$ \\ \hline
    1 & 0.745 & 0.800 \\
    2 &  0.216 & 0.160 \\
    3 &  0.023 & 0.032 \\
    4 &  0.010 & 0.006 \\
    5 & 0.003  & 0.001 \\
    6 & 0.003  &0.000 \\ \hline
\end{tabular}
\begin{tabular}[t]{|c|c|c|}
    \multicolumn{3}{c}{$\mathbf{p=7}$} \\\hline \hline
    \multicolumn{3}{|l|}{${n=203}$} \\ \hline
    $d$ & $r(d)$ & $g(d)$ \\ \hline
    1& 0.892 &0.857 \\
    2& 0.089 &0.122\\
    3& 0.015 &0.017\\
    4& 0.005 &0.002\\ \hline
\end{tabular}
\begin{tabular}[t]{|c|c|c|}
    \multicolumn{3}{c}{$\mathbf{p=11}$} \\\hline \hline
    \multicolumn{3}{|l|}{${n=125}$} \\ \hline
    $d$ & $r(d)$ & $g(d)$ \\ \hline
    1& 0.912 &0.909 \\
    2& 0.080 &0.083\\
    3& 0.008 &0.008\\ \hline
\end{tabular}
\begin{tabular}[t]{|c|c|c|}
    \multicolumn{3}{c}{$\mathbf{p=13}$} \\\hline \hline
    \multicolumn{3}{|l|}{${n=99}$} \\ \hline
    $d$ & $r(d)$ & $g(d)$ \\ \hline
    1& 0.929 &0.923 \\
    2& 0.061 &0.071\\
    3& 0.010 &0.005\\ \hline
\end{tabular}
\end{center}
\bigskip
\caption{Distribution of ranks $r(d)$ versus heuristic distribution $g(d)$}
\end{table}

Although the sample size is too small to be convincing, the data seems to align with the heuristic guess. This leads to the question: can one determine the statistical behavior of the Massey products $\dia{M}^k$? Are they uniformly random as $N$ varies?

\subsection{Outline of the proof} The proofs of our main theorems follow the basic strategy of Wiles \cite{wiles1995}: Hecke algebras are related to Galois deformation rings, which are related to Galois cohomology. This is also the strategy used by Calegari--Emerton \cite{CE2005}, but whereas they study ``rigidified" deformations of Galois representations, we use deformation theory of pseudorepresentations, as in our previous work \cite{WWE1}.

\subsubsection{The definition of $R$}
\label{sssec:defR} 

Let $G_\Q$ be an absolute Galois group of $\Q$, and let $G_{\Q,S}$ be its quotient ramified only at the places $S$ supporting $Np\infty$. Let $\Db=\psi(1\oplus \omega)$, a $\F_p$-valued 2-dimensional pseudorepresentation of $G_{\Q,S}$ -- here $\omega$ is the mod $p$ cyclotomic character, and $\psi$ means ``take the associated pseudorepresentation." This is the residual representation modulo $p$ associated to the Eisenstein series of weight 2 and level $N$. We consider deformations $D: G_{\Q,S} \to A$ of $\Db$ subject to the following constraints:
\begin{enumerate}
\item $\det(D)=\kcyc$
\item $D\vert_{I_N} =\psi(1\oplus 1)$, i.e.\ $D$ is trivial on an inertia group $I_N$ at $N$
\item $D\vert_{G_p}$ is ``finite-flat,'' where $G_p$ is a decomposition group at $p$. That is, it arises from the $\oQ_p$-points of a finite flat group scheme over $\Z_p$.
\end{enumerate}
Condition (1) is related to ``weight $2$" and condition (2) is related to ``level $\Gamma_0(N)$" (note that a pseudorepresentation being trivial is analogous to a representation being unipotent). 

Condition (3) is a kind of ``geometricity" condition, and is the most delicate to define. There is a well-known finite-flat deformation theory of representations, due to Ramakrishna \cite{ramakrishna1993}. The difficulty is transferring the notation of ``finite-flat" from representations to pseudorepresentations. We addressed a similar difficulty in our previous work \cite{WWE1} on the \emph{ordinary} condition. In \cite{WWE4}, which started as a companion paper to this one, we present an axiomatic approach to go from properties of representations to properties of pseudorepresentations. This allows us to construct pseudodeformation rings satisfying any ``deformation condition" (in the sense of Ramakrishna). In \S \ref{sec:ffps}, we overview the results of \cite{WWE4} as they apply to finite-flat pseudorepresentations. 

\subsubsection{Proving $R=\bT$} 

Once we have defined $R$, the pseudorepresentation attached to modular forms gives a map $R \to \bT$, and a standard argument shows that it is surjective. We use (a variant of) Wiles's numerical criterion \cite[Appendix]{wiles1995} to prove that the map is an isomorphism. To verify the criterion, we have to compare the $\eta$-invariant to the size of a relative tangent space of $R$. The $\eta$-invariant has been computed by Mazur \cite{mazur1978} using the constant term of the Eisenstein series. 

To study the relative tangent space of $R$, we first consider reducible deformations. These are the simplest deformations, arising as $D=\psi(\chi_1\oplus \chi_\omega)$ where $\chi_1$ and $\chi_\omega$ are characters deforming $1$ and $\omega$, respectively. We show that the ``size" of the space of reducible deformations is equal to the $\eta$-invariant. Next, we use computations in Galois cohomology to show, first, that any square-zero deformation is reducible, and, second, that the space of reducible deformations is cut out by a single equation. This allows us to conclude that the size of the relative tangent space of $R$ is equal to the size of the space of reducible deformations, which we know is the $\eta$-invariant. The numerical criterion then lets us conclude that $R=\bT$ and that both are complete intersections.

As a consequence of our $R=\bT$ theorem, we give new proofs of the results on Mazur on the structure of $\bT^0$, including the Gorenstein property, the principality of the Eisenstein ideal, and the classification of generators of the Eisenstein ideal in terms of ``good primes.''

\subsubsection{Studying deformations} 

Having proven $R=\bT$, we can reduce questions about $\mathrm{rank}_{\Z_p}(\bT^0)$ to questions about $\mathrm{rank}_{\Z_p}(R)$. As a consequence of the proof, we see that the tangent space of $R$ is $1$-dimensional -- in other words, there is a unique (up to scaling) mod $p$ first order deformation $D_1$ of $\Db$. The question of computing $\mathrm{rank}_{\Z_p}(R)$ is reduced to computing to what order $D_1$ can be further deformed. 

Using more detailed Galois cohomology computations, we show that $D_1$ and each of its further deformations (if they exist) arise as the pseudorepresentation associated to a representation. Then we can relate obstruction theory for representations, which is controlled by cup products (and, more generally, Massey products), to obstructions to deforming $D_1$. As explained in \cite{CarlAinf}, the formula of Theorem \ref{thm:intro main 2} determines the highest order unrestricted global deformation of a unique first order deformation corresponding to $M$. Our proofs imply that the local constraints do not contribute additional obstructions. 

\subsection{Acknowledgements} 

We thank Akshay Venkatesh for suggesting that we study this question. We have benefited greatly from conversations with Kevin Buzzard, Frank Calegari, Matt Emerton, Haruzo Hida, Rob Pollack, Ravi Ramakrishna, Barry Mazur, Romyar Sharifi, and Akshay Venkatesh. We thank Frank Calegari for sharing his unpublished notes and for pointing out the related preprint of Lecouturier \cite{lecouturier2018}. We thank Rob Pollack for advice on computing examples. We thank Tony Feng and the anonymous referee for their careful reading of an earlier draft.

The intellectual debt owed to the work of Mazur \cite{mazur1978} and Calegari--Emerton \cite{CE2005} will be obvious to the reader.

P.W.\ was supported by the National Science Foundation under the Mathematical Sciences Postdoctoral Research Fellowship No.~1606255 and Grant No. DMS-1638352. C.W.E.\ was supported by the Simons Foundation under an AMS-Simons travel grant and by the Engineering and Physical Sciences Research Council grant EP/L025485/1.

\subsection{Notation and conventions} 
\label{subsec:conv} 

\

\begin{itemize}[leftmargin=2em]
\item Rings are commutative and algebras are associative but not necessarily commutative. 
\item A representation of an $R$-algebra $E$ is the following data: a commutative $R$-algebra $A$, a finitely generated projective $A$-module $V$ of constant rank, and $R$-algebra homomorphism $\rho: E \to \End_A(V)$. We sometimes write this data as $\rho$ or as $V$, when the meaning is clear from context. 
\item A representation of a group $G$ is a representation of $R[G]$.
\item A character is a representation of constant rank 1.
\item The symbol $\psi(\rho)$ denotes the pseudorepresentation associated to a representation $\rho$. 
\item If $G$ is a profinite group, we let $G^{\mathrm{pro}\text{-}p}$ be the maximal pro-$p$ quotient. If $G$ is finite and abelian, we write $G^{p\mathrm{\text{-}part}}$ instead of $G^{\mathrm{pro}\text{-}p}$.
\item We use the symbol ``$\smile$" for the multiplication in the differential graded algebra of group cochains valued in an algebra, and ``$\cup$" for the cup product of cohomology classes. We sometimes use $[-]$ to denote the cohomology class of a cocycle. If $x,y$ are cocycles, then $[x \smile y] = [x] \cup [y]$, and we often denote this cohomology class by $x \cup y$.
\item Throughout the paper, we abbreviate the cohomology groups $H^i(\Z[1/Np],-)$ (resp.~ $H^i(\Q_\ell,-)$) to $H^i(-)$ (resp. $H^i_\ell(-)$). For further Galois cohomology notation, including the definition of the groups $H^i_{(c)}(-)$, $H^i_\fl(-)$, $H^i_{\fl,p}(-)$, and $H^i_{(N)}(-)$, see Appendix \ref{sec:galois cohomology appendix}.
\item For an integer $i\ge 0$ and a ring $A$, we abbreviate $A[\epsilon]/(\epsilon^{i+1})$ to $A[\epsilon_i]$.
\item We write $v_p(x) \in \Z \cup \{\infty\}$ for the $p$-adic valuation for $x \in \Q_p$.
\item We write $\kcyc:G_{\Q} \to \Z_p^\times$ or $\Z_p(1)$ for the $p$-adic cyclotomic character. When it cannot cause confusion, we abuse notation and write $\kcyc$ for $\kcyc \otimes_{\Z_p} \zp{s}$ or $\zp{s}(1)$.
\end{itemize}

\subsubsection{Notation for Galois groups}
\label{subsec:notation for inertia}
Recall that $N$ and $p$ are prime numbers. We fix algebraic closures $\oQ$ and $\oQ_\ell$, and embeddings $\oQ \rinj \oQ_\ell$, for $\ell = p,N$. This determines decomposition subgroups $G_N \subset G_{\Q}$ and $G_p \subset G_\Q$. We also have the quotient $G_{\Q,S}$ of $G_\Q$ discussed above, the Galois group of the maximal extension of $\Q$ ramified only at the set of places $S$ that support $Np\infty$. The cohomology groups above are the cohomology of continuous cochains on these Galois groups. 

Let $I_N \subset G_N$ and $I_p \subset G_p$ denote the inertia subgroups. We let $I_N^{\mathrm{pro}\text{-}p}$ denote the maximal pro-$p$ quotient of $I_N$. We let $I_N^{\mathrm{non}\text{-}p}$ denote the kernel of the map $I_N \to I_N^{\mathrm{pro}\text{-}p}$.

As is well-known, there is a non-canonical isomorphism $I_N^{\mathrm{pro}\text{-}p} \simeq \Z_p$. We fix, once and for all, a topological generator $\bar{\gamma}$ of $I_N^{\mathrm{pro}\text{-}p}$, and an element $\gamma \in I_N$ mapping to $\bar{\gamma}$.

\part{Pseudo-modularity and the Eisenstein Hecke algebra}

We first recall the results of \cite{WWE4} and construct a pseudodeformation ring with the ``finite-flat" property at $p$. We recall some results of Mazur on modular curves and the Eisenstein Hecke algebra $\bT$, and construct a map $R \to \bT$. We compute Galois cohomology groups to control the structure of $R$, and use the numerical criterion to prove $R \risom \bT$.

\section{Finite-flat pseudodeformations}
\label{sec:ffps}
This section is a summary of \cite{WWE4}. In that paper, we develop the deformation theory of pseudorepresentations with a prescribed property. Presently, we only consider the case that is needed in this paper, where the property is the ``flat'' condition of Ramakrishna \cite{ramakrishna1993}. (To avoid confusion with flat modules over a ring, we refer to this condition as ``finite-flat" in this paper.)

We only give a brief summary of the parts of the theory that are needed in this paper. We assume that the reader has some familiarity with pseudorepresentations and generalized matrix algebras. For a more detailed treatment, see \cite{WWE4}. Other references for pseudorepresentations and generalized matrix algebras include \cite[\S1]{BC2009}, \cite{chen2014}, and \cite[\S\S2-3]{WE2018}. Proofs or references for all of the results in this section are given in \cite{WWE4}; we only give specific references here to the results that are new to \cite{WWE4}.

In this section, we will work in a slightly more general setup than in the rest of the paper. Let $\F$ be a finite field of characteristic $p$. Let $\chi_1,\chi_2:  G_\Q \to \F^\times$ be characters such that $\chi_1|_{G_p} \ne \chi_2|_{G_p}$ and such that $\chi_i|_{G_p}$ are finite-flat representations in the sense defined below. Let $\bar{D}: G_\Q \to \F$ be $\psi(\chi_1 \oplus \chi_2)$, the associated pseudorepresentation. Let $S$ be a finite set of places of $\Q$ including $p$, the infinite places and any primes at which $\chi_i$ are ramified, and let $G_{\Q,S}$ be the Galois group of the maximal unramified-outside-$S$ extension of $\Q$. 

\subsection{Finite-flat representations} We have $G_p\cong \Gal(\overline{\Q}_p/\Q_p)$. Let $\mathrm{Mod}_{\Z_p[G_p]}^\mathrm{tor}$ denote the category of $\Z_p[G_p]$-modules of finite cardinality. Let $\mathrm{ffgs}_{\Z_p}$ denote the category of finite flat group schemes over $\Z_p$ of $p$-power rank. Via the generic fiber functor $\mathrm{ffgs}_{\Z_p} \to \mathrm{Mod}_{\Z_p[G_p]}^\mathrm{tor}$ given by $\cG \mapsto \cG(\overline{\Q}_p)$, which is known to be fully faithful, we can consider $\mathrm{ffgs}_{\Z_p}$ as a subcategory of $\mathrm{Mod}_{\Z_p[G_p]}^\mathrm{tor}$. We call objects in the essential image of this functor \emph{finite-flat} $G_p$-modules. 

Let $\cG_1, \cG_2 \in \mathrm{ffgs}_{\Z_p}$ and let $V_i = \cG_i(\overline{\Q}_p)$ be the associated finite-flat $G_p$-modules. The generic fiber functor defines a homomorphism
\[
\Ext_{\mathrm{ffgs}_{\Z_p}}^1(\cG_2,\cG_1) \to \Ext_{G_p}^1(V_2,V_1).
\]
We define $\Ext_{G_p,\fl}^1(V_2,V_1)$ to be the image of this homomorphism. If $\tilde{V}_i$ are $G_{\Q,S}$-modules such that $\tilde{V}_i|_{G_p}=V_i$, then we define
\[
\Ext_{G_{\Q,S},\fl}^1(\tilde{V}_2,\tilde{V}_1) = \ker \left( \Ext_{G_{\Q,S}}^1(\tilde{V}_2,\tilde{V}_1) \to \frac{\Ext_{G_p}^1(V_2,V_1)}{\Ext_{G_p,\fl}^1(V_2,V_1)}\right).
\]

Let $(A,\m_A)$ be a Noetherian local $\Z_p$-algebra, and let $M$ be a finitely generated $A$-module with a commuting action of $G_p$. Then $M/\m^i_A M \in \mathrm{Mod}_{\Z_p[G_p]}^\mathrm{tor}$ for all $i>0$, and we say $M$ is \emph{finite-flat} if $M/\m^i_AM$ is a finite-flat $G_p$-module for all $i>0$.

\subsection{Generalized matrix algebras}
\label{subsec:GMAs}

Let $(A,\m_A)$ be a Noetherian local $W(\F)$-algebra with residue field $\F$.

See \cite[\S2.1]{WWE4} for the definition of a pseudorepresentation. Let $E$ be an associative $A$-algebra. As noted in \emph{loc.\ cit.}, we may and do think of a pseudorepresentation of dimension $d$ on $E$, written $D:E \to A$ (or, if $E=A[G]$ for a group $G$, as $D:G \to A$), as a rule that assigns to an element $x \in E$ a degree $d$ polynomial $\chi_D(x)(t) \in A[t]$. These $\chi_D(x)$ satisfy many conditions as if they were characteristic polynomials of a representation $E \to M_d(A)$. The \emph{Cayley--Hamilton property} of a pseudorepresentation (defined in \cite[\S1.17]{chen2014}) implies that $\chi_D(x)(x)=0$ in $E$ for all $x \in E$. 

 A \emph{generalized matrix $A$-algebra} or \emph{$A$-GMA} (of type $(1,1)$) is an associative $A$-algebra $E$ equipped with an isomorphism 
\begin{equation}
\label{eq:GMA coords}
\Phi_\cE: E \isoto \ttmat{A}{B}{C}{A}.
\end{equation}
This means an isomorphism of $A$-modules $E \isoto A \oplus B  \oplus C \oplus A$ for some $A$-modules $B$ and $C$, such that the multiplication of $E$ is given by $2 \times 2$-matrix multiplication for some $A$-linear map $B \otimes_A C \to A$. We refer to the isomorphism \eqref{eq:GMA coords} as the \emph{matrix coordinates} of $E$. A morphism of GMAs $(E,\Phi_\cE) \to (E',\Phi_{\cE'})$ is an algebra morphism $\phi: E \to E'$ preserving idempotents; that is, it satisfies $\phi \Phi_\cE^{-1}(\sm{1}{0}{0}{0})=\Phi_{\cE'}^{-1}(\sm{1}{0}{0}{0})$ and $\phi \Phi_\cE^{-1}(\sm{0}{0}{0}{1})=\Phi_{\cE'}^{-1}(\sm{0}{0}{0}{1})$. Forming the trace and determinant as functions $E \ra A$ in the usual way from these coordinates, we have a Cayley--Hamilton pseudorepresentation denoted $D_\cE: E \ra A$. 

An \emph{GMA representation with residual pseudorepresentation $\Db$} is a homomorphism $\rho:G_\Q \to E^\times$ such that, in matrix coordinates, $\rho$ is given as
\[
\rho: \sigma \mapsto \ttmat{\rho_{11}(\sigma)}{\rho_{12}(\sigma)}{\rho_{21}(\sigma)}{\rho_{22}(\sigma)}
\]
with $\rho_{ii}(\sigma) \equiv \chi_i(\sigma) \pmod {\m_A}$. There is an associated pseudorepresentation $\psi_\mathrm{GMA}(\rho):G_\Q \to A$ given by $\tr(\psi_\mathrm{GMA}(\rho)) = \rho_{11} +\rho_{22}$ and $\det(\psi_\mathrm{GMA}(\rho)) = \rho_{11}\rho_{22} - \rho_{12}\rho_{21}$.

A \emph{Cayley--Hamilton representation} of $G_{\Q,S}$ over $A$ with residual pseudorepresentation $\Db$ is a triple $(E,\rho:G_{\Q,S} \to E^\times,D:E \to A)$ where $E$ is an associative $A$-algebra that is finitely generated as an $A$-module, $D$ is a Cayley--Hamilton pseudorepresentation, and $\rho$ is a homomorphism such that $D'=D \circ \rho$ is a pseudorepresentation deforming $\Db$.

\begin{prop}
\label{prop:existence of R_Db}
\ 
\begin{enumerate}
\item The functor sending a complete Noetherian local $W(\F)$-algebra $A$ with residue field $\F$ to the set of deformations $D: G_{\Q,S} \to A$ of $\Db$ is represented by a ring $R_\Db$ and universal pseudodeformation $D^u:G_{\Q,S} \to R_\Db$.
\item There is an $R_\Db$-GMA representation $\rho^u:G_\Q \to E_\Db^\times$ with residual representation $\Db$ such that $(E_\Db,\rho_u,D_\cE)$ is the universal Cayley--Hamilton representation with residual pseudorepresentation $\Db$, and $D^u =D_{\cE} \circ \rho^u$.
\end{enumerate}
\end{prop}

\begin{rem}
\label{rem:idems}
Whenever $\Db$ is multiplicity-free (i.e.\ $\chi_1 \neq \chi_2$, which we have assumed), any Cayley--Hamilton representation $(E,\rho:G_{\Q,S} \to E^\times,D:E \to A)$ with residual pseudorepresentation $\Db$ admits an orthogonal lift $(e_1, e_2)$ of the idempotents $(1,0), (0,1)$ over the kernel of $\chi_1 \oplus \chi_2 : E \ra \F_p \times \F_p$. See e.g.\ \cite[Lem.\ 5.6.8]{WWE1}. We always order the idempotents so that $e_1$ lifts $\chi_1$ and $e_2$ lifts $\chi_2$. It is these idempotents that specify the coordinate decomposition: for example $B = e_1 E e_2$ and $\rho_{i,j}(\gamma) = e_j \rho(\gamma) e_i$ for $i,j \in \{1,2\}$. We also refer to a choice of these idempotents by the corresponding choice of matrix coordinates. 
\end{rem}

\subsection{Finite-flat pseudorepresentations} We retain the notation of the previous subsection.

\begin{defn}
\label{defn:ff pseudo}
Let $(E,\rho,D)$ be a Cayley--Hamilton representation of $G_{\bQ,S}$ over $A$ with residual pseudorepresentation $\Db$. Then $E$ is a finitely generated $A$-module, and it has an action of $G_p$ via $\rho|_{G_p}$ and the left action of $E$ on itself by multiplication. We say that $(E,\rho,D)$ is \emph{finite-flat} if $E/\m_A^i E$ is a finite-flat $G_p$-module for all $i \geq 1$. 

We say a pseudorepresentation $D': G_{\Q,S} \to A$ is \emph{finite-flat} if $D '= D \circ \rho$ for some finite-flat Cayley--Hamilton representation $(E,\rho,D)$.
\end{defn}

We show that there is a universal finite-flat Cayley--Hamilton representation of $G_{\Q,S}$ with residual pseudorepresentation $\Db$. 

\begin{thm}[{\cite[\S2.5]{WWE4}}]
\label{thm:existence of R_flat}
\ 
\begin{enumerate}
\item  There is a universal finite-flat Cayley--Hamilton representation $(E_{\Db,\fl},\rho_\fl: G_{\Q,S} \to E_{\Db,\fl}^\times, D_\fl: E_{\Db,\fl} \to R_{\Db,\fl})$ of $G_{\Q,S}$ over $R_{\Db,\fl}$ with residual pseudorepresentation $\Db$. The algebra $E_{\Db,\fl}$ is a quotient of $E_{\Db}$.
\item The algebra $R_{\Db,\fl}$ is the quotient of $R_\Db$ such that, for any deformation $D:G_{\Q,S} \to A$ of $\Db$, the corresponding map $R_\Db \to A$ factors through $R_{\Db,\fl}$ if and only if $D$ is a finite-flat pseudorepresentation.
\end{enumerate}
\end{thm}

We let
\[
E_{\Db,\fl} = \ttmat{R_{\Db,\fl}}{B_{\Db,\fl}}{C_{\Db,\fl}}{R_{\Db,\fl}}
\]
represent a choice of matrix coordinates of $E_{\Db,\fl}$ induced by those of $E_{\Db}$. 

Finite-flat Cayley--Hamilton representations can arise from endomorphism algebras of modules. The following theorem shows that the notion of finite-flat Cayley--Hamilton representation behaves as expected in this case. 

\begin{thm}[{\cite[\S2.6]{WWE4}}]
\label{thm:ffgs GMA is ff}
Let $(E,\rho,D:E \to A)$ be a Cayley--Hamilton representation of $G_p$, and let $M$ be a faithful $E$-module that is finitely generated as an $A$-module. Consider $M$ as a $A[G_p]$-module via the map $\rho: A[G_p] \to E$. Then $M$ is a finite-flat $G_p$-module if and only if $(E,\rho,D)$ is a finite-flat Cayley--Hamilton representation.
\end{thm}

The following example illustrates the utility of this theorem. It is exactly the situation coming from the Jacobian $J_0(N)$, as encountered by Mazur in \cite[\S\S II.7-8]{mazur1978}, which we apply in \S\ref{subsec: tate module of Jacobian}. 

\begin{eg}
\label{eg:ffgs GMA is ff} 
Let $\cG=\{\cG_i\}$ be a $p$-divisible group with good reduction outside $S$. Then the Tate module $V=T_p\cG = \varprojlim \cG_i(\overline{\Q})$ is a finitely generated, free $\Z_p$-module with an action of $G_{\Q,S}$. In particular, $V$ is a finite-flat representation. 

Now assume that $V$ has a commuting action of $A$, where $A$ is a finite flat $\Z_p$-algebra, and that there is an isomorphism of $A$-modules 
\[
V \cong X_1 \oplus X_2
\]
where $X_i$ are $A$-modules satisfying $\End_{A}(X_i)=A$ (but $X_i$ may not be free as $A$-modules). This decomposition induces a decomposition
\[
\End_A(V) \cong \ttmat{A}{\Hom_A(X_1,X_2)}{\Hom_A(X_2,X_1)}{A}
\]
giving $\End_A(V)$ the structure of an $A$-GMA, where the idempotents arise from projection onto each summand. Let $\rho_V: G_{\Q,S} \to \Aut_A(V)$ be the action map, and let $D_V: \End_A(V) \to A$ be the GMA-pseudorepresentation. Then the theorem implies that $(\End_A(V), \rho_V,D_V)$ is a finite-flat Cayley--Hamilton representation.
\end{eg}

\subsection{Reducibility}
\label{subsec:reducible}
 We say that a pseudorepresentation $D$ is \emph{reducible} if $D=\psi(\nu_1 \oplus \nu_2)$ for characters $\nu_i$. 
\begin{prop}
\label{prop:reducibility}
Let $D:G_{\Q,S} \to R$ be a pseudorepresentation deforming $\Db$.
\begin{enumerate}
\item There is a quotient $R^\red$ of $R$ characterized as follows. For any homomorphism $\phi: R \to R'$, the map $\phi$ factors through $R^\red$ if and only if the composite pseudorepresentation $D'=\phi \circ D:G_{\Q,S} \to R'$ is reducible.
\item Let
\[
E = \ttmat{R}{B}{C}{R}
\]
be a choice of matrix coordinates of $E=E_{\Db} \otimes_{R_{\Db}} R$. Then the image of the $R$-linear map $B \otimes_A  C \to R$ equals the kernel of $R \to R^\red$.
\end{enumerate}
\end{prop}

We call the ideal $\ker(R \to R^\red)$ the \emph{reducibility ideal} of $D$. For the finite-flat pseudodeformation ring, we can describe the reducible quotient.

\begin{prop}[{\cite[\S4.3]{WWE4}}]
\label{prop:flat red}
For $i=1,2$, let $R_i$ denote Ramakrishna's finite-flat deformation ring of the character $\chi_i$, and let $\nu_i:G_{\Q,S} \to R_i^\times$ denote the universal character. Then there is an isomorphism $R_{\Db,\fl}^\red \isoto R_1 \hat{\otimes}_{W(\F)} R_2$ identifying $\psi(\nu_1 \oplus \nu_2)$ as the universal reducible finite-flat deformation of $\Db$.
\end{prop}

\subsection{Reducible GMAs and extensions} 
For this section, we fix a surjective homomorphism $R_{\Db,\fl}^\red \rsurj R'$. By Proposition \ref{prop:flat red}, this homomorphism determines finite-flat characters $\nu_i':G_{\Q,S} \to R'^\times$ deforming $\chi_i$ for $i=1,2$. We can determine the structure of $B_{\Db,\fl} \otimes_{R_{\Db,\fl}}  R'$ and $C_{\Db,\fl} \otimes_{R_{\Db,\fl}}  R'$ in terms of Galois cohomology. 

\begin{prop}[{\cite[\S4.3]{WWE4}}]
\label{prop:red GMA and exts}
Let $M$ be a finitely generated $R'$-module. Then there are canonical isomorphisms
\[
\Hom_{R'}(B_{\Db,\fl} \otimes_{R_{\Db,\fl}}  R',M) \isoto \Ext^1_{G_{\Q,S},\fl}(\nu_2',\nu_1'\otimes_{R'} M)
\]
and 
\[
\Hom_{R'}(C_{\Db,\fl} \otimes_{R_{\Db,\fl}}  R',M) \isoto \Ext^1_{G_{\Q,S},\fl}(\nu_1',\nu_2'\otimes_{R'} M).
\]
\end{prop}

\section{The modular pseudorepresentation}
\label{sec:modular}

In this section, we recall some results of Mazur \cite{mazur1978} on modular curves and Hecke algebras. 

\subsection{Modular curves, modular forms, and Hecke algebras} The statements given here are all well-known. We review them here to fix notations. Our reference is the paper of Ohta \cite{ohta2014}.

\subsubsection{Modular curves} Let $Y_0(N)_{/\Z_p}$ be the $\Z_p$-scheme representing the functor taking a $\Z_p$-scheme $S$ to the set of pairs $(E,C)$, where $E$ is an elliptic curve over $S$ and $C \subset E[N]$ is a finite-flat subgroup scheme of rank $N$. Let $X_0(N)_{/\Z_p}$ be the usual compactification of $Y_0(N)_{/\Z_p}$, and let $\mathrm{cusps}$ denote the complement of $Y_0(N)_{/\Z_p}$ in $X_0(N)_{/\Z_p}$, considered as an effective Cartier divisor on $X_0(N)_{/\Z_p}$. Finally, let
\[
X_0(N)=X_0(N)_{/\Z_p} \otimes \Q_p.
\]

\subsubsection{Modular forms} 
\label{subsubsec:mf}
The map $X_0(N)_{/\Z_p} \to \Spec(\Z_p)$ is known to be LCI, and we let $\Omega$ be the sheaf of regular differentials. Let 
\[
S_2(N;\Z_p) = H^0(X_0(N)_{/\Z_p},\Omega), \quad M_2(N;\Z_p) = H^0(X_0(N)_{/\Z_p},\Omega(\mathrm{cusps}))
\]
There is an element $E \in M_2(N;\Z_p)$ with $q$-expansion
\begin{equation}
\label{eq:eisenstein series}
E = \frac{N-1}{24}+ \sum_{n=1}^\infty \left( \sum_{0<d \mid n,\ N \nmid d} d\right) q^n.
\end{equation}

\subsubsection{Hecke algebras} Let $\bT'$ and $\bT'^0$ be the subalgebras of 
\[
\End_{\Z_p}(M_2(N;\Z_p)), \quad \End_{\Z_p}(S_2(N; \Z_p)),
\]
respectively, generated by all Hecke operators $T_n$ with $(N,n)=1$. These are commutative $\Z_p$-algebras.

Let $I'=\Ann_{\bT'}(E)$, and let $\bT$ be the completion of $\bT'$ at the maximal ideal $(I',p)$, and let $\bT^0=\bT'^0\otimes_{\bT'}\bT$. Let $I = I'\bT$ and let $I^0$ be the image of $I$ in $\bT^0$. Let $U_p \in \bT$ be the unique unit root of the polynomial
\[
X^2-T_pX+p=0,
\]
which exists by Hensel's lemma. Since $T_p-(p+1) \in I$, we see that $U_p - 1 \in I$. For a $\bT'$-module $M$, let $M_\mathrm{Eis}=M \otimes_{\bT'} \bT$. 

There are perfect pairings of free $\Z_p$-modules
\[
M_2(N;\Z_p)_\mathrm{Eis} \times \bT \to \Z_p, \quad S_2(N;\Z_p)_\mathrm{Eis} \times \bT^0 \to \Z_p
\]
given by $(f,t) \mapsto a_1(t\cdot f)$, where $a_1(-)$ refers to the coefficient of $q$ in the $q$-expansion. In particular, $M_2(N;\Z_p)_\mathrm{Eis}$ (resp.\ $S_2(N;\Z_p)_\mathrm{Eis}$) is a dualizing (and hence faithful) $\bT$-module (resp.\ $\bT^0$-module). The map $\bT \to \Z_p$ so induced by $E$ is a surjective ring homomorphism with kernel $I$. We refer to this as the augmentation map for $\bT$.

\subsection{Congruence number} We recall the following theorem of Mazur, and related results.

\begin{thm}[Mazur]
\label{thm:congruence number}
There is an isomorphism $\bT^0/I^0 \simeq \Z_p/(N-1)\Z_p$.
\end{thm}

 This is \cite[Prop.\ II.9.7, pg.~ 96]{mazur1978}. We give a slightly different proof of this theorem using ideas of Ohta and Emerton \cite{emerton1999}. This should help clarify the proof of \cite[Prop.\ 3.2.5]{WWE1}, which uses the same idea but is needlessly complicated; we thank the referee for pointing this out.

 We recall that if $A \to C$ and $B \to C$ are commutative ring homomorphisms, the pullback ring $A \times_C B$ is defined and the underlying set is the same as the pullback in the category of sets.

\begin{lem}
\label{lem:bT is a pull-back}
The composition of the augmentation map $\bT \to \Z_p$ with the quotient map $\Z_p \to \Z_p/(N-1)\Z_p$ factors through $\bT^0$ and induces an isomorphism
\[
\bT \simeq  \bT^0 \times_{\Z_p/(N-1)\Z_p} \Z_p.
\]
In particular, $\ker(\bT \to \bT^0)=\Ann_\bT(I)$.
\end{lem}
\begin{proof}
By \cite[Lem.\ 3.2.3]{ohta2014} there is an exact sequence
\[\tag{$*$}
0 \to S_2(N,\Z_p)_\mathrm{Eis} \to M_2(N,\Z_p)_\mathrm{Eis} \xrightarrow{a_0} \Z_p \to 0,
\]
where the first map is the inclusion and where $a_0(f)$ denotes the constant term in the $q$-expansion of $f$. By duality, we see that $\ker(\bT\to\bT^0) = \Ann_{\bT}(S_2(N,\Z_p)_\mathrm{Eis})$ is the free $\Z_p$-module generated by the element $T_0 \in \bT$ that satisfies $a_1(T_0f)=a_0(f)$ for all $f \in  M_2(N,\Z_p)_\mathrm{Eis}$.

Since $a_0(E)=\frac{N-1}{24}$ maps to $0$ in $\Z_p/(N-1)\Z_p$, we see that the composite $\bT \to \Z_p/(N-1)\Z_p$ factors through $\bT^0$. We have a commutative diagram with exact rows
\[
\xymatrix{
0 \ar[r] & T_0\Z_p \ar[d] \ar[r] & \bT \ar[r] \ar[d] & \bT^0 \ar[r] \ar[d] & 0 \\
0 \ar[r] & (N-1)\Z_p \ar[r] & \Z_p \ar[r] & \Z_p/(N-1)\Z_p \ar[r] & 0
}
\]
where the center vertical map is the augmentation $t \mapsto a_1(tE)$. Since $a_1(T_0E)=\frac{N-1}{24}$, the leftmost vertical map is surjective and hence an isomorphism since the domain and codomain are both free of rank 1. An easy diagram chase then shows that the map 
$\bT \to  \bT^0 \times_{\Z_p/(N-1)\Z_p} \Z_p$ is an isomorphism. The fact that $\ker(\bT \to \bT^0)=\Ann_\bT(I)$ follows formally from this and the fact that $I^0$ is a faithful $\bT^0$-module.
\end{proof}

\subsection{Trace and determinant}

\label{subsec: tate module of Jacobian}

Let $J_0(N)$ be the Jacobian of $X_0(N)$. The $p$-adic Tate module $\mathrm{Ta}_p(J_0(N)(\overline{\Q}))$ is a $\bT'^0[G_{\Q,S}]$-module. Let $\mathcal{T}=\mathrm{Ta}_p(J_0(N)(\overline{\Q}))_\mathrm{Eis}$.

\begin{lem}
The $\bT^0[1/p]$-module $\mathcal{T}[1/p]$ is free of rank $2$.
\end{lem}
\begin{proof}
See \cite[Lem.\ II.7.7, pg.~92]{mazur1978}, for example. 
\end{proof}

Let $\rho_{\mathcal{T}[1/p]}: G_{\Q,S} \to \Aut_{\bT^0[1/p]}(\mathcal{T}[1/p]) \simeq \GL_2(\bT^0[1/p])$ be the corresponding Galois representation.

\begin{lem}
\label{lem:I_N is unipotent}
The representation $\rho_{\mathcal{T}[1/p]} \vert_{I_N}$ is unipotent.
\end{lem}
\begin{proof}
This is proven in the course of the proof of \cite[Prop.\ II.14.1, pg.~113]{mazur1978}, and we recall the argument here. By the theorem of Mazur and Rapoport \cite[Thm.\ A.1, pg.~173]{mazur1978} (attributed there to Deligne), $J_0(N)$ has semi-stable reduction at $N$.  By the crit\`ere Galoisien de r\'eduction semi-stable \cite[Expos\'e IX, Prop.\ 3.5, pg.~350]{SGA7-I}, this implies the result. 
\end{proof}

\begin{lem}
\label{lem:determinant of modular representation}
Let $\ell \nmid Np$ be a prime, and let $\Fr_\ell \in G_{\Q,S}$ be a Frobenius element. Then the characteristic polynomial $\mathrm{char}(\rho_{\mathcal{T}[1/p]})(\Fr_\ell) \in \bT^0[1/p][X]$ is given by
\[
\mathrm{char}(\rho_{\mathcal{T}[1/p]})(\Fr_\ell) = X^2-T_\ell X +\ell.
\]
In particular, we have $\det(\rho_{\mathcal{T}[1/p]})=\kcyc$ and, for any $\sigma \in G_{\Q,S}$, $\tr(\rho_{\mathcal{T}[1/p]}(\sigma)) \in \bT^0$.
\end{lem}
\begin{proof}
The formula for the characteristic polynomial follow from the Eichler--Shimura relation (see e.g.\ \cite[\S II.6, pg.~89]{mazur1978}). The remaining parts follow by Chebotaryov density.
\end{proof}

From this lemma, we see that there is a pseudorepresentation $D_\mathcal{T}: G_{\Q,S} \to \bT^0$ determined by $\det(D_\mathcal{T})=\kcyc$ and $\tr(D_\mathcal{T})(\Fr_\ell)=T_\ell$ for all $\ell \nmid Np$, and that $D_{\mathcal{T}}\otimes_{\bT^0} \bT^0[1/p] = \psi(\rho_{\mathcal{T}[1/p]})$.

\begin{prop}
\label{prop:bT0 structure of TaJ} Assume that $p \mid (N-1)$. There is a short exact sequence of $\bT^0[G_p]$-modules
\[
0 \to \mathcal{T}^{mul} \to \mathcal{T} \to \mathcal{T}^\et \to 0
\]
where $\mathcal{T}^{mul}$ is free of rank $1$ as a $\bT^0$-module and $\mathcal{T}^\et$ is a dualizing $\bT^0$-module. The $G_p$-action on 
$\mathcal{T}^\et$ is unramified, and the sequence splits as $\bT^0$-modules.
\end{prop}
\begin{proof}
The sequence is constructed in \cite[\S II.8, pg.~93]{mazur1978}, using the connected-\'etale exact sequence for the N\'eron model of $J_0(N)$. It follows by construction that the $G_p$-action on $\mathcal{T}^\et$ is unramified. As remarked in \emph{loc.\ cit.}, the sequence is self-$\Z_p$-dual by Cartier duality. Then \cite[Cor.\ II.14.11, pg.~120]{mazur1978} implies that $\mathcal{T}^{mul}$ is a free $\bT^0$-module of rank $1$. By duality, $\mathcal{T}^\et$ is a dualizing $\bT^0$-module. 

Finally, to see that the sequence splits as $\bT^0$-modules, we note that (either by Lemma \ref{lem:determinant of modular representation} or by construction) $G_p$ acts on $\mathcal{T}^{mul}$ by the character $\kcyc$. Let $\tau \in I_p$ be an element such that $\kcyc(\tau)=-1$. Then we see that $\mathcal{T}=(\tau-1)\mathcal{T} \oplus (\tau+1)\mathcal{T}$ as $\bT^0$-modules. 
\end{proof}

\begin{rem}
Note that this proposition does not use the the fact that $\bT^0$ is a Gorenstein ring. See, for example, \cite[Thm.\ 3.5.10]{ohta2014}, where a similar statement is proven in a more general setting where the Hecke algebra need not be Gorenstein.
\end{rem}

\begin{lem}
\label{lem:T_p and Fr_p}
Let $\Fr_p \in G_p$ be a Frobenius element. Then $\Fr_p$ acts on $\mathcal{T}^\et$ by the scalar $U_p \in \bT^0$.
\end{lem}
\begin{proof}
By the previous proposition, we know that $\Fr_p$ acts on $\mathcal{T}^\et$ by a well-defined unit in $\bT^0$. To determine the unit, we extend scalars to $\bT^0[1/p]$. We know that $\mathcal{T}^\et = \mathcal{T}_{I_p}$ (the inertia coinvariants), so it suffices to determine the action of $\Fr_p$ on $(\rho_{\mathcal{T}[1/p]})_{I_p}$. The fact that $\Fr_p$ acts on $(\rho_{\mathcal{T}[1/p]})_{I_p}$ as $U_p$ follows from local-global compatibility for modular forms \cite[Thm.\ 1.2.4(ii)]{scholl1990}.
\end{proof}

We let $E_\mathcal{T}=\End_{\bT^0}(\mathcal{T})$,  and let $\rho_\mathcal{T}:G_{\Q,S} \to E_\mathcal{T}^\times$.  
\begin{cor}
\label{cor:GMA str on E_T}
The $\bT^0$-algebra $E_\mathcal{T}$ admits a $\bT^0$-GMA structure $\cE_\cT$ such that $D_\cT = D_{\cE_\cT} \circ \rho_\cT$, and $D_\cT$ is a finite-flat pseudorepresentation. 
\end{cor}
\begin{proof}
Following Example \ref{eg:ffgs GMA is ff}, a choice of $\bT^0$-module isomorphism $\cT \risom \bT^0 \oplus (\bT^0)^\vee$ arising from Proposition \ref{prop:bT0 structure of TaJ} produces a GMA structure $\cE_\cT$ on $E_\cT$. As in that example, it follows from Theorem \ref{thm:ffgs GMA is ff} that $(E_\cT,\rho_\cT,D_{\cE_\cT})$ is a finite-flat Cayley--Hamilton representation (since $\cT\vert_p$ is a finite-flat $\Z_p[G_p]$-module). It is easy to check that $D_\cT = D_{\cE_\cT} \circ \rho_\cT$, so $D_\cT$ is a finite-flat pseudorepresentation by Definition \ref{defn:ff pseudo}.
\end{proof}

\section{The pseudodeformation ring}
\label{sec:pseudodef ring}
Let $\Db = \psi(\omega \oplus 1)$. In this section, we construct $R$, the universal pseudodeformation ring for $\Db$ satisfying the following additional conditions:
\begin{enumerate}
\item $D$ is finite-flat at $p$
\item $D|_{I_N} =\psi(1\oplus 1)$
\item $\det(D)=\kcyc$
\end{enumerate}
Let $D_\mathrm{Eis}: G_{\Q,S} \to \Z_p$ be the reducible pseudorepresentation $\psi(\Z_p(1) \oplus \Z_p)$. We will show that $D_\mathrm{Eis}$ and the pseudorepresentation $D_\mathcal{T}: G_{\Q,S}\to \bT^0$ both satisfy conditions (1)-(3), and we use this fact to produce a surjection $R \onto \bT$.

\subsection{Construction of $R$}
\label{subsec:construction of R}
Let $R_{\Db,\fl}$ be the universal finite-flat pseudodeformation ring, and let $E_{\Db,\fl}=E_{\Db} \otimes_{R_{\Db}} R_{\Db,\fl}$ be the universal finite-flat Cayley--Hamilton algebra (see Theorem \ref{thm:existence of R_flat}). 

Let $I_{\det} \subset R_{\Db,\fl}$ denote the ideal generated by the set 
\[
\{\det(D)(\sigma) - \kcyc(\sigma) \ | \ \sigma \in G_{\Q,S}\}
\]
and let $I_{ss} \subset R_{\Db,\fl}$ denote the 
ideal generated by the set 
\[
\{\tr(D)(\tau) -2 \ | \ \tau \in I_N\}.
\]
(The notation $I_{ss}$ comes from ``semi-stable at $N$" (cf.\ Lemma \ref{lem:I_N is unipotent}).) Define
\[
R=R_{\Db,\fl}/(I_{\det}+I_{ss}).
\]

\begin{prop}
\label{prop:universal property of R}
The ring $R$ pro-represents the functor sending an Artinian local $\Z_p$-algebra $A$ with residue field $\F_p$ to the set of pseudorepresentations $D: G_{\Q,S} \to A$ satisfying
\begin{enumerate}
\item $D \otimes_A \F_p = \Db$
\item $D$ is finite-flat at $p$
\item $D|_{I_N} =\psi(1\oplus 1)$
\item $\det(D)=\kcyc$.
\end{enumerate}
\end{prop}
\begin{proof}
We already know by Theorem \ref{thm:existence of R_flat} that $R_{\Db,\fl}$ is the deformation ring for pseudorepresentations satisfying (1) and (2). We have to show that, for any $A$ as in the proposition, and any homomorphism $\phi: R_{\Db,\fl} \to A$, the corresponding pseudorepresentation $D$ satisfies (3) and (4) if and only if $\phi$ factors through $R$.

We note that, since $\kcyc$ is unramified at $N$, a pseudorepresentation $D$ satisfying (4) will also satisfy (3) if and only if $\tr(D)|_{I_N}=2$. We see that $D$ satisfies (4) if and only if $\ker(\phi) \supset I_{\det}$, and so $D$ satisfies (3) and (4) if and only if $\ker(\phi) \supset I_{\det}+I_{ss}$. This completes the proof.
\end{proof}

Let $E = E_{\Db,\fl} \otimes_{R_{\Db,\fl}} R$ and let $\rho = \rho_\fl \otimes_{R_{\Db,\fl}} R$. We fix an arbitrary choice of matrix coordinates on $E$, so that we can write $\rho$ as 
\begin{equation}
\label{eq: definition of abcds}
\rho: G_{\Q,S} \to E^\times, \quad \sigma \mapsto \ttmat{a_\sigma}{b_\sigma}{c_\sigma}{d_\sigma}.
\end{equation}
 Let $D = \psi(\rho) : G_{\Q,S} \ra R$ be the universal pseudorepresentation for the functor of Proposition \ref{prop:universal property of R}. 

\subsection{The map $R \to \bT$}
\label{subsec:R to T} First we construct a homomorphism $R \to \bT^0$. 

\begin{lem}
\label{lem:map R to T0}
The pseudorepresentation $D_\mathcal{T}: G_{\Q,S} \to \bT^0$ induces a homomorphism $R \to \bT^0$. Moreover, we have $\tr(D_\mathcal{T})(\Fr_\ell)=T_\ell$ and
\[
\tr(D_\mathcal{T})(\Fr_\ell) \equiv 1+\ell \pmod{I^0}
\]
for any $\ell \nmid Np$.
\end{lem}
\begin{proof}
To show the first statement, we have to check that $D_\mathcal{T}$ satisfies conditions (1)-(4) of Proposition \ref{prop:universal property of R}. Note that the second statement implies (1), so we prove the second statement first. The fact that $\tr(D_\mathcal{T})(\Fr_\ell)=T_\ell$ follows from Lemma \ref{lem:determinant of modular representation}. It follows from the formula \eqref{eq:eisenstein series} that $T_\ell-(1+\ell) \in I^0$, and so the second statement follows.

Condition (2) follows from Corollary \ref{cor:GMA str on E_T}, condition (3) follows from Lemma \ref{lem:I_N is unipotent}, and condition (4) follows from Lemma \ref{lem:determinant of modular representation}.
\end{proof}

\begin{lem}
\label{lem:augmentation of R}
The pseudorepresentation $D_\mathrm{Eis}=\psi(\Z_p\oplus \Z_p(1))$ induces a homomorphism $R \to \Z_p$. Moreover, we have $\tr(D_\mathrm{Eis})(\Fr_\ell)=1+\ell$ for all $\ell \nmid Np$.
\end{lem}
\begin{proof}
The second statement is clear, and implies that $D_\mathrm{Eis}$ satisfies condition (1) of Proposition \ref{prop:universal property of R}. Conditions (3) and (4) are clear, and condition (2) follows from Theorem \ref{thm:ffgs GMA is ff} and the fact that $\Z_p\oplus \Z_p(1)$ is the Tate module of the generic fiber of the $p$-divisible group
\[
(\Q_p/\Z_p \oplus \mu_{p^\infty}) _{/ \Z_p}. \qedhere
\]
\end{proof}

This map $R \to \Z_p$ gives $R$ the structure of an augmented $\Z_p$-algebra. We let $\Jm = \ker(R \to \Z_p)$, and refer to $\Jm$ as the augmentation ideal of $R$. We see that $\Jm \subset R$ is the ideal generated by the reducibility ideal $J$ of $R$ (since $D_\mathrm{Eis}$ is obviously reducible) along with lifts over $R \rsurj R/J$ of the image under $R_{\Db,\fl}^\red \rsurj R/J$ of the set 
\begin{equation}
\label{eq:gens of Jmin}
\{\nu_1(\sigma)-\kcyc(\sigma), \nu_2(\sigma)-1\mid \sigma \in G_{\Q,S}\} \subset R_{\Db,\fl}^\red,
\end{equation}
where $\nu_1, \nu_2$ (the universal finite-flat deformations of $\omega$ and $\F_p$, respectively) arise from Proposition \ref{prop:flat red}. 

Using the two maps $R \to \bT^0$ and $R \to \Z_p$, we can produce a map $R \to \bT$, as in \cite[Cor.\ 7.1.3]{WWE1}. 

\begin{prop}
\label{prop:R to T}
There is a surjective homomorphism $R \rsurj \bT$ of augmented $\Z_p$-algebras. Moreover $\bT$ and $\bT^0$ are generated as $\Z_p$-algebras by the Hecke operators $T_n$ with $(n,Np)=1$. In particular, $\bT$ and $\bT^0$ are reduced. 
\end{prop}
\begin{proof}
We already have $R \to \bT^0$ via $D_\mathcal{T}$, and $R \to \Z_p$ via $D_\mathrm{Eis}$. By Lemma \ref{lem:bT is a pull-back}, to construct a homomorphism $R \to \bT$, is suffices to show that the composite maps
\[
R \to \bT^0 \to \bT^0/I^0 \to \Z_p/(N-1)\Z_p
\]
and
\[
R \to \Z_p \to \Z_p/(N-1)\Z_p
\]
coincide. Equivalently, we have to show that 
\[
D_\mathcal{T} \otimes_{\bT^0} \Z_p/(N-1)\Z_p =  D_\mathrm{Eis} \otimes_{\Z_p} \Z_p/(N-1)\Z_p
\]
as pseudorepresentations $G_{\Q,S} \to \Z_p/(N-1)\Z_p$. However, we have already shown in Lemma \ref{lem:map R to T0} and Lemma \ref{lem:augmentation of R} that these two pseudorepresentations agree at $\Fr_\ell$ for all $\ell \nmid Np$, hence they agree by continuity. This defines a map $R \to \bT$. By construction, we see that the composite map
\[
R \to \bT \to \bT/I \cong \Z_p
\]
coincides with the augmentation $R \to \Z_p$, and so $R \to \bT$ is a map of augmented $\Z_p$-algebras.

Under the isomorphism $\bT \cong \bT^0 \times_{\Z_p/(N-1)\Z_p} \Z_p$, we can think of an element of $\bT$ as a pair $(t,a) \in \bT^0 \times \Z_p$. Then the pseudorepresentation $D_\bT:G_{\Q,S} \to \bT$ corresponding to the map $R \to \bT$ constructed above is given by $D_\bT(\sigma)=(D_\mathcal{T}(\sigma),D_\mathrm{Eis}(\sigma))$. In this notation, for any prime $\ell \nmid Np$, the element $T_\ell \in \bT$ for $\ell \nmid N$ corresponds to the pair $(T_\ell, \ell+1)$, and we see that $\tr(D_\bT)(\Fr_\ell)=(T_\ell, \ell+1)$ for any prime $\ell \nmid Np$. By Chebotaryov density, $R$ is generated as a $\Z_p$-algebra by the elements $\tr(D_\bT)(\Fr_\ell)$ for $\ell \nmid Np$; indeed, by \cite[Cor.\ 2.39]{chen2014}, $R_\Db$ is generated by the values of the characteristic polynomial on Frobenius elements, and its quotient $R$ has determinants valued in $\Z_p$. Then we see that the image of $R \to \bT$ is generated, as a $\Z_p$-algebra, by the elements $T_\ell$.

It remains to show that the image of $R \to \bT$ contains $T_p$. Since $U_p^2-T_pU_p+p=0$, we see that the it is enough to show that the image contains $U_p$ and $U_p^{-1}$. In the notation above, the element $U_p \in \bT$ corresponds to the pair $(U_p,1) \in \bT^0 \times \Z_p$. Choose a Frobenius element $\Fr_p \in G_p$ and let $z=\kcyc(\Fr_p)$. By Proposition \ref{prop:bT0 structure of TaJ} and Lemma \ref{lem:T_p and Fr_p}, we have
\[
\rho_\mathcal{T}(\Fr_p) = \ttmat{zU_p^{-1}}{*}{0}{U_p}.
\]
Choose an element $\sigma \in I_p$ such that $\omega(\sigma) \ne 1$, and let $x=\kcyc(\sigma)$. Then we have
\[
\rho_\mathcal{T}(\Fr_p\sigma) = \ttmat{xzU_p^{-1}}{*}{0}{U_p}.
\]
We see that $\tr(D_\mathcal{T})(\Fr_p\sigma)-x\tr(D_\mathcal{T})(\Fr_p)=(1-x)U_p$. We also see easily that $\tr(D_\mathrm{Eis})(\Fr_p\sigma)-x\tr(D_\mathrm{Eis})(\Fr_p)=1-x$. Hence we see that $((1-x)U_p,1-x)\in \bT$ is in the image of $R \to \bT$. Since 
\[
x \equiv \omega(\sigma) \not \equiv 1 \pmod{p}
\]
we see that $1-x \in \Z_p^\times$, and so we have that $U_p$ is in the image of $R \to \bT$. A similar argument shows that $U_p^{-1}$ is also in the image, completing the proof. 

The operators $T_n$ for $(n, N) = 1$ are well-known to act semi-simply on the modules of modular forms and cusp forms. Since $\bT$ and $\bT^0$ are generated by these operators, we see that they are reduced. 
\end{proof}

\begin{rem}
In \cite{CE2005}, the authors present a proof of a related result. However, the proof of \cite[Lem.\ 3.16]{CE2005} contains a subtle error about the difference between $T_p$ and $U_p$. To correct that error, one would have to argue as above. Similarly, the proof of \cite[Prop.\ 3.18]{CE2005} is flawed and must be corrected as in the proof of Corollary \ref{cor:gens of I} below. 
\end{rem} 

\section{Computation of $R^\red$}

Let $R^\red$ denote the quotient of $R$ representing the pseudodeformations of $\Db$ that satisfy the conditions of Proposition \ref{prop:universal property of R} and are also reducible. Such a quotient exists in light of the theory of reducibility for pseudorepresentations reviewed in \S\ref{subsec:reducible}. In this section we give a presentation of $R^\red$. 

\subsection{Presentation of $R^\red$}
For this section, we let $R' = R_{\Db,\fl}/I_{\det}$ (recall the notation of \S\ref{subsec:construction of R}). 
\begin{lem}
We have $R'^\red \simeq \Z_p[\Gal(\Q(\zeta_N)/\Q)^{p\mathrm{\text{-}part}}]$.
\end{lem}
\begin{proof}
By Proposition \ref{prop:flat red}, we have $R_{\Db,\fl}^\red = R_{1,\fl} \hat{\otimes}_{\Z_p} R_{\omega,\fl}$, where $R_{1,\fl}$ and $R_{\omega,\fl}$ are the finite-flat deformation rings of $1$ and $\omega$, respectively, and the universal deformation is $\psi(\nu_\omega \oplus \nu_1)$, where $\nu_\omega$ and $\nu_1$ are the universal deformation characters. Using the well-known description of the universal deformation ring of a character, and the fact that finite-flat deformations of $1$ (resp.\ $\omega$) are trivial (resp.\ trivial after a twist by $\kcyc^{-1}$) on $I_p$, we have 
\[
R_{\omega,\fl} \cong R_{1,\fl} \cong \Z_p[\Gal(\Q(\zeta_N)/\Q)^{p\mathrm{\text{-}part}}]
\]
and that $\nu_\omega =\kcyc \dia{-}$ and $\nu_1=\dia{-}$, where $\dia{-}$ is the character given by
\[
G_{\Q,S} \onto \Gal(\Q(\zeta_N)/\Q)^{p\mathrm{\text{-}part}} \subset \Z_p[\Gal(\Q(\zeta_N)/\Q)^{p\mathrm{\text{-}part}}]^\times
\]
(the quotient map, followed by map sending group element to the corresponding group-like element). 

By the definition of $I_{\det}$ we see that
\[
R'^\red \cong \frac{R_{\omega,\fl} \hat{\otimes}_{\Z_p} R_{1,\fl}}{(\nu_\omega(\sigma)\otimes\nu_1(\sigma)-\kcyc(\sigma) : \sigma \in G_{\Q,S})} \simeq \Z_p[\Gal(\Q(\zeta_N)/\Q)^{p\mathrm{\text{-}part}}]. \qedhere
\] 
\end{proof}

We fix this isomorphism so that the universal pseudodeformation $D'^\red : G_{\Q,S} \to R'^\red$ can be written as $D'^\red =\psi( \dia{-} \kcyc \oplus \dia{-}^{-1})$. 

Recall from \S \ref{subsec:notation for inertia} that we have chosen an element $\gamma \in I_N$ such that $\gamma$ topologically generates $I_N^{\mathrm{pro}\text{-}p}$. Let $g \in \Gal(\Q(\zeta_N)/\Q)^{p\mathrm{\text{-}part}}$ be the image of $\gamma$ in the quotient. Since $\Gal(\Q(\zeta_N)/\Q)^{p\mathrm{\text{-}part}}$ is the Galois group of a finite $p$-extension of $\Q$ that is totally ramified at $N$, we see that $g$ generates $\Gal(\Q(\zeta_N)/\Q)^{p\mathrm{\text{-}part}}$.

\begin{prop}
\label{prop:presentation of Rred}
There is a presentation
\[
R^\red \simeq \Z_p[X]/(X^2,(N-1)X)
\]
where the universal deformation $D^\red: G_{\Q,S} \to \Z_p[X]/(X^2, (N-1)X)$ is given by $D^\red = \psi( \bar{ \dia{-}} \kcyc \oplus \bar{ \dia{-}}^{-1})$. Here $\bar{\dia{-}}$ is the character $\sigma \mapsto (1+X)^{m_\sigma}$, where $m_\sigma \in \Z/p^n\Z$ is defined so that $\sigma$ maps to $g^{m_{\sigma}}$ in $\Gal(\Q(\zeta_N)/\Q)^{p\mathrm{\text{-}part}}$. 
\end{prop}
\begin{proof} 
Let $t = v_p(N-1)$, so that $\# \Gal(\Q(\zeta_N)/\Q)^{p\mathrm{\text{-}part}}=p^t$. There are isomorphisms
\[
\Z_p[x]/(x^{p^t}-1) \lrisom \Z_p[\Gal(\Q(\zeta_N)/\Q)^{p\mathrm{\text{-}part}}] \cong R'^\red,
\]
where the first sends $x$ to the group-like element $g$.  We use these isomorphisms as identifications in the rest of the proof.

The quotient $R^\red$ of $R'^\red$ corresponds to the condition that $D'^\red$ satisfy $D'^\red|_{I_N} = \psi(1 \oplus 1)$. We know that $\det(D'^\red)=\kcyc$, which satisfies $\kcyc|_{I_N}=1$. Then the only condition is that $\tr(D'^\red)|_{I_N}= 2$. We know that $\tr(D'^\red)|_{I_N} = \dia{-}+\dia{-}^{-1}$. For $\sigma \in I_N$, we have
\[
 \dia{\sigma}+\dia{\sigma}^{-1}=\dia{g^{m_\sigma}} + \dia{g^{m_\sigma}}^{-1} = x^{m_\sigma}+x^{-m_\sigma}.
\]
Since $m_{\gamma}=1$ by our choice of $g$, we see that the condition $\tr(D'^\red)|_{I_N}= 2$ is equivalent to the conditions
\[
x^{m}+x^{-m}=2
\]
for all $m=1, \dots, p^n$. This proves that $R^\red$ is the quotient of $\Z_p[x]$ by the ideal $\mathfrak{a}$ generated by the set
\[
\{x^{p^n}-1\} \cup \{ x^m+x^{-m}-2 \ : \ m= 1, \dots, p^n\}.
\]

It only remains to simplify the presentation. Notice that $x$ is a unit, and that
\[
x^m( x^m+x^{-m}-2) = x^{2m}-2x^m+1 = (x^m-1)^2.
\]
Since this is a multiple of $(x-1)^2$, we see that $\mathfrak{a}$ is generated by $\{x^{p^t}-1,(x-1)^2\}$.

Letting $X=x-1$, notice that 
\[
x^{p^t}-1 = (X+1)^{p^t}-1 \equiv p^tX \pmod{X^2}. 
\]
We see that $\mathfrak{a}$ is generated by $\{p^tX,X^2\}$, so $\mathfrak{a} = ((N-1)X, X^2)$. 
\end{proof}

\subsection{Structure of $\Jm/J$} Recall that $\Jm = \ker(R \to \Z_p)$, where $R \to \Z_p$ is the augmentation defined in Lemma \ref{lem:augmentation of R}. Let $J \subset R$ be the reducibility ideal, so that $R^\red = R/J$. Note that $J \subset \Jm$.

\begin{cor}
\label{cor: size of Jm/J}
We have $\Jm/J \simeq \Z_p/(N-1)\Z_p$.
\end{cor}
\begin{proof}
By Proposition \ref{prop:presentation of Rred}, we have a presentation
\[
R^\red \cong \Z_p[X]/(X^2,(N-1)X),
\]
which we will use as an identification. Then the image of $\Jm$ in $R^\red$ is $XR^\red$ and we have
\[
\Jm/J \cong XR^\red \simeq R^\red/(\Ann_{R^\red}(X)) = R^\red/(X,N-1) \cong \Z_p/(N-1)\Z_p. \qedhere
\]
\end{proof}

\begin{prop}
\label{prop:technical study of Jm}
Let $Y=1 - d_\gamma$ (here $\gamma \in I_N$ is as in \S \ref{subsec:notation for inertia} and $d_\gamma \in R$ is as in \eqref{eq: definition of abcds}). Then $Y \in \Jm$ and the image of $Y$ in $\Jm/J$ is a generator of that cyclic group. Moreover, $Y^2=-b_\gamma c_\gamma \in J$ and there is an inclusion $(\Jm)^2 \subset J$.
\end{prop}
\begin{proof}
The fact that $Y \in \Jm$ is immediate from the description of $\Jm$ in \eqref{eq:gens of Jmin}. By Proposition \ref{prop:presentation of Rred}, we have a presentation
\[
R^\red = \Z_p[X]/(X^2,(N-1)X).
\]
From the proof of that proposition, we see that $Y$ maps to $X$, which generates $\Jm/J$.

To see that $Y^2=-b_\gamma c_\gamma$, note that, in $R$, we have the equation $a_\gamma + d_\gamma =2$ and so $a_\gamma = 1+Y$. Then we have
\[
\rho(\gamma) = \ttmat{a_\gamma}{b_\gamma}{c_\gamma}{d_\gamma} = \ttmat{1+Y}{b_\gamma}{c_\gamma}{1-Y}.
\]
The equation $\det(\rho)(\gamma) = \kcyc(\gamma) = 1$ forces
\[
(1-Y)(1+Y)-b_\gamma c_\gamma =1.
\]
This implies $Y^2 = -b_\gamma c_\gamma$.

Finally, the fact that the image of $Y$ in $\Jm/J$ is a generator implies that $\Jm=YR + J$. Since $Y^2 \in J$, we see that $(\Jm)^2 \subset J$.
\end{proof}

We will use Galois cohomology to see that $J$ is a principal ideal and that $b_\gamma c_\gamma$ is a generator (Theorem \ref{thm:cohomology final answer}). This will imply that $J=(\Jm)^2$ and that $\Jm = YR$ (Corollary \ref{cor:Jm=YR}).

\section{Calculations in Galois cohomology}
\label{sec:cohom calculations}

In this section, our goal is to determine the structure of $E/\Jm E$. We have already determined this structure in terms of Galois cohomology. This was done in Proposition \ref{prop:reducibility}, which we will recall shortly. Therefore, we must calculate various Galois cohomology groups. Namely, certain \emph{global finite-flat cohomology} groups $H^\bullet_\fl(-)$ must be determined. This cohomology theory and other cohomological tools are defined in Appendix \ref{sec:galois cohomology appendix}. The reader will find it necessary to review Appendix \ref{sec:galois cohomology appendix} before following this section's arguments in detail. We use the notation and definitions introduced in Appendix \ref{sec:galois cohomology appendix} freely here.

The calculations of $H^1_\fl$ are crucial to our proof of $R=\bT$ and to our computation of ranks. The calculations of $H^2_\fl$, on the other hand, are not logically necessary for the proofs. We include them as a guide to understand this work in the general context of deformation theory: the groups $H^2_\fl$ are the ``correct $H^2$ groups,'' in that they are the right place to compute the obstructions to lifting a global finite-flat deformation. However, we prove an injectivity result in Proposition \ref{prop:H^2_f in H^2} that implies that it is sufficient to calculate these obstructions in the usual global cohomology $H^2$. Therefore, we can limit the amount of new technology we have to introduce, at the cost of, in places, doing ad hoc work to make a deformation finite-flat. See Remark \ref{rem:global flat massey} for more on this. 

\subsection{Main results}

Recall the notations of \S \ref{subsec:construction of R}. Let
\[
E = \ttmat{R}{B}{C}{R}
\]
be the GMA form of $E$ as in \eqref{eq: definition of abcds}, i.e.\ $B$ and $C$ are $R$-modules, and the multiplication in $E$ induces an $R$-module homomorphism $B \otimes_R C \to R$. We know from Proposition \ref{prop:reducibility} that the image of this homomorphism is the reducibility ideal $J$. 

Let $\Bm=B/\Jm B$ and $\Cm = C/\Jm C$. Since $I_{\det} + I_{ss} \subset \Jm$, the natural maps $B_{\Db,\fl}/\Jm B_{\Db,\fl} \to \Bm$ and 
$C_{\Db,\fl}/\Jm C_{\Db,\fl} \to \Cm$ are isomorphisms. By Proposition \ref{prop:red GMA and exts}, for any $\Z_p$-module $M$ we have
\[
\Hom(\Bm,M) \cong \Ext_{G_{\Q,S},\fl}^1(\Z_p, M(1)), \ \Hom(\Cm,M) \cong \Ext_{G_{\Q,S},\fl}^1(\Z_p(1), M).
\]
In the notation of Appendix \ref{sec:galois cohomology appendix}, this is
\begin{equation}
\label{eq: B and C are dual to cohomology}
\Hom(\Bm,M) = H^1_\fl(M(1)), \quad \Hom(\Cm,M) = H^1_\fl(M(-1)).
\end{equation}

In this section we compute these cohomology groups to reach our goal, the following characterizations of $\Bm$ and $\Cm$. 

\begin{thm}
\label{thm:cohomology final answer}
Let $\gamma\in I_N$ be the element chosen in \S \ref{subsec:notation for inertia}. Recall the notation of \eqref{eq: definition of abcds}.
\begin{enumerate}
\item There are isomorphisms
\[
\Bm \simeq \Z_p, \quad \Cm \simeq \Z_p/(N-1)\Z_p.
\]
\item The $R$-modules $B$ and $C$ are cyclic and $b_\gamma \in B$ and $c_\gamma \in C$ are generators.
\item The ideal $J \subset R$ is principal and $b_\gamma c_\gamma \in J$ is a generator.
\end{enumerate}
\end{thm}

\begin{rem}
While the $\Z_p$-module structures of $\Bm$ and $\Cm$ suffice for the sequel, the interested reader may find it useful to know that there are canonical isomorphisms
\[
\Bm \cong H^2_{\fl^\perp}(\Z_p), \qquad \Cm \cong H^2_{(N)}(\Z_p(2)),
\]
where ``$\fl^\perp$'' refers to the dual condition to the flat condition on the cohomology of $\Z_p(1)$, in the standard sense (see e.g.\ \cite[App.\ B]{GV2018}), but will not be used in our computations. The latter isomorphism is proved in Proposition \ref{prop:C computation}. 
\end{rem}

\begin{rem}
\label{rem:only cyclicity is needed in the proof of R=T}
We note that (2) implies (3) and (2) follows easily from (the proof of) (1). For the proof of $R=\bT$ (Corollary \ref{cor:R=T} below), it is only necessary to prove part (3). To prove (3) directly, one could work exclusively with cohomology with $\F_p$-coefficients, rather than the more cumbersome $\Zr$-coefficients we use below. However, the methods are essentially the same, and the payoff of using $\Zr$-coefficients is the result (1), which is crucial to our study of the finer structure of $R$ and $\bT$ (see \S \ref{subsec:structure of R B and C} and \S \ref{subsec:GMA lemma}).
\end{rem}

The following ``dual'' result to Theorem \ref{thm:cohomology final answer} specifies the cohomology groups generated by the cohomology classes $a,b,c$ of the introduction.

\begin{cor}
	\label{cor:flat cohom of ad}
	Let $s$ be an integer such that $p^s \mid (N-1)$. Then $H^1_\fl(\zp{s}(i)) \simeq \zp{s}$ for $i = -1, 0, 1$. Moreover, 
	\begin{enumerate}
		\item $H^1_\fl(\zp{s})$ is generated by the class of the cocycle 
		\[
		G_{\Q,S} \rsurj \Gal(\Q(\zeta_N)/\Q) \rsurj \zp{s}, 
		\]
		for any choice of surjective homomorphism $\Gal(\Q(\zeta_N)/\Q) \onto \zp{s}$.
		\item $H^1_\fl(\zp{s}(1))$ is generated by the Kummer class of $N$.
		\item $H^1_\fl(\zp{s}(-1))$ is equal to $H^1_{(p)}(\zp{s}(-1))$.
	\end{enumerate}
\end{cor}

Along the way, we also prove the following result, which will be used in our study of obstruction theory for $R$. 
\begin{prop}
\label{prop:H^2_f in H^2}
For any $r>0$ and $i \in \{0,1,-1\}$, the natural map
\[
H^1(\Zr(i)) \lra H^1_{p}(\Zr(i))/H^1_{p,\fl}(\Zr(i))
\]
is surjective. Equivalently, the natural map
\[
H^2_\fl(\Zr(i)) \lra H^2(\Zr(i))
\]
is injective. 
\end{prop}
\begin{rem} 
The equivalence is clear from the cone construction of $H^i_\fl$. See further comments in Remark \ref{rem:global flat massey}. 
\end{rem}

\subsection{Calculation of certain $H^1_{p,\fl}(V)$} 
\label{subsec: calc H1p-fl}

For this section and \S\ref{subsec: cohom compute}, we drop the convention that $N$ is prime and $p \mid (N-1)$, allowing it to be a squarefree integer $N$ such that $p \nmid N$. 

In order to begin computing, we first need to compute some extension groups in the category of finite flat group schemes. Here $\Q_p^\mathrm{nr}$ denotes the maximal unramified subextension of $\oQ_p/\Q_p$. 

\begin{lem}
\label{lem:flat subspaces}
For any $r>0$, we have:
\begin{enumerate}
\item $H^1_{p,\fl}(\Zr(-1))=0$.
\item Under the identification $H^1_p(\Zr(1)) \cong \Q_p^\times \otimes \Zr$ of Kummer theory, $H^1_{p,\fl}(\Zr(1))$ corresponds to the subgroup $\Z_p^\times \otimes \Zr$.
\item $H^1_{p,\fl}(\Zr) = \ker(H^1_p(\Zr) \to H^1(\Q_p^\mathrm{nr},\Zr))$.
\end{enumerate}
\end{lem}
\begin{proof}
\begin{enumerate}[leftmargin=2em]
\item  Indeed, this group corresponds to extensions
\[
0 \lra \Z/p^r\Z \lra \ ? \lra \mu_{p^r} \lra 0
\]
in the category of group schemes of exponent $p^r$ over $\Z_p$, and no such non-trivial extensions exist (see e.g.\ the proof of \cite[Thm.\ 1.8]{ConradBUFLT}). 

\item This can be proven by Kummer theory as in \cite[Lem.\ 2.6]{CE2005}, working over the $fppf$-site of $\Spec(\Z_p)$ (of which the category of finite flat group schemes is an exact subcategory).

\item Indeed, this group corresponds to extensions
\[
0 \lra \Z/p^r\Z \lra \ ? \lra \Zr \lra 0
\]
in the category of group schemes of exponent $p^r$ over $\Z_p$. In such an exact sequence, all the terms must be \'etale, and the category of finite \'etale groups schemes over $\Z_p$ is equivalent to the category of finite abelian groups with $\pi_1^\et(\Z_p)\cong \Gal(\Q_p^{\mathrm{nr}}/\Q_p)$-action. \qedhere
\end{enumerate}
\end{proof}

\subsection{Cohomology computations} 
\label{subsec: cohom compute}

In this section, we state the results of our computations, continuing to allow $N$ to be squarefree where $p \nmid N$. In many cases, when the computation is particularly straightforward and standard, we leave the proofs to the reader.

\begin{prop}
\label{prop: cohom of triv}
We have $H^0_{\fl}(\Z_p)=\Z_p$,  $H^i_{\fl}(\Z_p)=0$ for $i\not\in \{0,2\}$, and 
\[
H^2_{\fl}(\Z_p) \simeq \prod_{\ell | N \mathrm{ prime}} \Z_p/(\ell-1)\Z_p.
\]
\end{prop}
\begin{proof}
Exercise in class field theory.
\end{proof}

\begin{prop}
\label{prop:computation of B}
There are isomorphisms
\[
H^1_{\fl}(\Zr(1)) \cong \Z[1/N]^\times \otimes \Zr \simeq \Zr^{\#\{\ell | N \mathrm{ prime}\}}
\]
and
\[
H^2_\fl(\Zr(1)) \cong \ker \left( \bigoplus_{\ell | N \mathrm{ prime}} \Zr \xrightarrow{\Sigma} \Zr \right)\simeq \Zr^{\#\{\ell | N \mathrm{ prime}\}-1}. 
\]
\end{prop}
\begin{proof}
Exercise in Kummer theory.
\end{proof}

By Lemma \ref{lem:flat subspaces}, we have $H^1_{p,\fl}(\Zr(-1))=0$. Using the notation of \S\ref{subsec:duality}, we have $H^1_\fl(\Zr(-1))= H^1_{(p)}(\Zr(-1))$.

\begin{prop}
\label{prop:C computation}
We have 
\[
\Cm \cong H^2_{(N)}(\Z_p(2)) \cong H^1_N(\Z_p(2)) \simeq \bigoplus_{\ell | N \mathrm{ prime}} \Z_p/(\ell^2-1)\Z_p.
\]
\end{prop}

Proposition \ref{prop:C computation} will follow from the duality Theorem \ref{thm:Poitou--Tate with constraints}, together with the following two lemmas.

\begin{lem}
\label{lem:local C comp}
Let $\ell$ be a prime different from $p$. Then $H^0(\Q_\ell,\Z_p(2)) =0$, 
\[
H^1(\Q_\ell,\Z_p(2)) \simeq \Z_p/(\ell^2-1) \Z_p, \quad \text{and} \quad  
H^2(\Q_\ell,\Z_p(2)) \simeq \Z_p/(\ell-1)\Z_p.
\]
Also, $H^2(\Q_p,\Z_p(2))=0$.
\end{lem}
\begin{proof}
This follows from \cite[Thm.\ 7.3.10, pg.\ 400]{NSW2008}.
\end{proof}

\begin{lem}
\label{lem:global C comp}
For any $p>3$, there are isomorphisms
\[
H^2(\Z_p(2)) \cong \bigoplus_{\ell | N \mathrm{ prime}} \F_\ell^\times \otimes \Z_p \simeq \bigoplus_{\ell | N \mathrm{ prime}}\Z_p/(\ell-1)\Z_p.
\]
For $i \ne 2$, $H^i(\Z_p(2))=0$.
\end{lem}
\begin{proof}
This follows from combining the excision spectral sequence associated to $\Spec(\Z[1/Np])\subset \Spec(\Z[1/p])$ (cf.\ \cite[Prop.\ 1 of III.1.3, pg.\ 18]{Soule}) with the fact that $H^i(\Z[1/p],\Z_p(2))=0$ for $i>0$ if $p>3$. (The Chern class map
\[
c_{i,n}: K_{2n-i}(\Z)\otimes \Z_p \lra H^i(\Z[1/p],\Z_p(n))
\]
is known to be isomorphism, where $K_3(\Z) \simeq \Z/48\Z$ and $K_2(\Z) \simeq \Z/2\Z$.)
\end{proof}

\begin{proof}[Proof of Proposition \ref{prop:C computation}]
By the isomorphism \eqref{eq: B and C are dual to cohomology} along with Lemma \ref{lem:flat subspaces}, we have 
\[
\Cm \cong H^1_{(p)}(\Q_p/\Z_p(-1))^*.
\]
By duality Theorem \ref{thm:Poitou--Tate with constraints}, we have
\[
\Cm \cong H^2_{(N)}(\Z_p(2)).
\]

By Lemma \ref{lem:global C comp}, $H^1(\Z_p(2))=0$. By the duality theorem, $H^3_{(N)}(\Z_p(2))=H^0_{(p)}(\Q_p/\Z_p(-1))^*=0$. Then the cone construction of $H_{(N)}^\bullet$ gives an exact sequence 
\[
0 \to H_N^1(\Z_p(2)) \to H^2_{(N)}(\Z_p(2)) \to H^2(\Z_p(2)) \to H^2_{N}(\Z_p(2)) \to 0.
\]
By Lemmas \ref{lem:local C comp} and \ref{lem:global C comp}, we see that $H^2(\Z_p(2))$ and $H^2_{N}(\Z_p(2))$ are both finite groups of the same order (which is the $p$-part of $\prod_{\ell | N \mathrm{ prime} }(\ell-1)$). Therefore the rightmost surjection in the exact sequence is an isomorphism. Hence we have a canonical isomorphism
\[
 H_N^1(\Z_p(2)) \isoto H^2_{(N)}(\Z_p(2)).
\]
Finally, Lemma \ref{lem:local C comp} gives the computation of $H_N^1(\Z_p(2))$.
\end{proof}

Finally, we complete the proof of Proposition \ref{prop:H^2_f in H^2}. We leave the case of $i=0,1$ to the reader, and sketch the proof of $i=-1$ in the next lemma.

\begin{lem}
\label{lem:surj of H1->H1_p in C coord}
For any $r>0$, there are isomorphisms
\[
H^1_p(\Zr(-1)) \simeq \Zr, \quad H^1_{(p)}(\Zr(-1)) \simeq \bigoplus_{\ell | N \mathrm{prime}} \Z_p/(\ell^2-1,p^r)\Z_p
\]
and there is an exact sequence
\[
0 \to H^1_{(p)}(\Zr(-1)) \to H^1(\Zr(-1)) \to H^1_p(\Zr(-1)) \to 0.
\]
In particular,
\[
\# H^1(\Zr(-1)) =p^r \cdot \prod_{\ell | N \mathrm{prime}} \#\Z_p/(\ell^2-1,p^r)\Z_p.
\]
\end{lem}
\begin{proof}
The isomorphism $H^1_p(\Zr(-1)) \simeq \Zr$ follows from $H^0_p(\Q_p/\Z_p(-1))=0$ and $H^1_p(\Q_p/\Z_p(-1)) \simeq \Q_p/\Z_p$ (see \cite[Thm.\ 7.3.10, pg.\ 400]{NSW2008}). Since $H^3_{(N)}(\Z_p(2))=0$, we have $H^2_{(N)}(\Zr(2))=H^2_{(N)}(\Z_p(2))\otimes \Zr$, so the description of $H^1_{(p)}(\Zr(-1))$ follows from Proposition \ref{prop:C computation} and the duality Theorem \ref{thm:Poitou--Tate with constraints}.

The proof is completed by considering the exact sequence
\begin{equation}
\label{eq:C-H1}
0 \to H^1_{(p)}(\Zr(-1)) \to H^1(\Zr(-1)) \to H^1_p(\Zr(-1))
\end{equation}
and the inequality 
\[
\# H^1(\Zr(-1)) \ge p^r \cdot \prod_{\ell | N \mathrm{prime}} \#\Z_p/(\ell^2-1,p^r)\Z_p. 
\]
This inequality follows from the instance $H^1(\Zr(-1)) \cong H^2_{(c)}(p^{-r}\Z/\Z(2))^*$ of Poitou--Tate duality, along with the three consecutive terms 
\[
0 = H^1(\Z_p(2)) \to H^1_p(\Z_p(2)) \oplus \bigoplus_{\mathrm{prime}\ \ell \mid N} H^1_\ell(\Z_p(2)) \to H^2_{(c)}(\Z_p(2)) 
\]
of the standard long exact sequence in Galois cohomology. Here the leftmost vanishing is recorded in Lemma \ref{lem:global C comp}, $H^1_\ell(\Z_p(2))$ is calculated in Lemma \ref{lem:local C comp}, and it is well-known that $H^1_p(\Z_p(2)) \simeq \Z_p$. 
\end{proof}

\subsection{Proof of Theorem \ref{thm:cohomology final answer} and Corollary \ref{cor:flat cohom of ad}} We now return to the convention that $N$ is prime and $p \mid (N-1)$.

Propositions \ref{prop:computation of B} and \ref{prop:C computation} give us part (1) of the theorem. Part (3) follows from part (2) and the fact that $J=B \cdot C$. It remains to show (2). We give the proof for $B$, the proof for $C$ being almost identical. The strategy will be to use the following version of Nakayama's lemma.

\begin{lem}
Let $(A,\m,k)$ be a local ring and $M$ be a finitely generated $A$-module. Then $M$ is cyclic if and only if the $k$-vector space $\Hom_A(M,k)$ is one-dimensional. If $M$ is cyclic, then an element $m \in M$ is a generator if and only if $\phi(m) \neq 0$ for some non-zero $\phi \in \Hom_A(M,k)$.
\end{lem}

Now we let $\m \subset R$ be the maximal ideal (so $\m = \Jm+pR$). Using \eqref{eq: B and C are dual to cohomology} we calculate 
\begin{equation}
\label{eq: B dual = H1}
\Hom_{R}(B,R/\m)=\Hom_{R}(\Bm,R/\m) \cong H^1_\fl(\F_p(1)).
\end{equation}
Proposition \ref{prop:computation of B} shows that this is a $1$-dimensional $\F_p$-vector space. Hence $B$ is a cyclic $R$-module. Moreover, Proposition \ref{prop:computation of B} implies that any cocycle generating $H^1_\fl(\F_p(1))$ is ramified at $N$.

Now, the maps in \eqref{eq: B dual = H1} are given as follows. Let $\phi \in \Hom_{R}(B,R/\m)$ be non-zero (and hence a generator). Then the corresponding extension of $1$ by $\F_p(1)$ is
\[
\sigma \mapsto \ttmat{\omega(\sigma)}{\phi(b_\sigma)}{0}{1}
\]
If $\phi(b_\gamma)$ were zero, then this extension would be trivial at $I_N$ and hence unramified at $N$. Since we know, by \eqref{eq: B dual = H1}, that this extension generates $H^1_\fl(\F_p(1))$ and that any such generator is ramified at $N$, we must have $\phi(b_\gamma) \neq 0$. The lemma then implies that $b_\gamma$ generates $B$. This completes the proof of the theorem.

To prove Corollary \ref{cor:flat cohom of ad}, first we see that its main statement for $i = \pm 1$ follows directly from Theorem \ref{thm:cohomology final answer} in light of \eqref{eq: B and C are dual to cohomology}. The main statement for $i=0$, along with statements (1) and (2), are basic class field theory. Statement (3) follows immediately from Lemma \ref{lem:flat subspaces}. 

\section{$R = \bT$ and Applications}
\label{sec:R=T}

In this section, we use the numerical criterion to prove that the map $R \to \bT$ constructed in Proposition \ref{prop:R to T} is an isomorphism. We also give further information about the structure of $R$ and the $R$-modules $B$ and $C$.

\subsection{Numerical criterion} We will use the strengthening of Wiles's numerical criterion \cite[Appendix]{wiles1995} due to Lenstra (see \cite[Criterion I, pg.~343]{lci}).

\begin{thm}[Wiles--Lenstra numerical criterion]
Let $\sO$ be a DVR and let $R$ and $T$ be augmented $\sO$-algebras with augmentation ideals $I_R$ and $I_T$ and assume that $T$ is finite and flat over $\sO$. Let $\pi: R \to T$ be a surjective homomorphism of augmented $\sO$-algebras. Let $\eta_T$ be the image of $\Ann_T(I_T)$ in $\sO$. 

Then $\mathrm{length}(I_R/I_R^2) \ge \mathrm{length}(\sO/\eta_T)$ with equality if and only if $\pi$ is an isomorphism of complete intersection rings.
\end{thm}

We apply this to the map $R \to \bT$ constructed in \S\ref{subsec:R to T}. In this case, the DVR $\sO$ is $\Z_p$ and the augmentation ideals are $\Jm \subset R$ and $I \subset \bT$. Let $\eta \subset \Z_p$ be the  image of $\Ann_\bT(I)$ under the augmentation $\bT \to \Z_p$, so that
\[
\Z_p/\eta = \bT/(I+\Ann_\bT(I)).
\]
 By Theorem \ref{thm:congruence number} and Lemma \ref{lem:bT is a pull-back} we have
\[
\Z_p/\eta = \bT/(I+\Ann_\bT(I)) \cong \bT^0/I^0 \cong \Z_p/(N-1)\Z_p.
\] 

On the other hand, we have this consequence of Proposition \ref{prop:technical study of Jm} and Theorem \ref{thm:cohomology final answer}.
\begin{cor}
\label{cor:Jm=YR}
We have $\Jm= Y R$, $J = (\Jm)^2$ and $\Jm/(\Jm)^2 \cong \Z_p/(N-1)\Z_p$.
\end{cor}
\begin{proof}
We already know by Proposition \ref{prop:technical study of Jm} that $\Jm = YR + J$, and by Theorem \ref{thm:cohomology final answer} that $J$ is generated by $b_\gamma c_\gamma$. Since $b_\gamma c_\gamma= -Y^2$, we see that $J \subset YR$ and so $\Jm = YR$. It also follows that $J=(\Jm)^2$, and, since we know by Corollary \ref{cor: size of Jm/J} that $\Jm/J \cong \Z_p/(N-1)\Z_p$, the last part follows as well.
\end{proof}

We can now apply the numerical criterion.

\begin{cor}
\label{cor:R=T}
The surjection $R \rsurj \bT$ from Proposition \ref{prop:R to T} is an isomorphism and both rings are complete intersections.
\end{cor}
\begin{proof}
This is immediate from the numerical criterion: we know that $\bT$ is a finite flat $\Z_p$-algebra and we have the calculations of $\Z_p/\eta$ and $\Jm/(\Jm)^2$.
\end{proof}

\begin{cor}
\label{cor:I and I0 are principal}
The ideals $I \subset \bT$ and $I^0 \subset \bT^0$ are principal. In particular, $\bT^0$ is a complete intersection.
\end{cor}
\begin{proof}
It follows from Corollary \ref{cor:Jm=YR} that $\Jm$ is principal. Since $R \to \bT$ is an isomorphism of augmented algebras, it follows that $\Jm \cong I$ and so $I$ is also principal. Then $I^0$ must also be principal. Since $\bT^0$ is a flat $\Z_p$-algebra and $\bT^0/I^0$ is finite, $I^0$ must be generated by a non-zero divisor. Since $\bT^0/I^0 = \Z_p/(N-1)\Z_p$ is complete intersection, $\bT^0$ is also complete intersection.
\end{proof}

We can also reprove Mazur's results regarding generators of $I$ (see Corollary \ref{cor:gens of I} below).
\subsection{Structure of $R$, $B$ and $C$}
\label{subsec:structure of R B and C}

We have this immediate corollary.

\begin{cor}
\label{cor:R reduced and finite flat}
The ring $R$ is reduced, and it is finite and flat as a $\Z_p$-algebra. 
\end{cor}

\begin{proof}
This follows from the isomorphism $R \isoto \bT$ and the corresponding properties for $\bT$ (\S\ref{subsubsec:mf}, Proposition \ref{prop:R to T}).
\end{proof}

In particular, it follows that any generator of $\Jm$ as an ideal is a generator for $R$ as a $\Z_p$-algebra. Similarly, any generator of $I$ will generate $\bT$ as a $\Z_p$-algebra, as well as its quotient $\bT^0$.

\begin{cor}
\label{cor:structure of R}
Let $Y \in R$ be the element described in Proposition \ref{prop:technical study of Jm}, so that $Y$ is a generator of $\Jm$. Let $g(y) \in \Z_p[y]$ be the monic minimal polynomial of $Y$, so that there is an isomorphism
\[
\Z_p[y]/(g(y)) \lrisom R
\]
given by $y \mapsto Y$. Then $g(y)=yf(y)$ for some $f(y) \in \Z_p[y]$ with $f(y) \equiv y^{\deg{f}} \pmod{p}$ and $f(0)\Z_p = (N-1)\Z_p$, and $\Ann_R(\Jm)$ is the image of the ideal $(f(y))$.
\end{cor}
\begin{proof}
The fact that the map is an isomorphism is a standard exercise. The image of $(y)$ is the augmentation ideal $YR=\Jm$, so reducing modulo $(y)$ we obtain a $\Z_p$-algebra homomorphism $\Z_p/(g(0)) \to \Z_p$, which implies that $g(0)=0$, and so $g(y)=yf(y)$. The annihilator of $(y)$ is $(f(y))$, so the annihilator of $\Jm$ is the image of $(f(y))$. Since $R$ is local and $Y \in \Jm$, $g$ is distinguished and the congruence $f(y) \equiv y^{\deg f} \pmod{p}$ follows. Finally, under the isomorphism $R\risom \bT$, we see that $\Z_p/(f(0))$ corresponds to $\Z_p/\eta=\Z_p/(N-1)\Z_p$, so the valuation of $f(0)$ must equal that of $N-1$.
\end{proof}

We see that $\deg f = \mathrm{rank}_{\Z_p}(R)-1 = \mathrm{rank}_{\Z_p}(\bT^0)$.  We write $R^0=R/\Ann_R(\Jm)$, so that the isomorphism $R \isoto \bT$ induces $R^0 \isoto \bT^0$.

\begin{lem}
\ 
\begin{enumerate}
\item There are isomorphisms $J \simeq \Jm \simeq R^0$ of $R$-modules.
\item Any non-zero ideal $\mathfrak{a} \subset \Ann_R(\Jm)$ is of the form $p^i \Ann_R(\Jm)$ for some $i \ge 0$.
\end{enumerate}
\end{lem}
\begin{proof}
(1) Since both ideals are principal, it suffices to show $\Ann_R(J)=\Ann_R(\Jm)$. But we know that $\Jm=YR$ and $J=Y^2R$, so this follows from the fact that $R$ is reduced (Corollary \ref{cor:R reduced and finite flat}).

(2) By Corollary \ref{cor:structure of R}, we may study subideals of $(f(y))$ in $\Z_p[y]/(yf(y))$. As $\Z_p$-modules, the ideal $(f(y))$ is a free direct summand of $\Z_p[y]/(yf(y))$ of rank $1$. Since any subideal must also by a sub-$\Z_p$-module, the lemma follows.
\end{proof}

\begin{cor}
\label{cor:structure of B and C}
The module $B$ is free of rank $1$ as an $R$ module and there is an isomorphism $C \simeq J$ of cyclic $R$-modules. In particular, the map $B \otimes_R C \to J$ is an isomorphism.
\end{cor}

\begin{proof}
The second sentence follows from the first, since we already have a surjection $B \otimes_R C \onto J$ and the first sentence implies that $B \otimes_R C \simeq J$ as $R$-modules.

By Theorem \ref{thm:cohomology final answer}, $B$ and $C$ are cyclic $R$-modules, so it suffices to show that $B$ is faithful as an $R$-module and that $\Ann_R(C) = \Ann_R(J)$. Since we have a surjection $B \otimes_R C \onto J$, we know that $\Ann_R(B)$ and $\Ann_R(C)$ are subideals of $\Ann_R(J)$. By the previous lemma, we have $\Ann_R(B)$ and $\Ann_R(C)$ are either zero or of the form $p^i\Ann_R(\Jm)$ for some $i\ge 0$.

Now, by Corollary \ref{cor:structure of R}, we have isomorphisms
\[
R/(p^i\Ann_R(\Jm)) \otimes_R R/\Jm \simeq \Z_p[y]/(y,p^if(y)) \simeq \Z_p/(p^if(0)) =\Z_p/p^{i}(N-1)\Z_p.
\]
On the other hand, we know by Theorem \ref{thm:cohomology final answer} that 
\[
B \otimes_R R/\Jm \simeq \Z_p, \quad C \otimes_R R/\Jm \simeq \Z_p/(N-1)\Z_p.
\]
It follows that $\Ann_R(B)=0$ and that $\Ann_R(C)=\Ann_R(\Jm)=\Ann_R(J)$.
\end{proof}

We have the following immediate consequence of foregoing statements. 

\begin{cor}
\label{cor:structure of C}
Let $e=\mathrm{rank}_{\Z_p}(\bT^0)$. Using the isomorphism $R \risom \Z_p[y]/(y f(y))$ of Corollary \ref{cor:structure of R},  the $R$-modules $J$, $\Jm$, $R^0$, and $C$ are isomorphic to $\Z_p[y]/(f(y))$. In particular, we have that $C/pC \simeq \F_p[y]/(y^e)$ as a module for $R/pR \risom \F_p[y]/(y^{e+1})$. 
\end{cor}

\begin{rem}
\label{rem:B and C up to isom} These results on the $R$-module structures of $B$ and $C$ are proven for an arbitrary choice of GMA structure on $E$, and so they hold for any choice of GMA structure. This is not surprising because the modules obtained for a different choice of GMA structure will be a priori isomorphic.
\end{rem}

\begin{rem}
	Using Proposition \ref{prop:bT0 structure of TaJ} and Corollary \ref{cor:structure of B and C}, one can prove that $(E,\rho,D)$ is \emph{ordinary} in the sense of \cite[Defn.\ 5.9.1]{WWE1}.
\end{rem}

\section{The Newton polygon of $\bT$ and a finer invariant}

By the results of the previous section, there are isomorphisms
\[
\bT \simeq \Z_p[y]/(yf(y)), \quad \bT^0 \simeq \Z_p[y]/(f(y)).
\]
These presentations are not canonical, but, as is well-known, the Newton polygon of $f(y)$ is a canonical invariant of $\bT^0$ (and, of course, it can be determined from the Newton polygon of $\bT$). Mazur \cite[\S II.19,\ pg.\ 140]{mazur1978} asked what can be said about this Newton polygon. In this section, we introduce a finer invariant than the Newton polygon and prepare some lemmas to relate it to deformation theory.

\subsection{Newton polygons} For this subsection, we fix $g(x)= \sum_{i=0}^m \alpha_i x^i \in \Z_p[x]$ a monic, distinguished polynomial (i.e.~ $v_p(\alpha_i)>1$ for $i<m$ and $\alpha_m=1$). Note that a coefficient may be zero, so it is possible that $v_p(\alpha_i) = \infty$ in what follows. We slightly abuse terminology by calling these valuations ``integers'' nonetheless.  

\begin{defn}
The \emph{Newton polygon} of $g(x)$ is the lower convex hull of the points $\{(i,v_p(\alpha_i)) : i=0,\dots,m\}$ in $\bR^2$, where a point is omitted when $v_p(\alpha_i) = \infty$. We denote it by $\mathrm{NP}(g)$. 
\end{defn}

Define a sequence $z_0, \dots, z_m$ inductively by 
\[
z_0 = v_p(\alpha_0), \qquad z_i = \min \{z_{i-1}, v_p(\alpha_i)\}\ \text{ for } i=1,\dots, m.
\]
Then it is an easy exercise to see that $\mathrm{NP}(g)$ is the lower convex hull of the points $\{(i,z_i) : i=0,\dots,m\}$. 

Let $T=\Z_p[x]/(g(x))$. We will call an element $y \in T$ a \emph{generator} if $(x)=(y)$ as ideals of $T$. Such an element will also generate $T$ as a $\Z_p$-algebra. Since $T$ is local, we can see that $y=ux$ for some unit $u \in T^\times$. Recall from \S \ref{subsec:conv} that $(\Z/p^r\Z)[\epsilon_i]:= (\Z/p^r\Z)[\epsilon]/(\epsilon^{i+1})$.

\begin{lem}
	\label{lem:NPcontrol}
For $i=0, \dots, m$, define $t_i$ (resp.\ $r_i$) to be the maximal integer $r$ such that there exists a surjective ring homomorphism
\[
\varphi: T \onto (\Z/p^r\Z)[\epsilon_i]
\]	
such that $\varphi(y)=\epsilon$ for some generator $y\in T$ (resp.\ such that $\varphi(x)=\epsilon$).  Then $t_i=r_i=z_i$ for $i=0,\dots, m$.
\end{lem}
\begin{proof}
For $i=0$, a homomorphism $\varphi$ as in the statement must factor through $T/yT=T/xT= \Z_p/\alpha_0\Z_p$, and so we see $t_0=r_0=v_p(\alpha_0)=z_0$. By induction, we can assume the result for $i<n$ for some $1 \le n \le m$, and prove that $t_n=r_n=z_n$. Since the sequence $t_i$ is decreasing and since $r_{n-1} \le t_{n-1}$, we have $v_p(\alpha_i) \ge t_{n-1} \ge r_{n-1}$ for $i=0,\dots,n-1$. 

For $r \le t_{n-1}$, a homomorphism $\varphi: T \onto (\Z/p^r\Z)[\epsilon_n]$ with $\varphi(y)=\epsilon$ for a generator $y\in xT$ must factor through
\[
T/(p^rT+y^{n+1}T) = T/(p^rT+x^{n+1}T) = \zp{r}[x]/(\alpha_nx^n,x^{n+1}).
\]
For any generator $y$, there exists $u(x)=u_0+u_1x+\dots+u_{m-1}x^{m-1}$ with $u_0 \in \Z_p^\times$ such that $x=u(x)y$ in $T$. We see that there is a such a homomorphism $\varphi$ if and only if there is a homomorphism
\[
\zp{r}[x]/(\alpha_nx^n,x^{n+1}) \to (\Z/p^r\Z)[\epsilon_n]
\]
sending $x$ to $\epsilon u(\epsilon)$. Such a homomorphism exists if and only if $\alpha_n \epsilon^n u(\epsilon) = u_0 \alpha_n \epsilon^n$  is $0$ in $(\Z/p^r\Z)[\epsilon_n]$. Similarly, a homomorphism $\varphi$ such that $\varphi(x)=\epsilon$ exists if and only if $\alpha_n\epsilon^n$ is $0$ in $(\Z/p^r\Z)[\epsilon_n]$. Both of these happen if and only if $v_p(\alpha_n) \ge r$, so we see that $t_n=r_n = \min\{t_{n-1},v_p(\alpha_n)\} = \min\{z_{n-1},v_p(\alpha_n)\}=z_n$.
\end{proof}

Note that the integers $t_i$ are an invariant of the pair $(T,(x))$ of $T$ and the ideal $(x) \subset T$ generated by $x$. That is, we emphasize that $\{t_i\}$ do not depend on the particular choice of generator $x$ of $T$. The lemma implies that $\mathrm{NP}(g)$ is the lower convex hull of the points $\{(i,t_i) : i=0,\dots,m\}$, and hence is also an invariant of $(T,(x))$. In applications, $T$ will be $\bT$ or $\bT^0$ and $(x)$ will be the Eisenstein ideal. 

The following example witnesses the fact that the set $\{t_i\}$ is a strictly finer invariant than $\mathrm{NP}(g)$.
\begin{eg} 
Suppose that $g(x)=x^2+\alpha_1x+\alpha_0$, with $v_p(\alpha_0)=2$ and $v_p(\alpha_1)>0$. Then $\mathrm{NP}(g)$ must be the line segment from $(0,2)$ to $(2,0)$, but there are two possible values of $(t_0,t_1,t_2)$: either $(2,1,0)$ or $(2,2,0)$. Moreover, the two different possible values of $t_1$ encode information about $g(x)$. For example, if $g(x)$ is reducible and $t_1=2$, then the two roots of $g(x)$ can be additive inverses of each other, but not if $t_1=1$. This applies to the generators of $\bT$ given in Corollary \ref{cor:gens of I}. 
\end{eg}

\subsection{The Newton polygon of $\bT$}
\label{subsec:NPdef of t_i's} 

For the remainder of the paper, we will be interested in studying the integers $t_i$ associated by Lemma \ref{lem:NPcontrol} to $\bT$. Combining Lemma \ref{lem:NPcontrol} with the results of \S\ref{subsec:structure of R B and C}, we have the following.
\begin{prop}
\label{prop:NPcontrol of T}
Let $y \in \bT$ be a generator of $I$, so that $\Z_p[x]/(g(x)) \risom \bT$ via $x \mapsto y$, where $g(x)= \sum_{i=0}^{e+1} \alpha_i x^i \in \Z_p[x]$ is the monic minimal polynomial of $y$. Define $t_i$ inductively by $t_0 = v_p(\alpha_0)$ and $t_i = \min \{t_{i-1}, v_p(\alpha_i)\}$ for $i=1,\dots, e+1$. Then the sequence $\{t_i\}$ is independent of the choice of $y$, and $NP(g)$ is the lower convex hull of the set $\{(i,t_i)\}$. 

Moreover, for any $0 \le n \le e+1$ and any positive integer $s$, the following are equivalent:
\begin{enumerate}
\item $s \le t_n$,
\item For any generator $z \in \Jm$, there is a homomorphism $\varphi: R \onto \zp{s}[\epsilon_n]$ such that $\varphi(z)=\epsilon$.
\end{enumerate}
\end{prop}

Note that  $t_0 = \infty$, $t_1=v_p(N-1)$, $t_i>0$ for $i\le e$, and $t_{e+1}=0$.

\part{Massey products and deformations}

In this part, we use the results of the previous part to show that the tangent space of $R$ is $1$-dimensional and choose an explicit basis $D_1$ of this space. We consider three representations $\rho_0, \rho_0^c$ and $\rho_0^b$ that all have the same pseudorepresentation and relate the existence of deformations of these three representations to the structure of $R$ and to different Massey products.

For this whole part, we fix the integer $t=v_p(N-1) \geq 1$.

\section{The tangent space and cocycles}
\label{sec:cup}

In this section, we study the tangent space of $R$. Recall from \S \ref{subsec:conv} the notation that $\Z/p^s\Z[\epsilon_i]:= \Z/p^s\Z[\epsilon]/(\epsilon^{i+1})$. The numbering makes $\Hom(R,\Z/p^s\Z[\epsilon_i])$ the space of $i$-th order deformations modulo $p^s$. The tangent space (modulo $p^s$) is the space of first order deformations modulo $p^s$. 

Recall the other notations introduced in \S \ref{subsec:conv}, including the cyclotomic character $\kcyc$ and the element $\gamma \in I_N$.

\subsection{The tangent space of $R$ and generators of $\bT$}
We can describe the tangent space of $R$ using our explicit presentation of $R$ in Corollary \ref{cor:structure of R}. 

\begin{prop}
\label{prop: tangent space}
Let $t = v_p(N-1)$.
\begin{enumerate}
\item The $\F_p$-vector space $\Hom(R,\F_p[\epsilon_1])$ is $1$-dimensional. Any non-zero element of this space sends $\Jm$ to $(\epsilon)$ and $J$ to $0$.
\item There exists a local surjection $R \rsurj \zp{s}[\epsilon_1]$ if and only if $s \le t$. Any such surjection sends $\Jm$ to $(\epsilon)$ and $J$ to $0$.
\item Let $Y=1-d_\gamma$. Let $\varphi_1: R \to \zp{t}[\epsilon_1]$ be the unique homomorphism sending $Y$ to $\epsilon$. Let $a: G_{\Q,S} \to \zp{t}$ be the unique homomorphism factoring through $\Gal(\Q(\zeta_N)/\Q)$ and sending $\gamma$ to $1$. 

Then the pseudorepresentation $D_1:G_{\Q,S} \to  \zp{t}[\epsilon_1]$ associated to $\varphi_1$ is given by $\det(D_1)=\kcyc$ and 
\[
\tr(D_1)=(\kcyc+1)+ \epsilon a(\kcyc-1). 
\]
\end{enumerate}
\end{prop}
\begin{proof}
Parts (1) and (2) are clear from the structure of $R$ computed in Corollary \ref{cor:structure of R}. Since $J \subset \ker(\varphi_1)$, we see that $\varphi_1$ factors through $R^\red$, so part (3) follows from Proposition \ref{prop:presentation of Rred}.
\end{proof}

The following corollary was first proven by Mazur \cite[Prop.\ II.16.1, pg.~125]{mazur1978}. Recall that Mazur calls a prime number $\ell \ne N$ a \emph{good prime (for $N$ and $p$)} if both of the following are true: (i) $\ell \not \equiv 1 \pmod{p}$ and, (ii) $\ell$ is not a $p$-th power modulo $N$. 

\begin{cor}
\label{cor:gens of I}
Let $\ell \ne N$ be a prime number. Then $T_\ell-(\ell +1) \in I$ is a generator of the principal ideal $I$ if and only if $\ell$ is a good prime.
\end{cor}
Note that any generator of $I$ is also a generator of $\bT$ (and $\bT^0$) as a $\Z_p$-algebra. 

\begin{proof}
In this proof, for $x \in \Z_p$, we write $\bar{x}$ for the reduction modulo $p$ of $x$, and we define $\bar D_1=D_1 \otimes_{\zp{t}[\epsilon_1]} \F_p[\epsilon_1]$.

First, assume $\ell \ne p$. As in the proof of Proposition \ref{prop:R to T}, we have $\tr(D_\bT)(\Fr_\ell)=T_\ell$, so we see that $T_\ell-1-\ell$ is a generator of $I$ if and only if $\tr(\bar D_1)(\Fr_\ell)-1-\bar{\ell}$ is non-zero. Since
\[
\tr(\bar D_1)(\Fr_\ell)-(1 + \bar{\ell}) = (\bar{\ell}-1)\bar{a}(\Fr_\ell)\epsilon
\]
we see that $T_\ell-(1+\ell)$ is a generator of $I$ if and only if $(\bar{\ell}-1)\bar{a}(\Fr_\ell) \ne 0$, which happens if and only if $\ell \not \equiv 1 \pmod{p}$ and $\bar{a}(\Fr_\ell) \ne 0$. It follows from class field theory that $\bar{a}(\Fr_\ell) \ne 0$ if and only if $\ell$ is not a $p$-th power modulo $N$.

Now let $\ell=p$. Since $T_p=U_p+pU_p^{-1}$ we see that the images of $T_p-(p+1)$ and $U_p-1$ in $\F_p[\epsilon_1]$ are the same. In particular, $T_p-(p+1)$ generates $I$ if and only if $U_p-1$ generates $I$. Now let $\Fr_p \in G_p$ be a Frobenius element, choose $\sigma \in I_p$ such that $\omega(\sigma)\neq 1$, and let $x=\kcyc(\sigma)$. Then, as in the proof of Proposition \ref{prop:R to T}, we have
\[
U_p = \frac{1}{1-x} ( \tr(D_\bT)(\Fr_p\sigma) - x \tr(D_\bT)(\Fr_p) )
\]
 so $U_p-1$ generates $I$ if and only if
 \[
\frac{1}{1-\bar{x}} \left( \tr(\bar D_1)(\Fr_p\sigma) - \bar{x} \tr(\bar D_1)(\Fr_p) \right)\neq 1.
 \]
 Using the fact that $a$ is unramified at $p$, we see that
 \[
 \frac{1}{1-\bar{x}} \left( \tr(\bar D_1)(\Fr_p\sigma) - \bar{x} \tr(\bar D_1)(\Fr_p) \right)  = 1+\bar{a}(\Fr_p)\epsilon.
 \]
Hence we see that $U_p-1$ generates $I$ if and only if $\bar{a}(\Fr_p) \ne 0$ and the proof continues as above.
\end{proof}

\subsection{A normalization for certain cocycles} 
\label{subsec:matrix reps}

We now depart from the notation of \S\ref{subsec:GC} where $a,b$ and $c$ were cohomology classes chosen up to multiplication by $(\zp{t})^\times$. We let $a \in Z^1_\fl(\zp{t})$ be the cocycle defined in Proposition \ref{prop: tangent space}, let $b \in Z^1_\fl(\zp{t}(1))$ be a Kummer cocycle associated to a choice of $p^t$-th root of $N$, and let $c \in Z^1_\fl(\zp{t}(-1))$ be an element such that $c\vert_p = 0$ and whose image in $H^1_\fl(\zp{t}(-1))$ is a generator. Recall from Corollary \ref{cor:flat cohom of ad} that $H^1_\fl(\zp{t}(i)) \simeq \zp{t}$, and that the classes of $a,b$ and $c$ are generators. We have specified $a$ completely, and $b$ and $c$ up to a multiple of $(\zp{t})^\times$.

Next, as with $a$, we want to normalize $b$ and $c$ with respect to our choice of $\gamma \in I_N$ from \S\ref{subsec:notation for inertia}. Since $b(\gamma), c(\gamma) \not\equiv 0 \pmod{p}$, we can and do normalize so that $b(\gamma)=-1$ and $c(\gamma)=1$. Because of this choice we have 
\[
a(\gamma)^2+b(\gamma)c(\gamma)=0.
\]
Since $a,b,c$ are continuous homomorphisms on $I_N$, this implies that
\begin{equation}
\label{eq:normalization of a,b,c}
(a^2+bc)|_{I_N}=0.
\end{equation}

\begin{rem}
Note that these cocycles $a,b,c$ are not related to the elements $a_\sigma, b_\sigma, c_\sigma$ introduced in \eqref{eq: definition of abcds}. We write cochains in function notation (i.e.~$a(\sigma)$), so hopefully this does not cause confusion.
\end{rem}

\section{Matrix-valued deformations}
\label{sec: MVD}

Let $\rho_0: G_{\Q,S} \to  \GL_2(\zp{t})$ be the representation $\zp{t}(1) \oplus \zp{t}$. From the choice of the cocycles $a,b,c$ in \S \ref{subsec:matrix reps}, we have a first order deformation of $\rho_0$. Namely, let
\[
M = \ttmat{a}{b}{c}{-a} \in Z^1(\End(\zp{t}(1)\oplus \zp{t}))
\]
and define
\[
\rho_1 = (1 + M\epsilon) \rho_0 : G_{\Q,S}\to \GL_2(\zp{t}[\epsilon_1]).
\]
This is a finite-flat representation such that $\psi(\rho_1) = D_1$.

We specify two more representations $\rho_0^b, \rho_0^c: G_{\Q,S} \to  \GL_2(\zp{t})$ satisfying $\psi(\rho_0)= \psi(\rho_0^b)= \psi(\rho_0^c)$, namely
\[
\rho_0^b = \ttmat{\kcyc}{b}{0}{1}, \quad \rho_0^c =  \ttmat{\kcyc}{0}{\kcyc c}{1}
\]
In this section, we will consider deformations of $\rho_0, \rho_0^b$ and  $\rho_0^c$, and how they are related to $R$.

\subsection{Notation for deformations and Massey products} We define the notions of \emph{good}, \emph{very good}, and \emph{adapted} deformations.
\begin{defn}
Let $1\le r \le s$ and $0 \le n \le m$ be integers. Let $\nu: G_{\Q,S}\to \GL_2(\zp{s}[\epsilon_n])$ be a representation. A representation $\nu': G_{\Q,S}\to \GL_2(\zp{r}[\epsilon_m])$ is called an \emph{$m$-th order deformation of $\nu$ modulo $p^r$} if there is an isomorphism
\[
\nu' \otimes_{\zp{r}[\epsilon_m]} \zp{r}[\epsilon_n] \cong \nu \otimes_{\zp{s}[\epsilon_n]} \zp{r}[\epsilon_n].
\]
\end{defn}

\begin{defn}
\label{defn:good}
Let $r \le t$, and let $\rho_n: G_{\Q,S} \to \GL_2(\zp{r}[\epsilon_n])$ be a representation. We call $\rho_n$ \emph{good} if the following conditions are satisfied:
\begin{enumerate}
\item $\psi(\rho_n) \otimes_{\zp{r}[\epsilon_n]} \zp{r} = \psi(\rho_0) \otimes_{\zp{t}} \zp{r}$.
\item $\det(\rho_n)=\kcyc$.
\item $\rho_n|_p$ is finite-flat and upper-triangular.
\item $\tr(\rho_n)|_{I_N}=2$. 
\end{enumerate}
When $\rho_n$ is a good $n$-th order deformation of $\rho_1$ modulo $p^s$, we define $\chi_a(\rho_n),\chi_d(\rho_n): G_{\Q_p} \to (\zp{s}[\epsilon_n])^\times$ to be the diagonal characters of $\rho_n|_p$.
\end{defn}

Note that after assuming (2), (4) is equivalent to $\psi(\rho_n)\vert_{I_N}$ being trivial, cf.\ \S\ref{sssec:defR} and Proposition \ref{prop:universal property of R}. Thus a good representation induces a surjective homomorphism $R \onto \zp{r}[\epsilon_n]$ corresponding to $\psi(\rho_n)$. 

\begin{defn}
Suppose that $\rho_n$ is an $n$-th order deformation of $\rho_1$ modulo $p^s$ with $s\le t$. We call $\rho_n$ \emph{mildly ramified at $N$} if it satisfies
\[
\rho_n|_{I_N} = \left.\ttmat{1+a\epsilon}{b\epsilon}{c\epsilon}{1-a\epsilon}\right\vert_{I_N}.
\]
We call $\rho_n$ \emph{very good} if it is good and mildly ramified at $N$. 
\end{defn}

\begin{defn}
Let $1\le s \le t$ and let $\rho_n$ be an $n$-th order deformation of $\rho_1$ modulo $p^s$, and write $\rho_n$ as
\[
\rho_{n} = \left( 1+ \sum_{i=1}^{n} \ttmat{a_i}{b_i}{c_i}{d_i} \epsilon^i\right) \rho_0 .
\]
We say that an $n$-th order deformation $\rho_n^c$ of $\rho_0^c$ modulo $p^s$ is \emph{adapted to $\rho_n$} if
\[
\rho_n^c = \rho_0^c+ \sum_{i=1}^{n} \ttmat{\kcyc a_i}{b_{i-1}}{\kcyc c_{i+1}}{d_i} \epsilon^i
\]
for some $c_{n+1} \in C^1(\zp{s}(-1))$, where $b_0:=0$. In this situation, we call $c_{n+1}$ the \emph{cochain associated to} $\rho_n^c$, and we note that $\rho_n^c \mapsto c_{n+1}$ is a bijective correspondence between the set of $n$-th order deformations $\rho_n^c$ of $\rho_0^c$ modulo $p^s$ that are adapted to $\rho_n$ and the set of cochains $c_{n+1}$ satisfying 
\[
dc_{n+1} = \sum_{i=1}^n c_i \smile a_{n+1-i} + d_i \smile c_{n+1-i}.
\]

Similarly, we say that an $n$-th order deformation $\rho_n^b$ of $\rho_0^b$ modulo $p^s$ is \emph{adapted to $\rho_n$} if
\[
\rho_n^b = \rho_0^b+ \sum_{i=1}^{n} \ttmat{\kcyc a_i}{b_{i+1}}{\kcyc c_{i-1}}{d_i} \epsilon^i
\]
for some $b_{n+1} \in C^1(\zp{s}(1))$, where $c_0:=0$. We also call $b_{n+1}$ the \emph{cochain associated to} $\rho_n^b$, and note that there is a similar bijection $\rho_n^b \mapsto b_{n+1}$.
\end{defn}

One readily calculates that when $\rho^c_n$ and $\rho^b_n$ are adapted to $\rho_n$ as above, then there is an equality of pseudorepresentations 
\begin{equation}
\label{eq:psi-ad}
\psi(\rho_n) = \psi(\rho^c_n) = \psi(\rho^b_n) : G_{\bQ,S} \lra \zp{s}[\epsilon_n]. 
\end{equation}

We introduce some notation for Massey products, Massey powers, and their connection with deformations; see Appendix \ref{sec:massey appendix} for full details. For $s \le t$, let $M_s$ denote the image of $M$ in $Z^1(\End(\zp{s}(1) \oplus \zp{s}))$.  If $\rho_n$ is an $n$-th order deformation of $\rho_1$ modulo $p^s$, it provides a defining system $D$ for the Massey power $\dia{M_s}^{n+1}$. We will abuse notation and say \emph{$\dia{M}_D^{n+1}$ vanishes in $H^2(\End(\zp{s}(1) \oplus \zp{s}))$} if $\dia{M_s}_D^{n+1}=0$, and refer to the Massey relations for $\dia{M_s}_D^{n+1}$ as the \emph{Massey relations for $\dia{M}_D^{n+1}$ modulo $p^s$}. 

\begin{rem}
In the notation of the previous paragraph, if $r\leq s$, then $\rho_n \otimes_{\zp{s}[\epsilon_n]} \zp{r}[\epsilon_n]$ is a deformation of $\rho_1$ modulo $p^r$, which provides a defining system $D_r$ for the Massey power $\dia{M_r}^{n+1}$. Examining the definition, one sees that $\dia{M_r}_{D_r}^{n+1}$ is the image of $\dia{M_s}_D^{n+1}$ under the natural map 
\[
H^2(\End(\zp{s}(1) \oplus \zp{s})) \to H^2(\End(\zp{r}(1) \oplus \zp{r}))
\]
and similarly for the coordinate Massey relations. In particular, one sees that the following are equivalent:
\begin{enumerate}
\item $\dia{M_r}_{D_r}^{n+1}$ is $0$ in  $H^2(\End(\zp{r}(1) \oplus \zp{r}))$.
\item $\dia{M_s}_D^{n+1}$ is in the kernel of the natural map \[H^2(\End(\zp{s}(1) \oplus \zp{s})) \to H^2(\End(\zp{r}(1) \oplus \zp{r})).\]
\item $\dia{M_s}_D^{n+1}$ is in the image of the map \[H^2(\End(\zp{s}(1) \oplus \zp{s})) \to H^2(\End(\zp{s}(1) \oplus \zp{s}))\] induced by multiplication by $p^r$.
\end{enumerate}
This perhaps justifies the abuse of notation.
\end{rem}

\subsection{GMA lemmas}
\label{subsec:GMA lemma}
In this subsection, we consider homomorphisms from generalized matrix algebras to matrix algebras. This will be used to relate the existence of certain matrix-valued deformations to properties of $R$.

\begin{lem}
\label{lem:GMA lemma new}
Let $A$ be a commutative ring, let 
\[
E_A= \ttmat{A}{B_A}{C_A}{A}
\]
be an $A$-GMA, and let $\Phi: B_A \times C_A \to A$ be the $A$-bilinear map that induces the multiplication in $E_A$. Then there is a bijection between the set of $A$-GMA homomorphisms $E_A \to M_2(A)$ and the set of pairs $(\varphi_b: B_A \to A, \varphi_c: C_A \to A)$ of $A$-module homomorphisms satisfying $\Phi(b,c)=\varphi_b(b) \varphi_c(c)$ for all $b \in B_A$ and $c \in C_A$.
\end{lem}
\begin{proof}
The map sends an $A$-GMA homomorphism $\Psi:E_A \to M_2(A)$ to $(\Psi|_{B_A}, \Psi|_{C_A})$. The fact that $\Psi$ is an $A$-GMA homomorphism implies that $\Phi(b,c)=\Psi(b) \Psi(c)$ for all $b \in B_A$ and $c \in C_A$. Conversely, given a pair $(\varphi_b, \varphi_c)$, we can define a map of $A$-modules by
\[
E_A \xrightarrow{\sm{1}{\varphi_b}{\varphi_c}{1}} M_2(A),
\]
and we see that it is a homomorphism of $A$-GMAs if and only if $\Phi(b,c)=\varphi_b(b) \varphi_c(c)$ for all $b \in B_A$ and $c \in C_A$.
\end{proof}

For the universal Cayley--Hamilton $R$-algebra $E$ defined in \S\ref{subsec:construction of R}, the following lemma shows that deformations of $\rho_0$ give rise to GMA homomorphisms from $E$ to a matrix algebra.
\begin{lem}
	\label{lem:GMA lemma 2}
	Let $\rho_n: G_{\Q,S} \ra \GL_2(\zp{r}[\epsilon_n])$ be a good deformation of $\rho_0$ modulo $p^r$. Then there exists a GMA structure on the Cayley--Hamilton $R$-algebra $E$ such that the Cayley--Hamilton representation $\rho_n : E \ra M_2(\zp{r}[\epsilon_n])$ induced by $\rho_n$ is a homomorphism of GMAs.
\end{lem}
\begin{proof}
The Cayley--Hamilton representation $\rho_n$ exists by virtue of the universal property of $E$: $\rho_n$ is finite-flat and induces a pseudorepresentation with the properties enumerated in Proposition \ref{prop:universal property of R}. For the rest, see \cite[Thm.\ 3.2.2]{WWE4}, for example.
\end{proof}

\subsection{Criteria for goodness and very goodness}
Let $1 \le s \le t$, and let $n \ge 1$ be an integer. Fix an $n$-th order \emph{good} deformation $\rho_n$ of $\rho_1$ modulo $p^s$. 

\begin{lem}
\label{lem:criteria for goodness}
Let $\rho_n^c$ and $\rho_n^b$ be $n$-th order deformations of $\rho_1^c$ and $\rho_n^b$ modulo $p^s$, respectively, that are adapted to $\rho_n$. Let $c_{n+1} \in C^1(\zp{s}(-1))$ and $b_{n+1} \in C^1(\zp{s}(1))$ be the associated cochains. Then
\begin{enumerate}
\item $\rho_n^c$ is good if and only if $c_{n+1}|_p=0$.
\item $\rho_n^b$ is good if and only if $b_{n+1}$ makes $\sum_{i=0}^n b_{i+1}|_p \epsilon^i$ define a finite-flat extension of $\chi_d(\rho_n)$ by $\chi_a(\rho_n)$.
\end{enumerate}
\end{lem}

\begin{proof}
By \eqref{eq:psi-ad}, we see that $\rho^c_n$ and $\rho^b_n$ are good if and only if $\rho^c_n\vert_p$ and $\rho^b_n\vert_p$, respectively, are finite-flat and upper-triangular. This shows the ``only if'' part of (1), and, since $\rho^b_n\vert_p$ is automatically upper-triangular, (2) is immediate. It remains to show that $\rho_n^c|_p$ is finite-flat if $c_{n+1}|_p=0$.

If $c_{n+1}|_p=0$, then $\rho_n^c|_p$ is the extension of $\chi_d(\rho_n)$ by $\chi_a(\rho_n)$ defined by $\sum_{i=1}^{n-1} b_i \epsilon^{i+1}$. Its class in $\Ext_{\zp{s}[\epsilon_n][G_p]}(\chi_d(\rho_n), \chi_a(\rho_n))$ is the scalar multiple by $\epsilon$ of the class of $\rho_n|_p$ (which is finite-flat). Since scalar multiplication on extensions preserves finite-flatness (see Remark \ref{rem:Baer scalar}), $\rho_n^c|_p$ is finite-flat. 
\end{proof}

\begin{lem}
\label{lem:criteria for very goodness}
Let $1 \le s \le t$, and let $n \ge 1$ be an integer. Suppose that $\rho_{n+1}$ is an $(n+1)$-st order deformation of $\rho_1$ modulo $p^s$. Write $\rho_{n+1}$ as 
\[
\rho_{n+1} =  \left(1+ \sum_{i=1}^{n+1} \ttmat{a_i}{b_i}{c_i}{d_i} \epsilon^i\right) \rho_0,
\]
so $a_1=a$, $b_1=b$, $c_1=c$ and $d_1=-a$, and $a_i,d_i \in C^1(\zp{s})$, $b_i \in C^1(\zp{s}(1))$, $c_i \in C^1(\zp{s}(-1))$.  Suppose also that $\rho_n := \rho_{n+1} \otimes_{\zp{s}[\epsilon_{n+1}]} \zp{s}[\epsilon_{n}]$ is very good. Then $\rho_{n+1}$ is very good if and only if these three conditions hold: 
\begin{enumerate}
\item[(i)] $c_{n+1}|_p=0$ and $a_{n+1}|_{I_p} = d_{n+1}|_{I_p}=0$, 
\item[(ii)] $\displaystyle \sum_{i=1}^{n+1} b_{i}|_p \epsilon^i$ defines a finite-flat extension of $\chi_d(\rho_{n+1})$ by $\chi_a(\rho_{n+1})$,
\item[(iii)] $a_{n+1}|_{I_N}=b_{n+1}|_{I_N}=c_{n+1}|_{I_N}=d_{n+1}|_{I_N}=0$.
\end{enumerate}
\end{lem}

\begin{proof}
First assume that $\rho_{n+1}$ is very good. Then (iii) is clear from the definition of very good. Since $\rho_{n+1}$ is good, we have that $\rho_{n+1}\vert_p$ is finite-flat and upper-triangular. This implies that $c_{n+1}|_p=0$. Because $\chi_a(\rho_{n+1})$ (resp.\ $\chi_d(\rho_{n+1})$) is a finite-flat deformation of $\zp{s}(1)$ (resp.\ $\zp{s}$), which is equivalent to being an unramified deformation, we have (i). Then (ii) follows from $\rho_{n+1}\vert_p$ being finite-flat. 

Conversely, suppose that $\rho_{n+1}$ satisfies (i), (ii), and (iii). We have just explained why $\rho_{n+1}|_p$ is finite-flat and upper-triangular. Also $\rho_{n+1}$ is clearly mildly ramified at $N$ and $\tr(\rho_{n+1}|_{I_N})=2$.  Because $\rho_n$ is good, $\det(\rho_{n+1}) = \kcyc(1 + \epsilon^{n+1}\delta)$ for some $\delta \in Z^1(\zp{s})$. We observe that (i) implies $\delta\vert_{I_p} = 0$. Since $\rho_{n+1}$ is mildly ramified at $N$, we see that $\det(\rho_{n+1})\vert_{I_N}$ has the form $1-(a^2 + bc)\vert_{I_N}\epsilon^2 = 1$ by the normalizations of \eqref{eq:normalization of a,b,c}. Thus $\delta$ is unramified everywhere and consequently equals $0$. 
\end{proof}

\subsection{Residually lower-triangular deformations} We study deformations of $\rho_0^c$.

\begin{lem}
\label{lem:lower-triangular equivs}
Let $1\le r \le s \le t$, and let $n \ge 1$ be an integer. Suppose that $\rho_n$ is an $n$-th order deformation of $\rho_1$ modulo $p^s$, and suppose that $\rho_n$ is very good. Let $\varphi: R \onto \zp{s}[\epsilon_n]$ be the corresponding homomorphism, and let $D$ be the corresponding defining system for the Massey power $\dia{M}^{n+1}$. Write $\rho_{n,r} = \rho_n \otimes_{\zp{s}[\epsilon_n]}\zp{r}[\epsilon_n]$. Then the following are equivalent:
\begin{enumerate}
\item There is a surjective homomorphism $\varphi': R \onto \zp{r}[\epsilon_{n+1}]$ such that the following diagram commutes 
\[
\xymatrix{
R \ar[r]^-{\varphi'}  \ar[d]_-{\varphi} & \zp{r}[\epsilon_{n+1}] \ar[d] \\
\zp{s}[\epsilon_n] \ar[r] & \zp{r}[\epsilon_n],
}\]
where the unlabeled arrows are the quotient maps.
\item The $\zp{r}[\epsilon_n]$-module $C \otimes_{R,\varphi} \zp{r}[\epsilon_n]$ is free of rank $1$.
\item There is an $n$-th order deformation of $\rho_0^c$ modulo $p^r$ that is adapted to $\rho_{n,r}$.
\item There is an $n$-th order deformation of $\rho_0^c$ modulo $p^r$ that is adapted to $\rho_{n,r}$ and is good.
\item The Massey relation for $\dia{M}_D^{n+1}$ holds in the $(2,1)$-coordinate modulo $p^r$.
\end{enumerate}
\end{lem}
\begin{rem}
For $n=1$, we can take $\rho_n = \rho_1 \otimes_{\zp{t}[\epsilon_1]}\zp{s}[\epsilon_1]$, and the Massey relation for $\dia{M}_D^{2} = M \cup M$ in the $(2,1)$-coordinate is simply $c \cup a - a \cup c=0$. Using the skew-commutativity of the cup product, we see that this relation holds if and only if $a\cup c= 0$.
\end{rem}
\begin{proof}
In the proof, it will be helpful to induce an alternate characterization of $(2)$. First note that (2) only depends on the $R$-module structure of $C$, which is independent of the choice of GMA-structure on $E$. We apply Lemma \ref{lem:GMA lemma 2} to $\rho_{n,r}$, which defines a GMA structure on $E$, and we will write $E \risom\sm{R}{B}{C}{R}$ for this choice of GMA structure. We write $\Phi:B \times C \to R$ for the $R$-bilinear map coming from the multiplication in $E$. Write $C_{n,r} := C \otimes_{R,\varphi} \zp{r}[\epsilon_n]$ and $B_{n,r} := B \otimes_{R,\varphi} \zp{r}[\epsilon_n]$.

By Lemmas \ref{lem:GMA lemma new} and \ref{lem:GMA lemma 2}, the deformation $\rho_{n,r}$ defines $\zp{r}[\epsilon_n]$-module homomorphisms $\varphi_b: B_{n,r} \to  \zp{r}[\epsilon_n]$ and  $\varphi_c: C_{n,r} \to  \zp{r}[\epsilon_n]$, both having image $\epsilon\zp{r}[\epsilon_n]$, and satisfying $\Phi(b,c) = \varphi_b(b) \varphi_c(c)$ for all $b \in B_{n,r}$ and $c \in C_{n,r}$.

With this notation, we can see that (2) is equivalent to the following condition:
\begin{itemize}
\item[(2')] There is a homomorphism $\tilde{\varphi}_c: C_{n,r} \to  \zp{r}[\epsilon_n]$ of $\zp{r}[\epsilon_n]$-modules such that $\epsilon \cdot \tilde{\varphi}_c = \varphi_c$.
\end{itemize}
Indeed, if $C_{n,r}$ is free, then it has a generator $z$ such that $\varphi_c(z)=\epsilon$, and we can define $\tilde{\varphi}_c$ by $\tilde{\varphi}_c(z)=1$. Conversely, any such $\tilde{\varphi}_c$ must be surjective, so the fact that $C_{n,r}$ is cyclic implies that it must be free.

$(1) \Longleftrightarrow (2)$: Choose a generator $x \in R$ such that $\varphi(x)=\epsilon$. By Corollaries \ref{cor:structure of R} and \ref{cor:structure of C}, there is an isomorphism
\[
\Z_p[x]/(xg(x)) \isoto R
\]
for some distinguished monic polynomial $g(x) =\sum_{i=0}^{e} \beta_i x^{i}$, and an isomorphism of $R$-modules $C \simeq \Z_p[x]/(g(x))$. The existence of $\varphi$ implies that $v_p(\beta_i) \ge s$ for $i < n$. We see that
\[
C \otimes_{R,\varphi} \zp{r}[\epsilon_n] \simeq \frac{\zp{r}[\epsilon]}{(g(\epsilon),\epsilon^{n+1})} = \frac{\zp{r}[\epsilon]}{(\beta_n\epsilon^n,\epsilon^{n+1})}
\]
as an $\zp{r}[\epsilon_n]$-module. Hence $(2)$ is equivalent to $v_p(\beta_n) \ge r$, which is equivalent to $(1)$ by Lemma \ref{lem:NPcontrol}. 

$(2') \implies (3)$: Let $\tilde{\varphi}_c$ be as in (2'). Then we see that $\Phi(b,c) =(\epsilon \cdot \varphi_b)(b) \tilde{\varphi}_c(c)$ for all $b \in B_{n,r}$ and $c \in C_{n,r}$, so Lemma \ref{lem:GMA lemma new} implies that the pair $(\epsilon \cdot \varphi_b, \tilde{\varphi}_c)$ induces a GMA homomorphism $E \to M_2(\zp{r}[\epsilon_n])$. Pre-composing with $\rho^u : G_{\Q,S} \ra E^\times$, we obtain a representation $\rho_n^c: G_{\Q,S} \to \GL_2(\zp{r}[\epsilon_n])$ satisfying the conditions in $(3)$.

$(3) \implies (4)$: Let $\rho_n^c$ be a deformation of $\rho_0^c$ adapted to $\rho_{n,r}$, and let $c_{n+1} \in C^1(\zp{r}(-1))$ be the associated cochain. We have
\[
dc_{n+1} = \sum_{i=1}^n c_i \smile a_{n+1-i} + d_i \smile c_{n+1-i}.
\]
Since $\rho_n$ is very good, we have $c_i|_p=0$ for $i\le n$, and so $c_{n+1} \in Z^1_p(\zp{r}(-1))$. By Proposition \ref{prop:H^2_f in H^2}, the map
\[
H^1(\zp{r}(-1)) \onto H^1_p(\zp{r}(-1))
\]
is surjective. Then we can subtract an element of $Z^1(\zp{r}(-1))$ from $c_{n+1}$ to obtain an element $c_{n+1}'$ such that $dc_{n+1}=dc_{n+1}'$ and such that $c_{n+1}'|_{p}=0$. We let ${\rho_n^c}'$ be the deformation of $\rho_0^c$ that is adapted to $\rho_{n,r}$ associated to $c_{n+1}'$. By Lemma \ref{lem:criteria for goodness}, ${\rho_n^c}'$ is good.

$(4) \implies (2')$: Let $\rho_n^c$ be a good deformation of $\rho_0^c$ that is adapted to $\rho_{n,r}$. This induces an $\zp{r}[\epsilon_n]$-algebra homomorphism
\[
E \otimes_{R,\varphi} \zp{r}[\epsilon_n] \to M_2(\zp{r}[\epsilon_n]).
\]
By Lemma \ref{lem:GMA lemma new}, this defines a homomorphism $\tilde{\varphi}_c:C_{n,r} \to \zp{r}[\epsilon_n]$ that, from the definition of adapted, satisfies the condition of (2').

$(5) \Longleftrightarrow (3)$: This is Lemma \ref{lem:deformations from relations between coord masseys}.
\end{proof}

\subsection{Residually upper-triangular deformations}  We consider deformations of $\rho_0^b$.

\begin{lem}
\label{lem:upper-triangular equivs}
Let $1\le r \le s \le t$, and let $n \ge 1$ be an integer. Suppose that $\rho_n$ is an $n$-th order deformation of $\rho_1$ modulo $p^s$, and suppose that $\rho_n$ is very good. Let $\varphi: R \onto \zp{s}[\epsilon_n]$ be the corresponding homomorphism, and let $D$ be the corresponding defining system for the Massey power $\dia{M}^{n+1}$. Then the following are \textbf{true}:
\begin{enumerate}
\item The $\zp{s}[\epsilon_n]$-module $B \otimes_{R,\varphi} \zp{s}[\epsilon_n]$ is free of rank $1$.
\item There is a $n$-th order deformation of $\rho_0^b$ modulo $p^s$ that is adapted to $\rho_{n,r}$ and is good.
\item The Massey relation for $\dia{M}_D^{n+1}$ holds in the $(1,2)$-coordinate modulo $p^s$.
\end{enumerate}
\end{lem}
\begin{rem}
\label{rem: a cup b is zero}
For $n=1$ and $r=s=t$, we can take $\rho_n = \rho_1$, and the Massey relation for $\dia{M}_D^{2} = M \cup M$ in the $(1,2)$-coordinate is simply $b \cup a - a \cup b=0$. Using the skew-commutativity of the cup product, the lemma implies that $a \cup b=0$.
\end{rem}
\begin{proof}
By Lemma \ref{cor:structure of B and C}, $B$ is a free $R$-module of rank 1. Then (1) is clear, and (2) implies (3) by Lemma \ref{lem:deformations from relations between coord masseys}. To show (2), we follow the proof of Lemma \ref{lem:lower-triangular equivs}, and will use the same notation of $\varphi_b$ and $\varphi_c$ introduced there. As in that proof, (1) implies that there is a homomorphism $\tilde{\varphi}_b: B \to \zp{r}[\epsilon_n]$ such that $\epsilon \cdot \tilde{\varphi}_b= \varphi_b$. The data $(\tilde{\varphi}_b, \epsilon \cdot \varphi_c)$ give a GMA homomorphism $E \otimes_{R,\varphi} \zp{s}[\epsilon_n] \to M_2(\zp{s}[\epsilon_n])$, and this gives a finite-flat representation $\rho_n^b$ that is adapted to $\rho_{n,r}$. By Lemma \ref{lem:criteria for goodness}, $\rho_n^b$ is good.
\end{proof}

\subsection{Residually diagonal deformations} We consider deformations of $\rho_1$.

\begin{lem}
\label{lem:res diag defs}
Let $1\le r \le s \le t$, and let $n \ge 1$ be an integer. Suppose that $\rho_n$ is a very good $n$-th order deformation of $\rho_1$ modulo $p^s$. Let $D$ be the corresponding defining system for the Massey power $\dia{M}^{n+1}$. Then the following are equivalent:
\begin{enumerate}
\item There is an $(n+1)$-st order deformation of $\rho_n$ modulo $p^r$.
\item There is an $(n+1)$-st order deformation of $\rho_n$ modulo $p^r$ that is very good.
\item The Massey power $\dia{M}_D^{n+1}$ vanishes in $H^2(\End(\zp{r}(1)\oplus \zp{r}))$.
\end{enumerate}
\end{lem}
\begin{rem}
\label{rem:<M>^2 iff both cups}
For $n=1$, we can take $\rho_n = \rho_1 \otimes_{\zp{t}[\epsilon_1]}\zp{s}[\epsilon_1]$, and the Massey power $\dia{M}_D^{n+1} = \dia{M}_D^{2}$ is just the cup product 
\[
M \cup M = \ttmat{a \cup a + b \cup c}{a\cup b - b \cup a}{c \cup a - a \cup c}{a \cup a + c \cup b}.
\]
Using the skew-commutativity of the cup product and the fact that $a \cup b=0$ (Remark \ref{rem: a cup b is zero}), we see that $\dia{M}_D^{2}=0$ if and only if $a\cup c$ and  $b \cup c$ are both zero.
\end{rem}
\begin{proof}
First note that (1) is equivalent to (3) by Lemma \ref{lem:massey powers and deformations}, and clearly (2) implies (1). It remains to show that (1) implies (2). Fix a deformation $\rho_{n+1}: G_{\Q,S} \to \GL_2(\zp{r}[\epsilon_{n+1}])$ of $\rho_n$ and write
\[
\rho_{n+1} = \left(1+ \sum_{i=1}^{n+1} \ttmat{a_i}{b_i}{c_i}{d_i} \epsilon^i\right) \rho_0 
\]
with $a_1=a$, $b_1=b$, $c_1=c$ and $d_1=-a$, and $a_i,d_i \in C^1(\zp{r})$, $b_i \in C^1(\zp{r}(1))$, $c_i \in C^1(\zp{r}(-1))$. We will construct another deformation $\rho_{n+1}'$ of $\rho_n$ such that $\rho_{n+1}'$ satisfies the conditions (i), (ii), (iii) of Lemma \ref{lem:criteria for very goodness}. 

Since $\rho_{n+1}$ is a representation, we have
\begin{equation}
\tag{$*$}
da_{n+1} = \sum_{i=1}^{n} (a_i \smile a_{n+1-i} + b_i \smile c_{n+1-i})
\end{equation}
in $C^2(\zp{r})$. Since $\rho_n$ is very good, we have $a_i|_{I_p}=c_i|_{I_p}=0$ for $i=1,\dots,n$, so we see that $a_{n+1}|_{I_p}$ is a cocycle. Similarly, we can see that $d_{n+1}|_{I_p}$ and $c_{n+1}|_p$ are cocycles. As in the proof of $(3) \Rightarrow (4)$ in Lemma \ref{lem:lower-triangular equivs}, we can use Proposition \ref{prop:H^2_f in H^2} to show that, by subtracting a global cocycle, we can obtain elements $a_{n+1}',d_{n+1}'$ and $c_{n+1}'$ satisfying $a_{n+1}'|_{I_p}=d_{n+1}'|_{I_p}=c_{n+1}'|_p=0$.

Define $\chi_a' :=\kcyc(1+ \sum_{i=1}^n a_i|_p \epsilon^i + a_{n+1}'|_p\epsilon^{n+1})$ and $\chi_d' :=1+ \sum_{i=1}^n d_i|_p \epsilon^i + d_{n+1}'|_p\epsilon^{n+1})$, and note that they are unramified deformations of $\kcyc|_p$ and $1|_p$, respectively.

Now, applying Lemma \ref{lem:upper-triangular equivs}, we can find an $n$-th order deformation $\rho_n^b$ of $\rho_0^b$ modulo $p^s$ that is adapted to $\rho_{n,r}$ and is good. Let $b_{n+1}' \in C^1(\zp{r}(1))$ be the cochain associated to $\rho_n^b$, and note that $db_{n+1}'=db_{n+1}$. Since $\rho_n^b$ is good, $\rho_n^b|_p$ is a finite-flat extension of $\chi_d(\rho_n)$ by $\chi_a(\rho_n)$. Following Appendix \ref{sec:algebra appendix}, we see that $\sum_{i=1}^{n} b_{i}|_p\epsilon^{i}+b_{n+1}'|_p\epsilon^{n+1}$ is a finite-flat extension of $\chi_d'$ by $\chi_a'$, since it is obtained from $\rho_n^b|_p$ by first pulling back by $\chi_d' \onto \chi_d(\rho_n)$ and then pushing out along $\chi_a(\rho_n) \cong \epsilon \chi_a' \rinj \chi_a'$. 
	
We have now constructed cochains $a_{n+1}', b_{n+1}', c_{n+1}'$ and $d_{n+1}'$ such that $da_{n+1}' =da_{n+1}$, $db_{n+1}' =db_{n+1}$, $dc_{n+1}' =dc_{n+1}$, and $dd_{n+1}' =dd_{n+1}$, and satisfying conditions (i) and (ii) of Lemma \ref{lem:criteria for very goodness}. To show (iii), we use:
\begin{claim}  $a_{n+1}'|_{I_N}, b_{n+1}'|_{I_N},  c_{n+1}'|_{I_N},d_{n+1}'|_{I_N}$ are cocycles.
\end{claim}
\begin{proof}
Indeed, this is clear from (the analog of) $(*)$ if $n>1$, using the fact that $\rho_n$ is very good. For $n=1$, let $\sigma=x \gamma^i$ and $\tau=y \gamma^j$ with $x,y \in I_N^{\mathrm{non}\text{-}p}$ and $i,j \in \Z$. By our normalizations (see \eqref{eq:normalization of a,b,c}) we have $a(\sigma)=c(\sigma)=\bar{i}$ and $b(\sigma)=-\bar{i}$ (where $\bar{i} \in \zp{r}$ is the reduction of $i$). Then, by $(*)$, we have
\[
da_2'(\sigma,\tau) = a(\sigma)a(\tau)+b(\sigma)c(\tau) = \bar{i}\bar{j}-\bar{i}\bar{j}=0.
\]
Since pairs $(\sigma,\tau)$ of this type form a dense subset of $I_N \times I_N$, and since $da_2'$ is continuous, we see that we see that $a_2'|_{I_N}$ is a cocycle. The proof for $b,c$ and $d$ is similar.
\end{proof}
Subtracting a multiple of $a$ from $a_{n+1}'$, we can arrange so that $a_{n+1}'(\gamma)=0$ while maintaining the properties that $da_{n+1}'=da_{n+1}$, that $a_{n+1}'|_{I_p}=0$, and that $a_{n+1}'|_{I_N}$ is a cocycle. This implies that $a_{n+1}'|_{I_N}=0$, since $\gamma$ is a generator of the pro-$p$ part of $I_N$. 

Similarly, we can alter $b_{n+1}'$,  $c_{n+1}'$, and $d_{n+1}'$ so that they vanish on restriction to $I_N$, without changing their properties on restriction to $G_{\Q_p}$.

Now we define $\rho_{n+1}'$ to be
\[
\rho_{n+1}' =  \left(1+ \sum_{i=1}^{n} \ttmat{a_i}{b_i}{c_i}{d_i} \epsilon^i + \ttmat{a'_{n+1}}{b'_{n+1}}{c'_{n+1}}{d'_{n+1}} \epsilon^{n+1}\right) \rho_0.
\]
Since $da_{n+1}' =da_{n+1}$, $db_{n+1}' =db_{n+1}$, $dc_{n+1}' =dc_{n+1}$, and $dd_{n+1}' =dd_{n+1}$, we see that $\rho_{n+1}'$ is a deformation of $\rho_n$. By construction, we see that $\rho_{n+1}'$ satisfies the conditions (i), (ii), (iii) of Lemma \ref{lem:criteria for very goodness}, and so $\rho_{n+1}'$ is very good.
\end{proof}

\begin{rem}
\label{rem:global flat massey}
Another way to think of this proposition is that, morally speaking, the cup products and Massey products in this paper ``should be" valued in the global finite-flat cohomology group $H^2_\fl$ explained in \S\ref{sssec:global flat}. Then, for example, the unconditional vanishing of the Massey relation in the $(1,2)$-coordinate would follow from the fact that $H^2_\fl(\zp{t}(1))=0$ (see Proposition \ref{prop:computation of B}). 

More generally, the pattern of the arguments that relate Massey product vanishing to the existence of a global finite-flat representation has been
\begin{enumerate}
\item Choose a global cochain whose coboundary is the Massey product
\item Modify it by a global cocycle (so that its coboundary does not change) so that it is a finite-flat cocycle upon restriction to $G_p$. 
\end{enumerate}
We have developed a theory of cup products and Massey products in global finite-flat cohomology that would simplify such arguments. The same simplification can be achieved using a formulation in terms of $A_\infty$-operations, which induces a choice of Massey products compatible with this theory; for this, see \cite[Thm.\ 3.4.1 and \S12]{CarlAinf}. 

Since the relevant $H^1_\fl$ groups are $1$-dimensional in each coordinate (spanned by $a$, $b$, $c$, and $a$, respectively), the resulting Massey products are unambiguously defined (i.e.\ various choices of defining systems result in the same Massey product) and we would not need to consider specific defining systems. However, this theory would take several pages to properly develop. More importantly, it is not necessary for our arguments because Proposition \ref{prop:H^2_f in H^2} implies that it suffices to test a global finite-flat Massey condition (in $H^2_\fl$) as a global Massey condition (in $H^2$). An inductive procedure produces appropriate defining systems.
\end{rem}

\part{Massey products and arithmetic}

In this part, we study some analytic and algebraic number-theoretic interpretations of the vanishing of cup products and Massey products. We prove that some of the coordinate Massey relations considered in the previous part are equivalent to each other. Combining these equivalence with the results of the previous part, we prove our main result, interpreting the rank and Newton polygon of $\bT$ in terms of Massey products. The results of this part also explain how to deduce the main results of Calegari--Emerton \cite{CE2005} (for $p > 3$) and Merel \cite{merel1996} from our Theorem \ref{thm:cup products and rank 1}. 

For the entirety of Part 3, we continue to fix $t=v_p(N-1) \geq 1$, and also fix an integer $s$ with $1\le s \le t$ . We also let $\Delta=\Gal(\Q(\zeta_{p^s})/\Q) \cong (\Z/p^s\Z)^\times$.

\section{Cup products and arithmetic}
\label{subsec:cup products and Galois theory} 

In this section, we deduce a generalization of the main result of Calegari--Emerton \cite{CE2005}, relating $e = \mathrm{rank}_{\Z_p}(\bT^0)$ to certain class groups.

\subsection{Cup products and Galois theory} We let $\zeta_N^{(p^s)} \in \Q(\zeta_N)$ denote an element such that $[\Q(\zeta_N^{(p^s)}):\Q]=p^s$. Note that $\Q(\zeta_N^{(p^s)})$ is the fixed field of the kernel of the homomorphism $a:G_{\Q,S}\to\Z/p^t\Z \onto \Z/p^s\Z$.
 
\begin{prop}
\label{prop:cup gives class}

\begin{enumerate}
\item If $b \cup c =0$ in $H^2(\Z/p^s\Z)$, then $\Cl(\Q(N^{1/p^{s}}))[p^\infty]$ admits $\Z/p^{s}\Z \times \Z/p^{s}\Z$ as a quotient. 
\item If $a \cup c =0$  in $H^2(\Z/p^s\Z(-1))$, then $(\Cl(\Q(\zeta_N^{(p^s)},\zeta_{p^s}))[p^\infty] \otimes \Z_p(1))^\Delta$ admits $\Z/p^{s}\Z \times \Z/p^{s}\Z$ as a quotient. 
\end{enumerate}
\end{prop}

\begin{proof}
Replace $a,b,c$ with their reductions modulo $p^{s}$. 

(1) Let $F \in C^1(\Z/p^s\Z)$ be a cochain satisfying $dF = b \smile c$. Since $c\vert_p=0$, we have that $F\vert_{I_p}$ is a cocycle. Just as in the proof of Lemma \ref{lem:lower-triangular equivs}, we can subtract an element of $Z^1(\Z/p^s\Z)$ from $F$ to ensure that $F\vert_{I_p}=0$. Moreover, since $dF \ne 0$, we have $F \not \in Z^1(\Z/p^s\Z)$.

Consider the function $\nu: G_{\Q,S} \to \GL_3(\Z/p^s\Z)$ given by
\[
 \sigma \mapsto 
\left( \begin{array}{ccc}
1 & c & F \\
0 & \kcyc & b \\
0 & 0 & 1
\end{array} \right).
\] 
Since $dF=b \smile c$, we see that $\nu$ is a homomorphism. Since $\kcyc$ is unramified at $N$, the image of $\nu\vert_{I_N}$ is unipotent. Since the unipotent radical of the upper-triangular Borel in $\GL_3(\Z/p^s\Z)$ has exponent $p^{s}$, and since $I_N^{\mathrm{pro}\text{-}p}$ is pro-cyclic, we see that the image of $\nu|_{I_N}$ is a cyclic group. Because $b\vert_{I_N}, c\vert_{I_N}$ induce surjective homomorphisms $I_N \rsurj \Z/p^s\Z$, this cyclic group has order $p^{s}$. Since $\Q(N^{1/p^{s}})/\Q$ is totally ramified at $N$, this implies that the restriction of $\nu$ to $G_{\Q(N^{1/p^{s}})}$ is unramified at $N$. 

At the start of \S\ref{subsec:matrix reps}, we chose $b$ to be a constant multiple of the Kummer cocycle corresponding to the chosen root $N^{1/p^t}$ of $N$. Since we have now reduced $b$ modulo $p^{s}$, $b : G_{\Q,S} \to \Z/p^s\Z(1)$ is given by 
\[
\sigma \mapsto \frac{\sigma(N^{1/p^{s}})}{N^{1/p^{s}}}.
\]
In particular, $b\vert_{G_{\Q(N^{1/p^{s}})}}=0$. This implies that $F|_{G_{\Q(N^{1/p^{s}})}} \in  Z^1(\Q(N^{1/p^{s}}), \Z/p^s\Z)$, and so it corresponds to a cyclic degree $p^{s}$ extension $K_F/\Q(N^{1/p})$ that is unramified outside $Np$. Since $F$ is chosen to be unramified at $p$ and $\nu|_{G_{\Q(N^{1/p^s})}}$ is unramified at $N$, we see that $K_F$ is actually unramified everywhere. By class field theory, $K_F$ is cut out by a surjection $\Cl(\Q(N^{1/p^s})) \rsurj \Z/p^s\Z$.

Finally, since the image of $F$ in $C^1(\Z/p^r\Z)$ is not a cocycle for any $1 \leq r \leq s$, we see that $K_F$ is linearly disjoint from the genus field of $\Q(N^{1/p^{s}})$, which is $\Q(\zeta_N^{(p^{s})},N^{1/p^{s}})$. Hence the two unramified degree $p^{s}$ extensions of $\Q(N^{1/p^{s}})$ given by $K_F$ and $\Q(\zeta_N^{(p^{s})},N^{1/p^{s}})$ correspond to linearly independent elements of $\Cl(\Q(N^{1/p^{s}}))[p^{s}]$ of order $p^{s}$.

(2) Similar.
\end{proof}

\section{Cup products and Merel's number} 
\label{sec:without modular forms}

 Let $G=(\Z/N\Z)^\times$, recall $1 \leq s \leq t = v_p(N-1)$, and let $I_G$ be the augmentation ideal in $(\Z/p^s\Z)[G]$. This section concerns Merel's number $\prod_{i=1}^{\frac{N-1}{2}} i^i$ that appears in Merel's Theorem \ref{thm:merel}, and its relation to cup products and to the ``zeta element'' 
\[
\zeta := \sum_{i \in (\Z/N\Z)^\times}B_2(\lfloor i/N \rfloor)  [i] \in (\Z/p^s\Z)[G],
\]
where $B_2(x)=x^2-x+1/6$ is the second Bernoulli polynomial and where $\lfloor i/N \rfloor \in [0,1) \cap \frac{1}{N}\Z$ is the fractional part of $i/N$. We give a direct proof (not using deformation theory or modular forms) that the following four statements are equivalent:
\begin{enumerate}
\item $a \cup c=0$ in $H^2(\Z/p^s\Z(-1))$
\item $b \cup c= 0$ in $H^2(\Z/p^s\Z)$
\item Merel's number is a $p^{s}$-th power modulo $N$
\item $\zeta \in I_G^2$, i.e.\ $\ord_s \zeta \geq 2$. 
\end{enumerate}

Combining this equivalence for $s = 1$ with Theorem \ref{thm:cup products and rank 1}, which we will prove in \S\ref{sec:main result}, we have a new proof of Merel's Theorem \ref{thm:merel} without considering the geometry of modular Jacobians. For $s=1$, Theorem \ref{thm:cup merel equiv} (without the equivalent condition (1)) was known to Calegari and Emerton in unpublished work, but they did not know Theorem \ref{thm:cup products and rank 1}. We thank them for sharing their unpublished note with us. 

The proof of Theorem \ref{thm:cup merel equiv} is in two steps: first to relate the vanishing of cup products to the non-vanishing of a certain Selmer group, and second to use Stickelberger theory to relate the element $\zeta$ to the Selmer group, as in the proof of Herbrand's theorem. The second step has already been carried out beautifully in the paper \cite{lecouturier2018} of Lecouturier, which we use as a reference.

\subsection{Cup products and Selmer groups}
In this section, we give a simple proof that $a \cup c =0$ in $H^2(\Z/p^s\Z(-1))$ if and only if $b \cup c =0$ in $H^2(\Z/p^s\Z)$, and relate this vanishing to the non-vanishing of certain Selmer groups. The proof relies on considering the cohomology of $G_N$, so we start with some remarks about it. We note that since $p^s \mid (N-1)$, there is a primitive $p^s$-th root of unity $\zeta_{p^s}$ in $\Q_N$; we fix a choice of $\zeta_{p^s} \in \Q_N$, and this determines isomorphisms $\Z/p^s\Z \isoto \Z/p^s\Z(i)$ of $G_N$-modules for all $i$, which we will use as identifications. 

By Tate duality, we have a canonical isomorphism $H^2_N(\Z/p^s(1)) \cong \Z/p^s\Z$, which we use as an identification. By Kummer theory, we have $H^1_N(\Z/p^s(1)) \cong \Q_N^\times \otimes \Z/p^s\Z$, and we let $\cL_N \subset H^1_N(\Z/p^s\Z(1))$ be the free rank-$1$ $\Z/p^s\Z$-summand spanned by the image of $N$ under this isomorphism. Using our identification of $H^1_N(\Z/p^s\Z(1))$ and $H^1_N(\Z/p^s\Z)$, and the canonical basis of $H^2_N(\Z/p^s(1))$, we can think of Tate duality as providing a symplectic pairing on the free rank-$2$ $\Z/p^s\Z$-module $H^1_N(\Z/p^s\Z)$. 

Finally, note that since $B^1_N(\Z/p^t\Z)=0$, we have $Z^1_N(\Z/p^s\Z)=H^1_N(\Z/p^s\Z)$ and we can (and will) safely conflate cocycles with their cohomology classes.

\begin{lem}
\label{lem:H2 to H2N injective}
For $i=0,-1$, the map $H^2(\zp{s}(i)) \to H^2_N(\zp{s}(i))$ is an isomorphism.
\end{lem}
\begin{proof}
The map $H^2(\zp{s}(i)) \to H^2_{Np}(\zp{s}(i))$ is surjective because
\[
H^3_{(c)}(\zp{s}(i)) \cong H^0(\zp{s}(1-i))^*=0.
\]
Hence the map in question is surjective, so it is enough to show that the two groups have the same cardinality. We are reduced to showing that $\#H^2(\zp{s}(i))=p^s$.

Write $h^j(\Z/p^s\Z(i))=\#H^j(\zp{s}(i))$. By the global Euler characteristic formula (see, for example, \cite[Corollary 8.7.5, pg.~509]{NSW2008}), we have
\[
h^2(\Z/p^s\Z) = \frac{h^1(\Z/p^s\Z)}{h^0(\Z/p^s\Z)}, \quad  h^2(\Z/p^s\Z(-1)) = \frac{h^1(\Z/p^s\Z(-1))}{h^0(\Z/p^s\Z(-1)) \cdot p^s}.
\]
One sees easily that $h^1(\Z/p^s\Z)=p^{2s}$ and $h^0(\Z/p^s\Z)=p^s$, so $h^2(\Z/p^s\Z)=p^s$. We also have $h^0(\Z/p^s\Z(-1))=1$, and, by Lemma \ref{lem:surj of H1->H1_p in C coord}, $h^1(\Z/p^s\Z(-1))=p^{2s}$, so $h^2(\Z/p^s\Z(-1))=p^s$.
\end{proof}

\begin{prop}
\label{prop: cup commutativity}
For $i=0,1$, there is a commutative diagram
\[\xymatrix{
H^1(\Z/p^s\Z(i)) \times H^1(\Z/p^s\Z(-1)) \ar[d]^-{\vert _N} \ar[r]^-{\cup} & H^2(\Z/p^s\Z(i-1)) \ar[d]^\wr \\
H^1_N(\Z/p^s\Z(i)) \times H^1_N(\Z/p^s\Z(-1)) \ar[r]^-{\cup} & H^2_N(\Z/p^s\Z(i-1)).
}\]
In particular, for $x \in H^1(\Z/p^s\Z(i))$ and $y \in H^1(\Z/p^s\Z(-1))$, we have $x \cup y =0$ if and only if $x \vert_N \cup y \vert_N = 0$.
\end{prop}

\begin{proof}
The commutativity is clear, so this follows from the previous lemma.
\end{proof}

\begin{lem}
\label{lem:aN in cLN}
Under our identification $H^1_N(\Z/p^s\Z)=H^1_N(\Z/p^s\Z(1))$, both of the elements $a|_N$ and $b|_N$ are generators of $\cL_N \subset H^1_N(\Z/p^s\Z)$.
\end{lem}
\begin{proof}
We know that neither $a|_N$ nor $b|_N$ is divisible by $p$ because their value on $\gamma$ is $\pm 1$. So it will suffice to show that $a|_N, b|_N \in \cL_N$. 

We have $b|_N \in \cL_N$ by Proposition \ref{prop:computation of B}. Since the Tate pairing is symplectic, to show that $a|_N \in \cL_N$, it is enough to show that $a|_N \cup b|_N = 0$. But we know that $a \cup b=0$ by Lemma \ref{lem:upper-triangular equivs}, so we are done by the previous proposition.
\end{proof}

Let $H^1_\Sigma(\Z/p^s\Z(-1))$ denote the Selmer group
\[
H^1_\Sigma(\Z/p^s\Z(-1)) := \ker \left(H^1(\Z/p^s\Z(-1)) \to H^1_p(\Z/p^s\Z(-1)) \oplus H^1_N(\Z/p^s\Z)/\cL_N\right).
\]
Let $H^1_{\Sigma^\perp}(\Z/p^s\Z(2))$ denote the ``dual'' Selmer group
\[
H^1_{\Sigma^\perp}(\Z/p^s\Z(2)) := \ker \left(H^1(\Z/p^s\Z(2)) \to H^1_N(\Z/p^s\Z)/\cL_N \right).
\]

\begin{prop}
\label{prop:cup selmer equiv}
The following are equivalent:
\begin{enumerate}
\item $a \cup c = 0$ in $H^2(\Z/p^s\Z(-1))$
\item $b \cup c = 0$ in $H^2(\Z/p^s\Z)$
\item The image of $c|_N$ in $H^1_N(\Z/p^s\Z(-1))$ is in the subgroup $\cL_N$ 
\item $H^1_\Sigma(\Z/p^s\Z(-1)) \simeq \Z/p^{s}\Z$
\item $H^1_{\Sigma^\perp}(\Z/p^s\Z(2)) \simeq \Z/p^{s}\Z$
\item There is an element $x \in H^1(\Z/p^s\Z(2))$ with non-zero image in $H^1(\Z/p\Z(2))$ such that $x|_N \in \cL_N$.
\end{enumerate}
\end{prop}
\begin{rem}
\label{rem:<M>^2 and cups}
By Remark \ref{rem:<M>^2 iff both cups}, and using the notation from there, we see that all these items are also equivalent to $\dia{M}_D^2$ being zero in $H^2(\End(\Z/p^s\Z(1) \oplus \Z/p^s\Z))$. 
\end{rem}

\begin{proof}
The equivalence of $(1)$-$(3)$ follows from Proposition \ref{prop: cup commutativity}, Lemma \ref{lem:aN in cLN}, and the fact that the Tate pairing is symplectic. 

By the definition of $H^1_\Sigma(\Z/p^s\Z(-1))$, we have
\[
H^1_\Sigma(\Z/p^s\Z(-1))=\{x \in H^1_{(p)}(\Z/p^s\Z(-1)) \ | \ x|_N \in \cL_N\}.
\]
Since $H^1_{(p)}(\Z/p^s\Z(-1)) \simeq \Z/p^s\Z$ is generated by $c$, we see that (3) and (4) are equivalent.

By duality (Theorem \ref{thm:Poitou--Tate with constraints}), we have $H^1_\Sigma(\Z/p^s\Z(-1)) = H^2_{\Sigma^\perp}(\Z/p^s\Z(2))^*$, so (4) is equivalent to $H^2_{\Sigma^\perp}(\Z/p^s\Z(2)) \simeq \zp{s}$. Here $H^2_{\Sigma^\perp}(\Z/p^s\Z(2))$ fits into an exact sequence
\begin{gather*}
0 \lra H^1_{\Sigma^\perp}(\Z/p^s\Z(2)) \lra  H^1(\Z/p^s\Z(2)) \lra H^1_N(\Z/p^s\Z(2))/\cL_N \\ 
\lra H^2_{\Sigma^\perp}(\Z/p^s\Z(2)) \lra H^2(\Z/p^s\Z(2)) \lra H^2_{Np}(\Z/p^s\Z(2)) \lra 0.
\end{gather*}
As in the proof of Proposition \ref{prop:C computation}, the last map $H^2(\Z/p^s\Z(2)) \to H^2_{Np}(\Z/p^s\Z(2))$ is an isomorphism, so we have an exact sequence
\begin{gather*}
0 \lra H^1_{\Sigma^\perp}(\Z/p^s\Z(2)) \lra H^1(\Z/p^s\Z(2)) \lra \\ 
H^1_N(\Z/p^s\Z(2))/\cL_N \lra 
H^2_{\Sigma^\perp}(\Z/p^s\Z(2)) \lra 0.
\end{gather*}
By Lemma \ref{lem:global C comp}, we have $H^1(\Z/p^s\Z(2)) \cong H^2(\Z_p(2))[p^s] \simeq \Z/p^s\Z$, and we see that $x \in H^1(\Z/p^s\Z(2))$ is a generator if and only if its image in $H^1(\Z/p\Z(2))$ is non-zero. Since $H^1_N(\Z/p^s\Z(2))/\cL_N$ is also free $\Z/p^s\Z$-module of rank $1$ (see Lemma \ref{lem:local C comp}), this gives the equivalence of (4)-(6).
\end{proof}

In the end, we use condition (6) to relate cup products to Merel's number.

\subsection{Results of Lecouturier} 

We follow \cite{lecouturier2018}. Choose a surjective homomorphism $\log : \Z_N^\times \rsurj \Z/p^s\Z$; it factors through a map $\F_N^\times \rsurj \Z/p^s\Z$, which we also denote by $\log$. Note that Merel's number is a $p^{s}$-th power modulo $N$ if and only if $\sum_{i=1}^{\frac{N-1}{2}} i \log(i)=0$ in $\Z/p^s\Z$. 

\begin{lem} 
\label{lem:lecouturiers}
We have the equality
\[
\sum_{i=1}^{N-1} i^2 \log(i)=- \frac{4}{3} \sum_{i=1}^{\frac{N-1}{2}} i \log(i)
\]
in $\Z/p^s\Z$.
\end{lem}
\begin{proof}
This is \cite[Prop.\ 1.2]{lecouturier2018}. 
\end{proof}

Let $\Lambda: \Q_N^\times \otimes_\Z \Z/p^s\Z \to \Z/p^s\Z$ be defined by $\Lambda(N^kx\otimes \alpha) = \alpha \log(x)$ for $k \in \Z$, $x \in \Z_N^\times$ and $\alpha \in \Z/p^s\Z$.

Choose a prime ideal $\mathfrak{n} \subset \Z[\zeta_p]$ lying over $N$, so that the completion of $\Q(\zeta_p)$ at $\mathfrak{n}$ is $\Q_N$. For $x \in \Q(\zeta_p)$, let $x_\mathfrak{n} \in \Q_N$ denote the image in this completion. Finally, for a $\Z[\frac{1}{p-1}][\Gal(\Q(\zeta_p)/\Q)]$-module $M$ and a character $\chi:\Gal(\Q(\zeta_p)/\Q)\to \overline{\Q}^\times$, let $M_\chi$ denote the $\chi$-eigenspace.

\begin{prop}
\label{prop:lecouturiers}
There is an element $\mathcal{G} \in (\Z[1/Np,\zeta_p]^\times \otimes \Z_p)_{\omega^{-1}}$ such that
\[
\Lambda(\mathcal{G}_\mathfrak{n}) = - \frac{2}{3} \sum_{i=1}^{\frac{N-1}{2}} i \log(i).
\]
and whose natural image in $(\Z[1/Np,\zeta_p]^\times \otimes \F_p)_{\omega^{-1}}$ is non-trivial.
\end{prop}
\begin{proof}
Let $e_{\omega^{-1}} \cdot \mathcal{G} \in (\Z[1/N,\zeta_{Np}]^\times \otimes \Z_p)_{\omega^{-1}}$ be the element defined in \cite[\S3.3]{lecouturier2018}; it is a product of conjugates of Gauss sums. By \cite[Prop.\ 3.4]{lecouturier2018}, we actually have $e_{\omega^{-1}} \cdot \mathcal{G} \in (\Z[1/N,\zeta_{p}]^\times \otimes \Z_p)_{\omega^{-1}}$. 

Let $\beta = \sum_{i=1}^{p-1} \omega(i)i \in \Z_p$; as is well-known, $\beta = py$ for some $y \in \Z_p^\times$. Using the Gross--Koblitz formula, Lecouturier computes that
\begin{equation}
\label{eq: LGK}
(e_{\omega^{-1}} \cdot \mathcal{G})_\mathfrak{n} = (-N \otimes y)\cdot \left(\prod_{i=1}^{p-1} \Gamma_N \left(\frac{i}{p}\right)\otimes \omega(i) \right) \in \Q_N^\times \otimes \Z_p
\end{equation}
where $\Gamma_N$ is the $N$-adic Gamma function. In particular, $(e_{\omega^{-1}} \cdot \mathcal{G})_\mathfrak{n}$ is not in the image of $\Z_N^\times \otimes \Z_p\to \Q_N^\times \otimes \Z_p$, so its image in $(\Z[1/Np,\zeta_p]^\times \otimes \F_p)_{\omega^{-1}}$ is non-trivial.

Finally, the formula for $\Lambda((e_{\omega^{-1}} \cdot \mathcal{G})_\mathfrak{n})$ follows by \eqref{eq: LGK} and the formula
\[
\sum_{j=1}^{p-1} j \log \left( \Gamma_N \left(\frac{j}{p}\right) \right) =  - \frac{2}{3} \sum_{i=1}^{\frac{N-1}{2}} i \log(i)
\]
obtained by combining \cite[Lem.\ 4.3]{lecouturier2018} (with $\chi=\omega^{-1}$) with Lemma \ref{lem:lecouturiers}.
\end{proof}

\subsection{Merel's number and the zeta element}

Recall the zeta element $\zeta \in \Z/p^s\Z[G]$ defined at the start of this section. Let $\ord_s \zeta$ denoted the greatest integer $r$ such that $\zeta \in I_G^r$, where $I_G \subset \Z/p^s\Z[G]$ is the augmentation ideal. We now show, following Lecouturier, how Theorem \ref{thm:merel and zeta} follows from Merel's result (Theorem \ref{thm:merel}).

\begin{lem}
\label{lem:merel and zeta}
The following are equivalent:
\begin{enumerate}
\item Merel's number is a $p^s$-th power modulo $N$
\item $\ord_s \zeta \ge 2$.
\end{enumerate}
\end{lem}

\begin{proof}
We first note that $\zeta \in I_G$ because $s \leq t = v_p(N-1)$. Furthermore, we recall that there is an isomorphism
\[
 I_G/I_G^2 \isoto G \otimes_\Z \Z/p^s\Z \xrightarrow{\log} \Z/p^s\Z
\]
sending $[g] -1 \in I_G$ to $\log(g)$ for $g \in G$. Under this isomorphism, $\zeta \pmod{I_G^2}$ is sent to
\[
\sum_{i=1}^{N-1} (i^2-i+1/6)\log(i).
\]
One sees easily that $\sum_{i=1}^{N-1} \log(i)$ and $\sum_{i=1}^{N-1} i\log(i)$ are both $0$ in $\Z/p^s\Z$. Then $\ord_2 \zeta \ge 2$ if and only if $\sum_{i=1}^{N-1} i^2\log(i)=0$. The lemma now follows from Lemma \ref{lem:lecouturiers}.
\end{proof}

\subsection{Hochschild--Serre arguments} 
Let $\Delta=\Gal(\Q(\zeta_{p^s})/\Q) \cong (\Z/p^s\Z)^\times$ and let $\Delta^0 \cong (\Z/p\Z)^\times$ denote the prime-to-$p$ subgroup, so $\Delta\cong \Delta^0 \times \Delta_p$, where $\Delta_p \simeq \Z/p^{s-1}\Z$. 

\begin{lem}
Let $n$ be an integer such that $(p-1) \nmid n$. Then, for all $i\ge 0$, we have
\[
H^i(\Delta,\Z/p^s\Z(n))=0.
\]
\end{lem}
\begin{proof}
Since $\Delta^0 \subset \Delta$ is prime-to-$p$, we have
\[
H^i(\Delta,\Z/p^s\Z(n)) = H^i(\Delta_p, H^0(\Delta^0,\Z/p^s\Z(n))).
\]
Let $\zeta_{p-1} \in (\Z/p^s\Z)^\times$ be a primitive $(p-1)$-st root of unity. Then
\[
H^0(\Delta^0,\Z/p^s\Z(n)) = \{ x \in \Z/p^s\Z \ | \ \zeta_{p-1}^nx=x\}.
\]
Since $(p-1) \nmid n$, we see that $\zeta_{p-1}^n \not \equiv 1 \pmod{p}$. Hence $H^0(\Delta^0,\Z/p^s\Z(n))=0$.
\end{proof}

\begin{lem}
\label{lem: Hochs-Serre}
Let $n$ be an integer such that $(p-1) \nmid n$. Then 
\[
H^1(\Z/p^s\Z(n))= H^1(\Z[1/Np,\zeta_{p^s}],\Z/p^s\Z(n))^\Delta.
\]
\end{lem}
\begin{proof}
Note that $H^0(\Z[1/Np,\zeta_{p^s}],\Z/p^s\Z(n)) \cong \Z/p^s\Z(n)$ as $\Delta$-modules. Then, by the Hochschild--Serre spectral sequence, there is an exact sequence
\begin{align*}
H^1(\Delta,\Z/p^s\Z(n)) &  \lra H^1(\Z/p^s\Z(n)) \lra H^1(\Z[1/Np,\zeta_{p^s}],\Z/p^s\Z(n))^\Delta\\
 & \lra H^2(\Delta,\Z/p^s\Z(n)),
\end{align*}
so this follows from the previous lemma.
\end{proof}

\subsection{Merel's number and cup products}
We can now complete the proof of the following theorem.

\begin{thm}
\label{thm:cup merel equiv}
The following are equivalent:
\begin{enumerate}
\item $a \cup c =0$ in $H^2(\Z/p^s\Z(-1))$
\item $b \cup c=0$ in $H^2(\Z/p^s\Z)$
\item Merel's number is a $p^s$-th power modulo $N$
\item $\ord_s(\zeta) \ge 2$.
\end{enumerate}
\end{thm}

\begin{proof}
By Proposition \ref{prop:cup selmer equiv} and Lemma \ref{lem:merel and zeta}, we are reduced to showing that Merel's number is a $p^s$-th power modulo $N$ if and only if there exists some $x \in H^1(\Z/p^s\Z(2))$ with non-zero image in $H^1(\Z/p\Z(2))$ such that $x|_N \in \cL_N$.

By Lemma \ref{lem: Hochs-Serre} (and with the notation there), we have
\[
H^1(\Z/p^s\Z(2))= H^1(\Z[1/Np,\zeta_{p^s}],\Z/p^s\Z(2))^\Delta = (H^1(\Z[1/Np,\zeta_{p^s}],\Z/p^s\Z(1))(1)^\Delta.
\]
Then, by Kummer theory, we have an isomorphism
\[
\iota: H^1(\Z/p^s\Z(2)) \cong (\Z[1/Np,\zeta_{p^s}]^\times \otimes \Z/p^s\Z(1))^\Delta.
\]
There is a commutative diagram
\[\xymatrix{
(\Z[1/Np,\zeta_{p}]^\times \otimes \Z_p)_{\omega^{-1}} \ar[r]^-j \ar[d] & (\Z[1/Np,\zeta_{p^s}]^\times \otimes \Z/p^s\Z(1))^\Delta \ar[r]^-\sim \ar[d] & H^1(\Z/p^s\Z(2)) \ar[d] \\
(\Z[1/Np,\zeta_{p}]^\times \otimes \F_p)_{\omega^{-1}} \ar@{=}[r] & (\Z[1/Np,\zeta_{p}]^\times \otimes \Z/p\Z(1))^\Delta \ar[r]^-\sim & H^1(\Z/p\Z(2)), 
}\]
where $j$ is induced by the inclusion $\Z[1/Np,\zeta_{p}]^\times \subset \Z[1/Np,\zeta_{p^s}]^\times$. Letting $x=\iota^{-1}(j(\mathcal{G}))$, where $\mathcal{G}$ is as in Proposition \ref{prop:lecouturiers}, we see that the image of $x$ in $H^1(\Z/p\Z(2))$ is non-zero. We have $x|_N \in \cL_N$ if and only if $\Lambda(\mathcal{G}_\mathfrak{n}) =0$, and by Proposition \ref{prop:lecouturiers}, this happens if and only if Merel's number is a $p^s$-th power modulo $N$.
\end{proof}
 
Taking $s=1$ in the theorem gives Proposition \ref{prop:cup equivalence with zeta} from the introduction.
 
\section{Equivalence of Massey products}

In the previous section, we gave a direct algebraic proof that $a \cup c =0$ if and only if $b \cup c =0$. In this section, we prove the analogous result for higher Massey powers -- namely, that the Massey relations in the $(1,1)$, $(2,1)$ and $(2,2)$ coordinates are all equivalent.

To state the result, we fix $n \geq 2$ and $s \le t$ and assume that we have a very good deformation $\rho_n: G_{\Q,S} \to \zp{s}[\epsilon_n]$ of $\rho_1 \otimes_{\zp{t}} \zp{s}$, which we write as
\[
\rho_n = \ttmat{\kcyc}{0}{0}{1} + \sum_{i=1}^n \ttmat{\kcyc a_i}{b_i}{\kcyc c_i}{d_i} \epsilon^i
\]
with $a_1 = a \pmod{p^{s}}$, etc. Let $D$ be the associated defining system for the Massey power $\dia{M}^{n+1}$. 

\begin{prop}
\label{prop:massey equiv}
The $\Z/p^{s}\Z$-valued 1-cochains $a_n|_N,b_n|_N,c_n|_N, d_n|_N$ are 1-cocycles, i.e.\ they lie in $Z^1_N(\zp{s})$. In addition,
\begin{enumerate}
\item $a_n|_N =-d_n|_N$ in $H^1_N(\zp{s})$,
\item $a_n|_N = -b_n|_N$ in $H^1_N(\zp{s})$,
\item The Massey relation for $\dia{M}_D^{n+1}$ in the $(1,1)$-coordinate holds modulo $p^s$ if and only if $c_{n}|_N\equiv -b_n|_N$ in $H^1_N(\zp{s})$,
\item The Massey relation for $\dia{M}_D^{n+1}$ in the $(2,1)$-coordinate holds modulo $p^s$ if and only if $c_{n}|_N \equiv a_n|_N$ in $H^1_N(\zp{s})$
\item The Massey relation for $\dia{M}_D^{n+1}$ in the $(2,2)$-coordinate holds modulo $p^s$ if and only if $c_{n}|_N \equiv -b_n|_N$ in $H^1_N(\zp{s})$. 
\end{enumerate}
In particular, the Massey relation for $\dia{M}_D^{n+1}$ modulo $p^s$ in the $(1,1)$, $(2,1)$ and $(2,2)$-coordinates are all equivalent to $\dia{M}_D^{n+1}$ vanishing in $H^2(\End(\zp{s}(1)\oplus \zp{s}))$.
\end{prop}

\begin{proof}
The final statement of the proposition follows from Lemmas \ref{lem:upper-triangular equivs} and \ref{lem:massey iff all massey coords}.

As in the previous section, since $H^1_N(\zp{s})=Z^1_N(\zp{s})$, we conflate 1-cocycles with their cohomology classes. Since $n \ge 2$, by Lemma \ref{lem:res diag defs}, we have $M \cup M=0$, which implies that $a \cup c =0$. By Proposition \ref{prop:cup selmer equiv}, this implies that  $c|_N \in \cL_N \subset H^1_N(\zp{s})$. Then, by Lemma \ref{lem:aN in cLN}, we have $a|_N,b|_N,c|_N,d|_N \in \cL_N$, where $d=-a$. We may write them as multiples of $[N]$, the Kummer class of $N$. Write $b|_N = x[N]$ with $x \in (\zp{t})^\times$. By our normalizations \eqref{eq:normalization of a,b,c}, we have $a|_N=c|_N=-x[N]$ and $d|_N=x[N]$, as elements of $Z^1_N(\zp{s})$.

By induction, we may assume that $a_i|_N =c_i|_N = -b_i|_N =-d_i|_N$ for $i=1,\dots,n-1$. We first prove that $a_n|_N$ is a cocycle. Note that
\[
d a_n = \sum_{i=1}^{n-1} a_i \smile a_{n-i} +b_i \smile c_{n-i}.
\]
Restricting to $G_N$, we have by induction
\begin{align*}
d a_n |_N & = \sum_{i=1}^{n-1} a_i|_N\smile a_{n-i}|_N +b_i|_N\smile c_{n-i}|_N\\
& = \sum_{i=1}^{n-1} c_i|_N\smile c_{n-i}|_N +(-c_i|_N)\smile c_{n-i}|_N =0. \\
\end{align*}
Hence $a_n|_N \in Z_N^1(\zp{s})$. Similarly for $b_n,c_n,d_n$.

\begin{enumerate}[leftmargin=2em]
\item Since $\det(\rho_n)=\kcyc$, we have
\[
a_n+d_n = \sum_{i=1}^{n-1} b_ic_{n-i} -a_i d_{n-i}. 
\]
Using the induction hypotheses, this implies that
\begin{align*}
a_n|_N+d_n|_N &= \sum_{i=1}^{n-1} b_i|_Nc_{n-i}|_N -a_i|_N d_{n-i}|_N  \\
&= \sum_{i=1}^{n-1} (-c_i|_N)c_{n-i}|_N -(c_i|_N)(-c_{n-i}|_N)   = 0.
\end{align*}
\item By Lemma \ref{lem:upper-triangular equivs}, the Massey relation for $\dia{M}_D^{n+1}$ in the $(1,2)$-coordinate holds. In other words, the class of the cocycle
\[
\sum_{i=1}^{n} a_i \smile b_{n-i+1} + b_i \smile d_{n-i+1}
\]
is $0$ in $H^2(\Z/p^{s}\Z(1))$. Restricting to $H^2_N(\zp{s}(1))$ and applying (1) and the induction hypotheses, we find 
\begin{align*}
0 & = \sum_{i=1}^{n}  a_i|_N \smile b_{n-i+1}|_N + b_i|_N \smile d_{n-i+1}|_N \\
& = \sum_{i=2}^{n-1} \left( c_i|_N \smile (-c_{n-i+1}|_N) + (-c_i|_N) \smile (-c_{n-i+1}|_N) \right) \\
& \hspace{.5in} + x( -[N] \smile b_{n}|_N + [N] \smile (-a_{n}|_N) + a_n|_N \smile [N] + b_n|_N \smile [N]) \\
& =  x( -[N] \smile b_{n}|_N + [N] \smile (-a_{n}|_N) + a_n|_N \smile [N]+ b_n|_N \smile [N]) \\
& = x( -[N] \smile (a_n|_N+b_n|_N) + (a_n|_N+b_n|_N) \smile [N]) \\
& = 2x (a_n|_N+b_n|_N) \smile [N].
\end{align*}
Since $2x \in (\Z/p^{s}\Z)^\times$, the fact that the Tate pairing is symplectic implies that the class of $a_n|_N+b_n|_N$ in $H^1_N(\Z/p^{s}\Z)/\cL_N$ is zero. Since $\rho_n$ is very good, both $a_n$ and $b_n$ are unramified at $N$, so this implies that $a_n|_N=-b_n|_N$.
\item Let $\alpha \in Z^2(\Z/p^{s}\Z)$ denote the cocycle
\[
\alpha = \sum_{i=1}^{n} a_i \smile a_{n+1-i} + b_i \smile c_{n+1-i}.
\]
The Massey relation for $\dia{M}^{n+1}$ in the $(1,1)$-coordinate holds modulo $p^s$ if and only if $[\alpha] =0$ in $H^2(\Z/p^{s}\Z)$. By Lemma \ref{lem:H2 to H2N injective}, this is equivalent to the equation $[\alpha|_N] =0$ in $H^2_N(\Z/p^{s}\Z)$. Using (2), this equation can be simplified to $2x[N] \smile (b_n\vert_N + c_n\vert_N) = 0$. Then we apply the same kind of final argument as in the proof of (2). 
\end{enumerate}
The remaining parts are similar.
\end{proof}

\section{Main result}
\label{sec:main result}

Let $e=\mathrm{rank}_{\Z_p}(\bT^0)$. Recall the sequence $t= t_1 \ge \dots \ge t_e >t_{e+1}=0$, defined in Proposition \ref{prop:NPcontrol of T}. This sequence is invariant of $\bT$ that determines its Newton polygon, but may even be finer than it. In this section, we complete the proof of our main theorem, which is an inductive procedure: assuming we know $t_1, \dots, t_n$, we describe $t_{n+1}$ in terms of Massey products.

\begin{thm}
\label{thm:main in reorg}
Let $n$ and $s$ be integers such that $1 \le n \le e$ and $1 \le s \le t_{n}$. Then there is a very good $n$-th order deformation $\rho_n$ of $\rho_1$ modulo $p^{t_n}$. 

Fix such an $\rho_n$, and let $D$ be the corresponding defining system for the Massey power $\dia{M}^{n+1}$ modulo $p^{t_n}$. Then the following are equivalent:
\begin{enumerate}
\item We have $s \le t_{n+1}$.
\item The Massey power $\dia{M}_D^{n+1}$ vanishes in $H^2(\End(\zp{s}(1) \oplus \zp{s}))$.
\end{enumerate}
\end{thm}
\begin{proof}
We first prove that (1) and (2) are equivalent, assuming the existence of $\rho_n$. Let $\varphi: R \onto \zp{t_n}[\epsilon_n]$ be the corresponding surjective homomorphism. Let $z \in \Jm$ be a generator such that $\varphi(z)=\epsilon$. Then (1) is equivalent to
\begin{enumerate}
\item[(1)'] There is a surjective homomorphism $\varphi': R \onto \zp{s}[\epsilon_{n+1}]$ such that the following diagram commutes 
\[
\xymatrix{
R \ar[r]^-{\varphi'}  \ar[d]_-{\varphi} & \zp{s}[\epsilon_{n+1}] \ar[d] \\
\zp{t_n}[\epsilon_n] \ar[r] & \zp{s}[\epsilon_n],
}\]
where the unlabeled arrows are the quotient maps.
\end{enumerate}
Indeed, by Proposition \ref{prop:NPcontrol of T}, (1) implies the existence of a homomorphism $\varphi'$ such that $\varphi'(z)=\epsilon$, and such a homomorphism makes the diagram commute. Conversely, any $\varphi'$ as in (1)' must satisfy $\varphi(z')=\epsilon$ for some generator $z' \in \Jm$, which, by Proposition \ref{prop:NPcontrol of T}, implies (1).

By Lemma \ref{lem:lower-triangular equivs}, (1)' is equivalent to the Massey relation for $\dia{M}_D^{n+1}$ in the $(2,1)$-coordinate modulo $p^s$. This is equivalent to (2) by Proposition \ref{prop:massey equiv} (and by Remark \ref{rem:<M>^2 and cups} for $n=1$).

Now we prove that $\rho_n$ exists by induction on $n$, the base case $n=1$ being vacuous. Assume that $\rho_{n-1}$ exists modulo $p^{t_{n-1}}$, and let $D'$ be the corresponding corresponding defining system for the Massey power $\dia{M}^{n}$ modulo $p^{t_{n-1}}$. By the equivalence of (1) and (2) already proven, we see that the Massey power $\dia{M}_{D'}^{n}$ vanishes in $H^2(\End(\zp{t_{n}}(1) \oplus \zp{t_{n}}))$. By Lemma \ref{lem:res diag defs}, there is a very good $n$-th order deformation of $\rho_{n-1}$ modulo $p^{t_n}$, which we can take as $\rho_n$.
\end{proof}

\begin{rem}
\label{rem:defining systems}
Note that since (1) in the theorem does not depend on the choice of defining system $D$, the vanishing behavior of the Massey power $\dia{M}_D^{n+1}$ does not depend on the choice of $D$ (as long as it is associated to a very good $n$-th order deformation).
\end{rem}

We now explain how to deduce the results stated in the introduction from this main theorem. For $n=1$ in the theorem, we let $\rho_n=\rho_1$ and observe that $e \ge2 $ if and only if $t_2 >0$ if and only if $\dia{M}_D^{2}=M \cup M$ vanishes in $H^2(\End(\F_p(1) \oplus \F_p))$. This is equivalent to $b \cup c=0$ and to $a \cup c=0$ (see Remark \ref{rem:<M>^2 and cups}). This proves Theorem \ref{thm:cup products and rank 1}. Corollary \ref{cor:class groups and rank 1} follows from this and Proposition \ref{prop:cup gives class}. 

Finally, to prove Theorem \ref{thm:newton poly}, we note that if Merel's number is not a $p^2$-th power, then Theorem \ref{thm:cup merel equiv} implies that $b \cup c$ is non-zero in $H^2(\zp{2})$, which implies that $M \cup M$ is non-zero in $H^2(\End(\zp{2}(1) \oplus \zp{2}))$. By the main theorem, this implies that $t_2\le 1$, which implies Theorem \ref{thm:newton poly} by standard properties of Newton polygons.

\part{Appendices}
In the appendices, we collect some formal results. With the possible exception of \S \ref{subsec:massey coords} and \S \ref{subsec:ff cohom}, the contents are standard and will be known to experts. We include them here for completeness and to fix notation.

\appendix

\section{Massey products}
\label{sec:massey appendix}
Massey products are a generalization of cup products. They were first introduced in topology by Massey and Uehara--Massey \cite{massey1958,MU1957}. For an introduction to the subject, see Kraines \cite{kraines} and May \cite{may1969}. Massey products are closely related to $A_\infty$-operations; see e.g.\ \cite[Part 2]{CarlAinf} for this relation, and the connection with deformation theory. For applications of Massey products in Galois cohomology, see Sharifi \cite{sharifi2007}. 

In this section, we collect some statements that we will need and define ``Massey powers.'' We do not give proofs, as all the results either follow immediately from the definitions or by a purely formal computation.

In this section, we let $G$ be a group, $A$ be a ring, and $V$ a $A[G]$-module equipped with a pairing $V \otimes V \to V$. Given $a \in C^i(G,V)$, $b \in C^j(G,V)$ we let $a \smile b \in C^{i+j}(G,V)$ denote the composite of the usual cup product with the pairing $V \otimes V \to V$:
\[
C^i(G,V) \times C^j(G,V) \to C^{i+j}(G,V \otimes V) \to C^{i+j}(G,V).
\]

\subsection{Massey products} 

\begin{defn}
Let $a_1, \dots, a_n \in C^1(G,V)$ be cochains. We say that a set $D=\{a(i,j): 1 \le i \le j \le n, (i,j) \ne (1,n)\} \subset C^1(G,V)$ is a \emph{defining system for the Massey product} $\dia{a_1,\dots, a_n}$ if
\begin{enumerate}
\item $a(i,i)=a_i$ for all $i=1,\dots n$, and
\item $\displaystyle da({i,j}) = \sum_{k=i}^{j-1} a(i,k)\smile a(k+1,j)$ for all $i,j$.
\end{enumerate}
In particular, (1) and (2) for $i=j$ imply that $da(i,i)=da_i=0$ for all $i$.

If $D$ is a defining system for the Massey product $\dia{a_1,\dots, a_n}$, then we note that 
\[
c(D)=\sum_{k=1}^{n-1} a(1,k)\smile a(k+1,n) 
\]
is an element of $Z^2(G,V)$ and we let $\dia{a_1,\dots, a_n}_D \in H^2(G,V)$ be the class of $c(D)$. We let
\[
\dia{a_1,\dots, a_n}= \{ \dia{a_1,\dots, a_n}_D  \} \subset H^2(G,V)
\]
where $D$ ranges over all defining systems.

We say that $\dia{a_1,\dots, a_n}$ is \emph{defined} if it is non-empty (i.e.~ if there exists a defining system). We say that $\dia{a_1,\dots, a_n}$ \emph{vanishes} if $0 \in \dia{a_1,\dots, a_n}$.
\end{defn}

It is known that the set $\dia{a_1,\dots, a_n}$ only depends on the cohomology classes of $a_1, \dots, a_n$ \cite[Thm.\ 3]{kraines}.

\begin{eg}
If $n=2$, then the Massey product is defined if and only if $a_1, a_2 \in Z^1(G,V)$. If they are, then $D=\{a(1,1)=a_1,a(2,2)=a_2\}$ is the only defining system, and $\dia{a_1,a_2}_D=[a_1 \smile a_2]$.
\end{eg}

\begin{eg}
\label{example:massey products and unipotent upper-triangular}
Take $V=A$ with trivial $G$-action. Suppose that $D=\{a(i,j): 1 \le i \le j \le n, (i,j) \ne (1,n)\} \subset C^1(G,A)$ is a defining system. Condition (2) implies that the the cochains $\nu_1, \nu_2 \in C^1(G,M_{n}(A))$ given by
\[
\nu_1 = \left( \begin{array}{ccccc}
1 & a(1,1) & a(1,2) & \cdots  & a(1,n-1) \\
0 & 1 & a(2,2) & \cdots & a(2,n-1) \\
\cdots \\
0 & \cdots &   & 1 & a(n-1,n-1) \\
0 & \cdots &  & 0  &  1
\end{array} \right)
\]
and
\[
\nu_2 = \left( \begin{array}{ccccc}
1 & a(2,2) & a(2,3) & \cdots  & a(2,n) \\
0 & 1 & a(3,3) & \cdots & a(3,n) \\
\cdots \\
0 & \cdots &   & 1 & a(n,n) \\
0 & \cdots &  & 0 &  1
\end{array} \right)
\]
are cocycles (i.e.~ $\nu_1$ and $\nu_2$ are homomorphisms). Notice that $\nu_1$ and $\nu_2$ have a $n-1 \times n-1$-submatrix in common. The class $\dia{a_1,\dots, a_n}_D \in H^2(G,A)$ measures the obstruction to concatenating $\nu_1$ and $\nu_2$, in the following sense. If $\dia{a_1,\dots, a_n}_D=0$, then there exists $a \in C^1(G,A)$ such that $da=c(D)$ and the cochain $\nu \in C^1(G,M_{n+1}(A))$ given by
\[
\nu= \left( \begin{array}{cccccc}
1 & a(1,1) & a(1,2) & \cdots  & a(1,n-1) & a\\
0 & 1 & a(2,2) & \cdots & a(2,n-1) & a(2,n) \\
\cdots \\
0 & \cdots &   & 1 & a(n-1,n-1) & a(n-1,n) \\
0 & \cdots &  & &  1 & a(n,n) \\
0 & \cdots & & & 0 &  1
\end{array} \right)
\]
is a cocycle. Moreover, if $\dia{a_1,\dots, a_n}_D \neq 0$, then no such $\nu$ exists.
\end{eg}

\subsection{Massey powers}
\label{subsec: Massey powers}

We remark that ``Massey power'' is not standard terminology; we use it to refer to certain Massey products. More precisely, a Massey power is not merely a Massey product of a set of identical 1-cochains, but also requires a symmetry in the defining system that is not required by the definition of a Massey product. These symmetries naturally occur in the defining systems induced by deformations, as discussed in \S\ref{sec: MVD}. 

\begin{defn}
Let $a \in C^1(G,V)$ be a cochain, and let $m_1, \dots, m_{k-1} \in C^1(G,V)$. We say that $D:=\{m_1, \dots, m_{k-1}\}$ is a \emph{defining system for the Massey power} $\dia{a}^k$ if the set
\[
\tilde{D}=\{a(i,j)=m_{j-i+1} : 1 \le i \le j \le k, (i,j) \ne (1,k)\}
\]
is a defining system for the Massey product $\dia{a, \dots, a}$ (with $a$ repeated $k$ times). If $D$ is a defining system for the Massey power $\dia{a}^k$, then we let $\dia{a}^k_D=\dia{a,\dots,a}_D$, and we let $c(D):=c(\tilde{D})$. We let 
\[
\dia{a}^k = \{ \dia{a}^k_D\} \subset H^2(G,V)
\]
where $D$ ranges over defining systems for the Massey powers. Note that $\dia{a}^k \subset \dia{a, \dots, a}$.
\end{defn}

Note that, for $D=\{m_1, \dots, m_{k-1}\} \subset C^1(G,V)$, $D$ is a defining system for the Massey power $\dia{a}^k$ if and only if $m_1=a$ and, for all $i=1, \dots, k-1$, we have
\[
dm_i = \sum_{j=1}^{i-1} m_j \smile m_{i-j}.
\]
We also note that, for such $D$, we have
\[
c(D) = \sum_{j=1}^{k-1} m_j \smile m_{k-j}.
\]

\begin{lem}
	\label{lem:massey powers and deformations}
	Let $\nu: G \to \GL_n(A)$ be a representation, and let $V=\End(\nu)$. Let $M_1, \dots, M_{r} \in C^1(G,V)$ and let $M=M_1$, and, for $i=1, \dots, r$, define $\nu_i: G \to \GL_n(A[\epsilon_i])$ by
	\[
	\nu_i = \nu + \sum_{j=1}^i M_j\epsilon^j.
	\] 
	Assume that $\nu_{r-1}$ is a homomorphism. Then $D=\{M_1,\dots, M_{r-1}\}$ is a defining system for $\dia{M}^r$, and $\nu_r$ is a homomorphism if and only if $dM_r=c(D)$ (in which case $\dia{M}^r_D=0$).
\end{lem}

\subsection{Coordinates of matrix Massey products}
\label{subsec:massey coords} 
In the situation of the previous lemma, if $\nu$ is a reducible representation, it is interesting to consider the matrix coordinates of the Massey power, as we now explain. For the rest of this section, we fix two characters $\chi_1, \chi_2: G \to A^\times$, and let $\nu=\chi_1 \oplus \chi_2$. We also fix $M\in Z^1(G,\End(\nu))$. 

\begin{defn}
\label{defn:massey coordinates}
Let $M\in Z^1(G,\End(\nu))$ and let $D=\{M_1,\dots,M_{r-1}\}$ be a defining system for the Massey power $\dia{M}^r$ in $H^2(G, \End(\nu))$. Write $M_i$ as
\[
M_i = \ttmat{\chi_1a_{11}^{(i)}}{\chi_2a_{12}^{(i)}}{\chi_1a_{21}^{(i)}}{\chi_2a_{22}^{(i)}}.
\]
where we think of $a_{11}^{(i)}$ and $a_{22}^{(i)}$ as elements of $C^1(G,A)$ and $a_{12}^{(i)}$ and $a_{21}^{(i)}$ as elements of $C^1(G,\chi_1^{-1}\chi_2)$ and $C^1(G,\chi_1\chi_2^{-1})$, respectively.

Consider the matrix
\[\tag{$*$}
\sum_{j=1}^{r-1} \ttmat{a_{11}^{(j)} \smile a_{11}^{(r-j)}+ a_{12}^{(j)} \smile a_{21}^{(r-j)}}
{a_{11}^{(j)} \smile a_{12}^{(r-j)}+ a_{12}^{(j)} \smile a_{22}^{(r-j)}}
{a_{21}^{(j)} \smile a_{11}^{(r-j)}+ a_{22}^{(j)} \smile a_{21}^{(r-j)}}
{a_{21}^{(j)} \smile a_{12}^{(r-j)}+ a_{22}^{(j)} \smile a_{22}^{(r-j)}} 
\]
as an element in
\[
\ttmat{Z^2(G,A)}
{Z^2(G,\chi_1^{-1}\chi_2)}
{Z^2(G,\chi_1\chi_2^{-1})}
{Z^2(G,A)}.
\]
For $s,t\in\{1,2\}$, we say that \emph{the Massey relation for $\dia{M}^r_D$ holds in the $(s,t)$-coordinate} if the $(s,t)$-coordinate of the matrix $(*)$ vanishes in cohomology. 

For example, the Massey relation for $\dia{M}^r_D$ holds in the $(2,1)$-coordinate if and only if
\[
\sum_{j=1}^{r-1} a_{21}^{(j)} \smile a_{11}^{(r-j)}+ a_{22}^{(j)} \smile a_{21}^{(r-j)} \in B^2(G,\chi_1\chi_2^{-1}).
\]
\end{defn}

\begin{lem}
\label{lem:massey iff all massey coords}
Let $D=\{M_1,\dots,M_{r-1}\}$ be a defining system for the Massey power $\dia{M}^r$ in $H^2(G, \End(\nu))$. The Massey relation for $\dia{M}_D^r$ holds in the $(s,t)$-coordinate for all $s,t\in \{1,2\}$ if and only if $\dia{M}_D^r=0$.
\end{lem}

The purpose of the $(s,t)$-Massey relations is that they are useful for comparing Massey products for different representations with the same semi-simplification.

With the notation as above, define a function $\nu'$ by
\[
\nu'=\ttmat{\chi_1}{0}{\chi_1a_{21}^{(1)}}{\chi_2}.
\]
Since $M_1$ is a cocycle, $\nu'$ is a homomorphism.

\begin{prop}
\label{prop:relations between coord masseys}
Let $r>1$ and let $D=\{M_1,\dots,M_{r-1}\}$ be a defining system for the Massey power $\dia{M}^r$ in $H^2(G,\End(\nu))$. Define $a_{st}^{(i)}$ as in Definition \ref{defn:massey coordinates}. For $1 \leq i<r-1$, define $M_i'$ by the formula
\[
M_i' = \ttmat{\chi_1a_{11}^{(i)}}{\chi_2a_{12}^{(i-1)}}{\chi_1a_{21}^{(i+1)}}{\chi_2a_{22}^{(i)}}.
\]
with $a_{12}^{(0)}=0$ and let $M'=M_1'$. Then 
\begin{enumerate}
\item $D'=\{M_1',\dots,M_{r-2}'\}$ is a defining system for $\dia{M'}^{r-1}$ in $H^2(G,\End(\nu'))$, and
\item $\dia{M'}^{r-1}_{D'}=0$ in $H^2(G,\End(\nu'))$ if and only if Massey relation for $\dia{M}^r_D$ holds in the $(2,1)$-coordinate.
\end{enumerate} 
\end{prop}

\begin{lem}
\label{lem:deformations from relations between coord masseys}
Let $\{M_1,\dots,M_{r-1}\} \subset C^1(G,V)$, and suppose that $\nu_{r-1}: G \to \GL_2(A[\epsilon_{r-1}])$ is a homomorphism, where 
\[
\nu_{r-1} = \nu + \sum_{j=1}^{r-1} M_j \epsilon^j.
\]
Define $M_i'$ as in Proposition \ref{prop:relations between coord masseys}. Choose an element $a \in C^1(G,\chi_1^{-1}\chi_2)$ and define
\[
M_{r-1}' =  \ttmat{\chi_1a_{11}^{(r-1)}}{\chi_2a_{12}^{(r-2)}}{\chi_1a}{\chi_2a_{22}^{(r-1)}}.
\]
For $i=1,\dots, r-1$, define $\nu_i': G \to \GL_2(A[\epsilon_{i}])$ by
\[
\nu_i' = \nu' + \sum_{j=1}^{i} M_j' \epsilon^j.
\]
Then $\nu_i'$ is a homomorphism for $i<r-1$, and $\nu_{r-1}'$ is a homomorphism if and only if
\[
da= \sum_{j=1}^{r-1} a_{21}^{(j)} \smile a_{11}^{(r-j)}+ a_{22}^{(j)} \smile a_{21}^{(r-j)}.
\]
\end{lem}

\section{Galois cohomology - generalities}
\label{sec:galois cohomology appendix}

In this section, we use cone constructions to define cochain complexes that compute Galois cohomology with various local conditions. In particular, we discuss the compactly supported, partially compactly supported, and finite-flat variants. The idea to consider derived versions of Selmer groups is due to Nekov\'a\u{r} \cite{nekovar2006}. For a more down-to-earth treatment (and all that will be needed here), see \cite[App.~B]{GV2018}, where they use the notation of fundamental groups $\pi_1^\et(\Z[1/Np])$ (resp.\ $\pi_1^\et(\Q_\ell)$) in place of our $G_{\Q,S}$ (resp.\ $G_\ell$).

While $N$ usually denotes a prime in the main text, here we allow it to be a squarefree integer prime to $p$, as in \S\S\ref{subsec: calc H1p-fl}-\ref{subsec: cohom compute}.

\subsection{Notation from homological algebra} If $(C^\bullet,d)$ is a cochain complex, we let $Z^i(C^\bullet)=\ker(d:C^i \to C^{i+1})$ and $B^i(C^\bullet)=\mathrm{im}(d:C^{i-1}\to C^i)$. Let $(C[i]^\bullet,d[i])$ be the complex $C[i]^j=C^{j-i}$ with differential $d[i]=(-1)^id$. If $f:A^\bullet\to B^\bullet$ is a map of cochain complexes, we let $\mathrm{Cone}(f)^\bullet$ be the complex $\mathrm{Cone}(f)^i=B^i \oplus A^{i+1}$ and $d(b,a)=(db-f(a),-da)$. Then there is an exact sequence
\[
0\lra B^\bullet \xrightarrow{b \mapsto (b,0)} \mathrm{Cone}(f)^\bullet \xrightarrow{(a,b) \mapsto a} A[-1]^\bullet \lra 0.
\]

\subsection{Notation for group cochains and cohomology groups}
\label{subsec:cohomology notation} Let $G$ be a topological group, and let $M$ be a continuous $G$-module. Let $C^\bullet(G,M)$ denote the complex of continuous inhomogeneous cochains.

For $N' \mid Np$, we define 
\begin{align*}
C^\bullet(-):= C^\bullet(\Z[1/Np],-) := C^\bullet(G_{\Q,S}, -),& \quad C^\bullet_\ell(-):= C^\bullet(\Q_\ell,-) := C^\bullet(G_\ell, -), \\
 C^\bullet_\loc(-) : = \bigoplus_{\ell \mid Np \ \mathrm{ prime}} C^\bullet_\ell(-),& \quad
 C^\bullet_{N'}(-) : = \bigoplus_{\ell \mid N' \ \mathrm{ prime}} C^\bullet_\ell(-)
\end{align*}
We let $x \mapsto x|_{N'}$ denote the restriction map $C^\bullet(-) \to C^\bullet_{N'}(-)$. We let 
\[
C_{(c)}^\bullet(M) = \mathrm{Cone}(C^\bullet(M) \to C^\bullet_\loc(M))[1], \ \  C_{(N')}^\bullet(M) = 
\mathrm{Cone}(C^\bullet(M) \to C^\bullet_{N'}(M))[1]
\]

The associated cohomology groups are $H^i_{\star}(-):=H^i(C_\star^\bullet(-))$, where $\star$ is one of the symbols $\{ -, \ell, N', \mathrm{loc}, (c), (\ell), (N')\}$. We call $H^i_{(c)}(-)$ compactly supported Galois cohomology, in analogy with the geometric situation.

\subsection{Duality theories}
\label{subsec:duality} Let $M$ denote a $p$-power torsion $G_{\Q,S}$-module, and let $M^*$ denote the Pontryagin dual of $M$. We have the the following duality theorem of Poitou--Tate, which resembles Poincar\'e duality.
\begin{thm}
\label{thm:Poitou-tate}
For $i=0,\dots,3$, the cup product induces a perfect paring
\[
H^i(M)  \times H^{3-i}_{(c)}(M^*(1)) \to \Q_p/\Z_p.
\]
\end{thm}

Then duality theory with ``local constraints" gives the following generalization of Theorem \ref{thm:Poitou-tate}. 

\begin{thm}
\label{thm:Poitou--Tate with constraints}
For any divisor $N' \mid Np$ and $i=0,\dots,3$, the cup product induces a perfect paring
\[
H^i_{(N')}(M)  \times H^{3-i}_{(Np/N')}(M^*(1)) \lra \Q_p/\Z_p.
\]
\end{thm}
\begin{proof}
This is a special case of the duality theorem of \cite[App.~B]{GV2018}. In the notation of that theorem, we take the set of primes $S$ to be the primes dividing $Np$, and, for any divisor $n$ of $Np$, we define a condition $\cL_{n}$ by
\[
C_{\cL_{n}}(\Q_\ell, M) = \left \{  \begin{array}{ll}
0 & \text{ if } \ell \mid n \\
C(\Q_\ell,M)  & \text{ if } \ell \nmid n
\end{array}
\right.
\]
for all primes $\ell$ dividing $Np$. Then we see that $H^*_{\cL_{N'}}(M)=H^*_{(N')}(M)$, and that the dual condition $\cL_{N'}^\perp$ is given by $\cL_{Np/N'}$. The theorem follows from \cite[Thm.\ B.1]{GV2018}.
\end{proof}

\subsection{Extensions of finite flat group schemes and cohomology}
\label{subsec:ff cohom}

Let $\mathcal{G}/\Z_p$ be a finite flat group scheme such that $\nu_\mathcal{G} :=\mathcal{G}(\bQp)$ is free of finite rank as a $\Z/p^r$-module. Then there is a subgroup
\[
\Ext_{\fl}^1(\nu_\mathcal{G},\nu_\mathcal{G}) \subset \Ext^1_{\Z/p^r\Z[G_p]}(\nu_\mathcal{G},\nu_\mathcal{G}) \cong H^1_p(\End(\nu_\mathcal{G}))
\]
coming from extensions in the category of finite flat group schemes over $\Z_p$ that are killed by $p^r$. We denote 
\[
H^1_{p,\fl}(\End(\nu_\mathcal{G})) := \Ext_{\fl}^1(\nu_\mathcal{G},\nu_\mathcal{G}).
\]

We also want to define $H^1_{p,\fl}(-)$ in two other specific cases. If $\mathcal{G} \simeq \mu_{p^r}\otimes_{\Z/p^r\Z} A$ for some $\Z/p^r\Z$-module $A$, then $\nu_\mathcal{G}\simeq A(1)$. We write $H^1_{p,\fl}(A(1))$ for the subgroup
\[
\Ext_{\fl}^1(\Zr,A(1)) \subset \Ext^1_{\Z/p^r\Z[G_p]}(\Zr,A(1)) \cong H^1_{p}(A(1))
\]
and $H^1_{p,\fl}(A(-1))$ for the subgroup
\[
\Ext_{\fl}^1(A(1),\Zr) \subset \Ext^1_{\Z/p^r\Z[G_p]}(A(1),\Zr) \cong H^1_{p}(A(-1)).
\]

Now suppose $V$ is a $G_{\Q,S}$-module such that $V|_{G_p}$ is isomorphic to either $\nu_\mathcal{G}$, $A(1)$, or $A(-1)$, as above. We wish to define cochain complexes $C_{p,\fl}^\bullet(V)$ and $C_\fl^\bullet(V)$ such that 
\begin{enumerate}
\item $H^1(C_{p,\fl}^\bullet(V)) = H^1_{p,\fl}(V)$
\item $H^1(C_\fl^\bullet(V)) = \ker(H^1(V) \to H^1_p(V)/H^1_{p,\fl}(V))$
\item $H^2(C_\fl^\bullet(V))$ controls obstructions to global finite-flat deformations.
\end{enumerate}

\subsubsection{The flat and non-flat local cochain complexes} We define $C_{p,\fl}^\bullet(V)$ by
\[
C_{p,\fl}^i(V)= \left \{ \begin{array}{ll}
C_p^0(V) & \text{if } i=0 \\
Z_{p,\fl}^1(V) & \text{if } i=1 \\
0 & \text{if } i\ge 2
\end{array}
\right.
\]
where 
\[
Z_{p,\fl}^1(V) := \ker(Z^1_p(V) \to H^1_p(V)/H^1_{p,\fl}(V)).
\]
Then it is clear that $C_{p,\fl}^\bullet(V) \subset C_p^\bullet(V)$ is a subcomplex, and that $H^1(C_{p,\fl}^\bullet(V)) = H^1_{p,\fl}(V)$. We define $H^i_{p,\fl}(V) : = H^i(C_{p,\fl}^\bullet(V))$.

We define 
\[
C_{p,\nfl}^\bullet(V) = \mathrm{Cone}(C_{p,\fl}^\bullet(V) \to C_p^\bullet(V))
\]
and $H^i_{p,\nfl}(V) : = H^i(C_{p,\nfl}^\bullet(V))$. Then we have 
\[
H^0_{p,\nfl}(V)=0, \qquad H^2_{p,\nfl}(V)=H^2_{p}(V),
\]
and an exact sequence 
\[
0 \lra  H^1_{p,\fl}(V) \lra H^1_{p}(V)  \lra  H^1_{p,\nfl}(V)  \lra 0.
\]
\subsubsection{The global finite-flat cochain complex} 
\label{sssec:global flat}
Let $(-)|_{p,\fl}: C^\bullet(V) \to C_{p,\nfl}^\bullet(V)$ denote the composition
\[
C^\bullet(V) \buildrel{\vert_p}\over\lra C_{p}^\bullet(V) \lra C_{p,\nfl}^\bullet(V).
\]
Let 
\[
C_\fl^\bullet(V) : = \mathrm{Cone}(C^\bullet(V) \xrightarrow{|_{p,\fl}} C_{p,\nfl}^\bullet(V))[1].
\]
We call the resulting cohomology $H^\bullet_\fl(V):= H^\bullet(C_\fl^\bullet(V))$ \emph{global flat cohomology}. The long exact sequence of the cone is 
\begin{align*}
0=H^0_{p,\nfl}(V) &\ \lra H^1_\fl(V) \lra H^1(V) \lra  H^1_{p,\nfl}(V) \\ & \lra H^2_\fl(V) \lra H^2(V) \lra H^2_{p,\nfl}(V) \lra H^3_\fl(V) \lra 0.
\end{align*}
Take note of the isomorphisms 
\[
H^1_{p,\nfl}(V) \cong \frac{H^1_{p}(V)}{H^1_{p,\fl}(V)}, \quad H^2_{p,\nfl}(V) \cong H^2_p(V), \text{ and } 
H^3_\fl(V) \cong H^3_{(p)}(V),
\]
which are useful interpretations of terms of the sequence.

We will often refer to $Z^1_\fl(V)$, which we will take to be the kernel of $Z^1(V) \ra H^1_p(V)/H^1_{p, \fl}$. (This is part of the data of a cocycle in the cone defining $H^1_\fl(V)$.)

\section{Operations in homological algebra in terms of cocycles}
\label{sec:algebra appendix}

In this section, we show that some standard operations on representations, described in terms of matrices and cocycles, behave nicely with finite-flat cohomology. The reason is that these operations correspond to operations on extensions in a general exact category, and so can be done equally well in the category of finite flat group schemes of fixed exponent over a scheme, which is an full additive subcategory of the category of abelian category of \emph{fppf}-sheaves of abelian groups of that exponent, and is closed under extensions (see \cite[Prop.\ III.17.4, pg.~110]{oort1966}).

Below $\cC$ will denote any exact category. This means $\cC$ is an additive category equipped with a class of pairs of composable morphisms $A \to X \to B$ that should be thought of as exact sequences, and satisfy certain axioms -- for a precise definition, see \cite{buhler2010}, for example. For our purposes, it suffices to assume that $\cC$ is a full additive subcategory of an abelian category that is closed under extensions.

\subsection{Pushout}

Suppose we have short exact sequences
\begin{align*}
\cE': 0 \lra C \lra X \buildrel{j}\over\lra B \lra 0 \\
\cE: 0 \lra X \buildrel{i}\over\lra X' \lra A \lra 0.
\end{align*}
in the exact category $\cC$. Then, by the axioms of an exact category, the pushout $X''$ of $i$ and $j$ sits in an exact sequence
\[
0 \lra B \lra X'' \lra A \lra 0
\]
where $X'' \to A$ is induced by the composite $X' \times B \to X' \to A$. We called this extension the pushout of $\cE$ by $\cE'$.

\begin{eg}
\label{eg:pushout}
Let $\cC$ be the exact category of representations of a group $G$ in projective $R$-modules of constant finite rank, where $R$ is a commutative ring. We explain how to interpret the pushout construction in terms of matrices. Write an object $A$ of $\cC$ as a pair $(V_A,\rho_A)$ with $V_A$ a finite constant rank projective $R$-module and $\rho_A:G \to \GL(V_A)$ a homomorphism.

Suppose we have $A,B,C,X,X'$ as above in this category. Then we may write $X$ in block matrix form as
\[
\rho_X = \ttmat{\rho_C}{\rho_{\cE'}}{0}{\rho_B}
\]
and $X'$ as
\[
\rho_{X'} =\left( \begin{matrix}
\rho_A & 0 & 0 \\
\rho_{\cE,1} & \rho_C & \rho_{\cE'} \\
 \rho_{\cE,2} & 0 & \rho_{B}
\end{matrix}\right).
\]
Direct computation as in Example \ref{eg:baer in terms of matrices} below shows that the pushout of $\cE$ by $\cE'$ is given by the block matrix
\[
\ttmat{\rho_A}{0}{\rho_{\cE,2}}{\rho_B}.
\]
\end{eg}
\subsection{Pullback}
Suppose we have short exact sequences
\begin{align*}
\cE': 0 \lra B \lra X \buildrel{j}\over\lra A \lra 0 \\
\cE: 0 \lra C \buildrel{i}\over\lra Y \lra A \lra 0.
\end{align*}
in the exact category $\cC$. Then, by the axioms of an exact category, the pullback $Z=X \times_A Y$ sits in an exact sequence
\[
0 \lra B \lra Z \lra Y \lra 0.
\]
We call this the pullback of $\cE$ along $\cE'$.

Suppose $Y$ gives an extension $\cE$ of $A$ by $C$ and $X$ gives an extension $\cE'$ of $A$ by $B$. Then we can construct the pullback extension of $Y$ by $B$ as follows. Let $Z=X \times_A Y$. The map $Z \to Y$ is an epimorphism with kernel isomorphic to $B$. We call the resulting extension of $Y$ by $B$ the pullback of $\cE$ along $\cE'$.

\begin{eg}
\label{eg:pullback}
We return to the category of finite rank representations from the previous example, and retain the notation there. Suppose we have $A,B,C,X,Y$ as above in this category. Then we may write $Y$ in block matrix form as
\[
\rho_Y = \ttmat{\rho_A}{0}{\rho_{\cE}}{\rho_C}
\]
and $X$ as
\[
\rho_X = \ttmat{\rho_A}{0}{\rho_{\cE'}}{\rho_B}
\]

By direct computation as in Example \ref{eg:baer in terms of matrices} below, we see that the pullback of $\cE$ along $\cE'$ is given by the block matrix
\[
\left(
\begin{matrix}
\rho_A & 0 & 0 \\
\rho_{\cE} & \rho_C & 0\\
\rho_{\cE'} & 0 & \rho_B
\end{matrix}\right).
\]
\end{eg}

\subsection{Baer sum}

Suppose we have short exact sequences
\begin{align*}
\cE: \ \ & 0 \lra B \buildrel{i}\over\lra X \buildrel{j}\over\lra A \lra 0 \\
\cE': \ \ & 0 \lra B \buildrel{i'}\over\lra X' \buildrel{j'}\over\lra A \lra 0
\end{align*}
in $\cC$. Then, by the axioms of an exact category, the direct sum $\cE \oplus \cE'$ is an extension of $A\oplus A$ by $B \oplus B$. The Baer sum $\cE + \cE'$ is the extension of $A$ by $B$ obtained by pulling back $\cE \oplus \cE'$ by the diagonal $A \to A \oplus A$ and then pushing out the result by the sum map $B\oplus B \to B$.

In an abelian category, there is an alternate construction of $\cE + \cE'$ given as follows. There is a skew diagonal map $\Delta^s: B \to X \times_A X'$ given by $\Delta^s=i \times (-i')$. Let $Y=\coker(\Delta^s)$. The composite $X \times_A X' \to X' \xrightarrow{j'} A$ induces an epimorphism $Y \to A$ whose kernel is isomorphic to $B$. The resulting extension of $A$ by $B$ is defined to be $\cE + \cE'$.

\begin{eg}
\label{eg:baer in terms of matrices}
We return to the category $\cC$ of the previous examples. We explain how to interpret the Baer sum construction in terms of matrices.

Recall that we write an object $A$ of $\cC$ as a pair $(V_A,\rho_A)$. For an extension $\cE$ of $A$ by $B$ as above, we can choose a decomposition $V_X = V_A \oplus V_B$, and write $\rho_X$ in block matrix form as
\[
\rho_X = \ttmat{\rho_A}{0}{\rho_\cE}{\rho_B}
\]
with $\rho_\cE \in Z^1(G, \Hom(V_A,V_B))$; the extension $\cE$ is determined by this cocycle. 

Now, given two extensions $\cE$ and $\cE'$ of $A$ by $B$, to describe the extension $\cE + \cE'$, we need only describe the cocycle $\rho_{\cE + \cE'}$. We claim that it is given by $\rho_{\cE + \cE'} = \rho_\cE+\rho_{\cE'}$. Indeed, for $Y=\coker(\Delta^s)$ as above, we can write a $V_Y$ as a direct sum $V_Y = \{(b,0)|b \in V_B\} \oplus \{(a,a)|a \in V_A\}$. Then for $\sigma \in G$ and $a \in V_A$, the cocycle $\rho_{\cE + \cE'}$ is defined by the formula
\[
\rho_Y(\sigma)(a,a) = (\rho_A(\sigma)a,\rho_A(\sigma)a)+(\rho_{\cE + \cE'}(\sigma)a,0).
\]
On the other hand, we compute that
\begin{align*}
\rho_Y(\sigma)(a,a) & = (\rho_A(\sigma)a+ \rho_\cE(\sigma)a,\rho_A(\sigma)a+\rho_{\cE'}(\sigma)a) \\
& = (\rho_A(\sigma)a,\rho_A(\sigma)a)+(\rho_{\cE}(\sigma)a,\rho_{\cE'}(\sigma)a) \\
& = (\rho_A(\sigma)a,\rho_A(\sigma)a)+(\rho_{\cE}(\sigma)a+\rho_{\cE'}(\sigma)a, 0),
\end{align*}
using the fact that $(-\rho_{\cE'}(\sigma)a,\rho_{\cE'}(\sigma)a)=0$ in $V_Y$.
\end{eg}

\begin{rem}
\label{rem:Baer scalar}
In this situation, there is also a ``Baer scalar product'' defining an $R$-module structure on $\Ext^1_\cC(B,A)$. For $r \in R$ and $\cE \in \Ext^1_\cC(B,A)$ as above, the extension $r \cdot \cE$ is obtained as a quotient of the direct sum $X \oplus A$.
\end{rem}

\subsection{Application to finite-flat representations}
We apply the above examples to the case of finite-flat deformations. Let $R$ be a commutative ring of finite cardinality.

\begin{lem}
\label{lem:cocycles from flat reps are flat}
Let $\nu: G_p \to \GL_n(R)$ be a finite-flat representation, and let $\nu_r: G_p \to \GL_n(R[\epsilon_r])$ be a finite-flat deformation of $\nu$ for some $r \ge 1$. Let $x  \in C^1(G_p,\End(\nu))$, and let $\nu_r' = \nu_r+x\epsilon^r$. Then $\nu_r'$ is a finite-flat representation if and only if $x \in Z^1_\fl(G_p,\End(\nu))$.
\end{lem}

\begin{proof}
We can think of a free $R[\epsilon_r]$-module of rank $n$ as being an $R$-module of rank $n(r+1)$ with additional structure. In this way, we can apply the two examples above to this situation. We write $\nu_r =\nu+ \sum_{i=1}^r x_i \epsilon^i$ with $x_i \in C^1(G_p,\End(\nu))$. Let $\nu_{r-1} = \nu_r/\epsilon^r \nu_r$, and $\nu_{r-2}=\nu_r/\epsilon^{r-1}\nu_r$ (so $\nu_{r-2}=0$ if $r=1$).

First suppose that $\nu_r'$ is a finite-flat representation, and let $x_r'=x_r+x$. Then since $\epsilon^{r}\nu_{r} \cong \epsilon^{r}\nu'_{r} \cong \nu$, and $\nu'_{r}/\epsilon^{r}\nu_{r}  \cong \nu_{r-1}$, we can consider $\nu_{r}$ and $\nu_{r}'$ as being extensions of $\nu_{r-1}$ by $\nu$. In block matrix form, they look like
\[
\sbox0{$\begin{matrix}x_{r} & x_{r-1} & \cdots & x_1 \end{matrix}$}
\sbox1{$\begin{matrix}x'_{r} & x_{r-1} & \cdots & x_1 \end{matrix}$}
\sbox2{$\left(\begin{array}{cccc}
0&0&0&0 \\
0&0&0&0 \\
0&0&0&0 \\
0&0&0&0
\end{array}\right)$}
\nu_{r}=\left(
\begin{array}{c|c}
  \vphantom{\usebox{2}}\makebox[\wd0]{$\nu_{r-1}$}& 0\\
  \hline
\usebox{0}& \nu
\end{array}
\right), \
\nu'_{r}=\left(
\begin{array}{c|c}
  \vphantom{\usebox{2}}\makebox[\wd0]{$\nu_{r-1}$}& 0\\
  \hline
\usebox{1}& \nu
\end{array}
\right).
\]
By Example \ref{eg:baer in terms of matrices}, the Baer difference extension is given by 
\[
\sbox0{$\begin{matrix}{\small x} & 0 & \cdots & 0 \end{matrix}$}
\sbox1{$\begin{matrix}x'_{r+1} & x_{r} & \cdots & x_1 \end{matrix}$}
\sbox2{$\begin{array}{c|ccc}
\nu &0&\cdots&0 \\
x_1&\nu&0&\cdots \\
\cdots&\cdots&\cdots&\cdots \\
x_r&x_{r-1}&\cdots &\nu
\end{array}$}
\left(
\begin{array}{c|c}
\vphantom{\usebox{2}}\makebox[\wd0]{$\nu_r$}& 0\\
  \hline
\usebox{0}& \nu
\end{array}
\right) =
\left(
\begin{array}{c|ccc|c}
\nu & 0 & \cdots & 0 & 0 \\ 
\hline
x_1 & \nu & 0 & \cdots & 0 \\
\cdots & \cdots  & \cdots & \cdots & \cdots \\
x_{r-1} & x_{r-2} & \cdots & \nu & 0 \\
\hline
x & 0 & \cdots & 0 & \nu
\end{array}
\right).
\]
As in Example \ref{eg:pushout}, we can pushout to obtain an extension of $\nu$ by $\nu$ whose cocycle is given by $x$. Since the Baer sum and pushout can be done in any exact category, we could equally well do these operations to the finite-flat groups schemes giving rise to $\nu_{r}$ and $\nu_{r}'$, and obtain an extension of finite flat group schemes whose cocycle is $x$. This implies that $x$ is a finite-flat cocycle.

Conversely, suppose that $x \in Z^1_\fl(G_p,\End(\nu))$. Then $x$ gives rise to an extension $\cE_x$ of $\nu$ by $\nu$. As above, we can consider $\nu_r$ as an extension of $\nu_{r-1}$ by $\nu$. We can also think of $\nu_{r-1}$ as an extension $\cE_{r-1}$ of $\nu$ by $\nu_{r-2}$. By Example \ref{eg:pushout}, the pullback extension $\cE$ of $\cE_{r-1}$ along $\cE_x$ can be written in block matrix form as
\[
\sbox0{$\begin{matrix}x &0  & \cdots &0 \end{matrix}$}
\sbox1{$\begin{matrix}x'_{r+1} & x_{r} & \cdots & x_1 \end{matrix}$}
\sbox2{$\left(\begin{array}{cccc}
0&0&0&0 \\
0&0&0&0 \\
0&0&0&0 \\
0&0&0&0
\end{array}\right)$}
\left(
\begin{array}{c|c}
  \vphantom{\usebox{2}}\makebox[\wd0]{$\nu_{r-1}$}& 0\\
  \hline
\usebox{0}& \nu
\end{array}
\right)
\]
Then, by Example \ref{eg:baer in terms of matrices}, we see that the Baer sum of $\cE$ with $\nu_{r}$ is given by the same matrix as $\nu_{r}'$. As above, we see that the representation obtained from pullback and Baer sum is finite-flat, so this implies that $\nu_{r}'$ is finite-flat.
\end{proof}

\bibliographystyle{alpha}
\bibliography{CWEbib-2019-PG3}

\end{document}